\documentclass[american]{article}
\usepackage[utf8]{inputenc}
\usepackage{color}
\usepackage{mathtools}
\usepackage{amsmath}
\usepackage{amscd}
\usepackage{amstext}
\usepackage{amsthm}
\usepackage{amssymb}
\usepackage{geometry}
\usepackage{booktabs}
\usepackage{tikz}
\usepackage{tikz-cd}
\usepackage{mathrsfs}
\usepackage{enumitem}

\usepackage[unicode=true,
 bookmarks=false,
 breaklinks=false,pdfborder={0 0 1},backref=section,colorlinks=true]
 {hyperref}
\hypersetup{
 linkcolor=blue,citecolor=blue}

\newcommand{\depthstd}{\operatorname{depth}^{\mathrm{std}}}

\makeatletter
\numberwithin{equation}{section}
\numberwithin{figure}{section}

\theoremstyle{plain}
\newtheorem{theorem}{Theorem}[section]
\newtheorem{lemma}[theorem]{Lemma}
\newtheorem{proposition}[theorem]{Proposition}

\newtheorem{conjecture}[theorem]{Conjecture}

\theoremstyle{definition}
\newtheorem{definition}[theorem]{Definition}
\theoremstyle{remark}
\newtheorem{remark}[theorem]{Remark}

\newcounter{thmA}
  
\newtheorem{theoremA}[thmA]{Theorem}

\DeclareMathOperator{\Hom}{Hom}
\DeclareMathOperator{\depth}{depth}
\DeclareMathOperator{\dep}{dep}
\DeclareMathOperator{\Gal}{Gal}

\DeclareMathOperator{\vol}{vol}
\DeclareMathOperator{\Irr}{Irr}

\DeclareMathOperator{\tr}{tr}

\DeclareMathOperator{\Ext}{Ext}

\DeclareMathOperator{\End}{End}
\DeclareMathOperator{\Frob}{Frob}
\DeclareMathOperator{\Res}{Res}
\DeclareMathOperator{\Ind}{Ind}

\DeclareMathOperator{\cind}{c-ind}

\DeclareMathOperator{\Cent}{Cent}
\DeclareMathOperator{\Kaz}{Kaz}

\DeclareMathOperator{\Supp}{Supp}

\DeclareMathOperator{\std}{std}
\DeclareMathOperator{\Stab}{Stab}
\DeclareMathOperator{\pr}{pr}
\DeclareMathOperator{\red}{red}

\DeclareMathOperator{\sh}{sh}

\newcommand{\CC}{\mathbb{C}}

\newcommand{\OO}{\mathcal{O}}

\newcommand{\Gm}{\mathbb{G}_m}

\newcommand{\cH}{\mathcal{H}}            
\newcommand{\Hc}{\mathcal{H}}

\newcommand{\scrO}{\mathscr{O}}

\newcommand{\fracp}{\mathfrak{p}}

\newcommand{\cO}{\mathcal{O}}
\newcommand{\cR}{\mathcal{R}}

\newcommand{\Del}{\mathrm{Del}}

\newcommand{\G}{G}

\newcommand{\id}{\mathrm{id}}
\newcommand{\one}{\mathbf{1}}

\newcommand{\ind}{\operatorname{ind}}

\newcommand{\C}{\mathbb{C}}
\newcommand{\Q}{\mathbb{Q}}
\newcommand{\Z}{\mathbb{Z}}

\setlist[enumerate,1]{label=\textup{(\roman*)}, ref=(\roman*)}

\usepackage[hang,flushmargin]{footmisc}

\date{}
\title{Depth Preservation and Close-Field Transfer in the Local Langlands Correspondence}
\author{Manish Mishra\thanks{
Email: \texttt{manish@iiserpune.ac.in}.\\
Address: Department of Mathematics, IISER Pune, Dr. Homi Bhabha Road, Pashan, Pune 411008, India.\\
The author was partially supported by SERB Core Research Grant (CRG/2022/000415).
}}
\begin{document}

\maketitle

\begin{abstract}
We introduce a revised notion of depth for Langlands parameters for tori defined over a nonarchimedean local field \(F\) that restores depth preservation under the local Langlands correspondence (LLC). We leverage that preservation to derive structural results that, taken together, yield a canonical transfer of broad harmonic-analytic results from characteristic \(0\) to characteristic \(p\). When \(F\) has suitably large positive characteristic, we prove a block-by-block equivalence: each Bernstein block of  \(G(F)\) is equivalent to a corresponding block for some \(G'(F')\) with \(F'\) of characteristic \(0\) \(\ell\)-close to \(F\); using this, we show that a LLC in characteristic \(0\) corresponds canonically to a LLC in characteristic \(p\). For regular supercuspidals we give a direct, more structured construction via Kaletha. Along the way we recover and extend results on \(\ell\)-close fields---introducing a depth-transfer function generalizing the normalized Hasse--Herbrand function, proving truncated isomorphisms for arbitrary tori and parahorics, establishing a depth and supercuspidality preserving Kazhdan-type Hecke-algebra isomorphism for arbitrary maximal  parahorics of arbitrary connected reductive groups; and a generalized Cartan decomposition for arbitrary maximal parahorics---thereby subsuming several earlier results in the literature. Collectively, the results let one work in characteristic \(0\) without loss of generality for a wide swath of harmonic analysis on \(p\)-adic groups.
\end{abstract}

\tableofcontents

\section{Introduction}

The depth of representations and parameters serves as a fundamental invariant in the harmonic analysis of p-adic groups and the local Langlands correspondence (LLC). The central contribution of this work is the introduction of a new notion of depth of a Langlands parameter which rectifies a key deficiency in the standard notion - the preservation of depth. Leveraging this foundation, we establish a canonical transfer of harmonic analysis results from non-archimedean local fields of characteristic zero to those over local function fields, including the local Langlands correspondence. 

For a non-archimedean local field $F$  with maximal ideal \(\mathfrak p_F\subset\mathcal O_F\), local class-field theory interprets the filtration
\[
1 + \mathfrak p_F^{\,r}\;\subset\;F^{\times},\qquad r\ge 0,
\]
as the upper--numbering filtration $\{W_F^{\,s}\}_{s\ge 0}$ of the Weil group. The \emph{Hasse--Herbrand function} \(\phi_{E/F}\) transfers this numbering across an extension $E/F$, matching the depth-$r$ elements of $E^{\times}$ with the \(s:=\phi_{E/F}(er)\)-th filtration subgroup \(W_F^s\), where $e$ denotes the ramification index of the extension. The correspondence is rigorously captured in the local Langlands correspondence for the induced torus \(T=\Res_{E/F}\mathbb{G}_m\): a depth-\(r\) character of \(T(F)\) corresponds to its \emph{standard} depth \(\Phi_T(r):=\phi_{E/F}(er)\) Langlands parameter. Consequently, when $E/F$ is wildly ramified, this standard depth  fails to respect depth preservation limiting its compatibility with harmonic analysis. 
To resolve this, we work with Yu's \emph{minimal congruence filtration} subgroups (see \S \ref{sec:filtrations}) and we introduce the following notion of the depth of a Langlands parameter \(\varphi\) of an arbitrary \(F\)-torus \(T\), 
\[
   \depth_T(\varphi)\;:=\;
   \sup_{f:R\to T}\,\depth_R\!\bigl(\widehat f\circ\varphi\bigr).
\]
Here, the supremum is taken over all induced \(F\)-tori \(R\) and all \(F\)-morphisms \(f:R \to T\) and where depth in the case of an induced torus is defined in such a way as to make its LLC depth preserving. When \(T\) is tamely ramified, this notion of depth reduces to the standard notion. 

To control depth bounds  for connected reductive group $G/F$, we attach to
each vertex $x$ of the (reduced) Bruhat--Tits building a \emph{depth-transfer function} (\S \ref{subsec:defglobaldepcomp}),
\[
   \Phi_{G,x}\colon\mathbb{R}_{\ge0}\longrightarrow\mathbb{R}_{\ge0},
\]
which simultaneously accounts for the ramification of all \emph{maximally unramified} elliptic maximal tori associated to $x$ and of all root subgroups (\S \ref{subsec:defglobaldepcomp}). For an induced torus $T=\Res_{E/F}\mathbb{G}_m$, one has $\Phi_T=\phi^{\mathrm{norm}}_{E/F}$, so $\Phi_{G,x}$ can be considered a generalization of the normalized Hasse--Herbrand function. The new depth and depth-transfer function enable a uniform treatment of depth-$r$ phenomena across the representation theory of $G$ and its behavior under Deligne--Kazhdan close-field transfer. 

Throughout this article, we work with Yu's minimal congruence filtration in place of Moy-Prasad filtrations.  Our first theorem (Theorem \ref{thm:dep-r-LLC-forTori}) establishes depth preservation for tori without any tameness hypothesis.

\paragraph{Depth-\textit{r} local Langlands for tori}
\begin{theoremA}[Depth-$r$ LLC for tori]\label{thm:A}
Let $T/F$ be an arbitrary torus and let $r\ge0$. The  local Langlands correspondence for tori
$\mathcal L_T\colon\Hom\!\bigl(T(F),\mathbb{C}^\times\bigr)\xrightarrow{\sim} H^1(W_F,\widehat T)$
preserves depth, i.e., for every character $\chi$ of $T(F)$, we have,
        \[
          \operatorname{dep}\bigl(\mathcal L_T(\chi)\bigr)
          \;=\;
         \!\operatorname{dep}(\chi),
        \]
where the depth of a Langlands parameter is as defined in Definition \ref{def:depth-parameter}.
\end{theoremA}
Theorem~\ref{thm:A} extends Yu's depth-preservation theorem (\(T\) tame, \(r>0\)) and the induced-torus result of \cite{MishraPattanayak2015} (\(T=\Res_{E/F}\mathbb{G}_m\)).
\smallskip
Using Deligne's isomorphism on the Galois side of this correspondence, we obtain (Theorem \ref{thm:congruent-isom}),
\paragraph{Truncated isomorphisms over $\ell$-close fields}
\begin{theoremA}[Truncated-torus close field isomorphisms]\label{thm:B}
Let $F$ and $F'$ be $\ell$-close and let $T/F$ be a torus with
transfer $T'/F'$.  If $r\geq 0$ such that $\Phi_T(r)\leq \ell$, there is a canonical
isomorphism,
\[
   T(F)\big/\!T(F)_r
   \;\xrightarrow{\;\sim\;}
   T'(F')\big/\!T'(F')_{\,r}.
\]

\end{theoremA}
When $r\in\mathbb{Z}_{\ge0}$, this recovers the main result of Aubert--Varma \cite{AV24}.
\smallskip

Before stating our next result, we first recall a result of  Chai--Yu \cite{ChaiYu2001}.
Let \(F\) and \(F'\) be \(\ell\)-close local fields, let \(T/F\) be a torus, and let \(T'/F'\) be its Deligne transfer.
Write \(\mathscr{T}\) and \(\mathscr{T}'\) for the respective N\'{e}ron models.  Then for any given integer \(m\), there is a suitably large \(\ell\) such that there is a canonical isomorphism of smooth group schemes
\[
   \mathscr{T}\,\times_{\cO_F}
        \cO_F/\mathfrak {p}_F^{\,m}
   \;\xrightarrow{\;\sim\;}
   \mathscr{T}'\times_{\cO_{F'}}
        \cO_{F'}/\mathfrak {p}_{F'}^{\,m}.
\]
Consequently, one has a canonical isomorphism,
\(
   \mathscr T(\cO_F/\mathfrak {p}_F^{\,m})
   \cong
   \mathscr T'(\cO_{F'}/\mathfrak {p}_{F'}^{\,m}).
\) Building on this, by showing that Bruhat--Tits root-group filtration schemes behave just as well, Ganapathy \cite[Cor.\;6.3]{ganapathy2019} extends this result to parahoric group schemes.
\smallskip

Our next Theorem dispenses with the Chai--Yu input which is the main ingredient in \cite{AV24} and \cite{ganapathy2019}.  Working with Yu's minimal-congruence filtration, we prove (Theorem \ref{thm:depth-comp-HH}),

\begin{theoremA}[Truncated-parahoric close field isomorphisms]\label{thm:C}
Let $G/F$ be connected reductive, $x\in\mathcal{B}^{\mathrm{red}}(G,F)$, and let
$(G',x')$ be the Deligne transfer over an $\ell$-close field $F'$.
For every $0\le r\le\Phi_{G,x}(\ell)$, there is a canonical isomorphism,
\[
   G(F)_{x,0}\big/\!G(F)_{x,r}
   \;\xrightarrow{\;\sim\;}
   G'(F')_{x',0}\big/\!G'(F')_{x',r}.
\]
\end{theoremA}
In particular, by allowing real \(r\) instead of integers, our result extends a key aspect of the main results of Chai-Yu \cite{ChaiYu2001} and Ganapathy \cite{ganapathy2019} and which is also sufficient for most harmonic analysis applications.

\smallskip
\paragraph{Close field Hecke algebra isomorphisms.}
After extending the Cartan decomposition to arbitrary parahorics (see Theorem~\ref{thm:D} below),
and following the method of \cite{ganapathy2019}, we obtain a parahoric-level close-fields
Hecke algebra isomorphism (proved in Appendix~\ref{app:kazhdan}, Theorem~\ref{thm:deligne-kazhdan-Hecke-isomorphism}):
for each $m\in\Q_{\ge 0}$ there exists $\ell(m)\gg 0$ (depending only on $(G,x,m)$) such that,
whenever $F$ and $F'$ are $\ell$-close with $\ell\ge \ell(m)$, the assignment coming from
Theorem~\ref{thm:C} yields an algebra isomorphism
\[
\kappa_{\ell,x}:\ \cH\!\bigl(G,G(F)_{x,m}\bigr)\ \xrightarrow{\ \sim\ }\ \cH\!\bigl(G',G'(F')_{x',m}\bigr).
\]
Moreover, under this isomorphism supercuspidal representations correspond to supercuspidal
representations, and depth at the vertex is preserved:
\[
\dep_x\!\bigl(\kappa_{\ell,x}(\pi)\bigr)=\dep_x(\pi)\qquad\text{for all smooth $\pi$ of $G(F)$.}
\]
Here $\dep_x(\pi)$ denotes the depth evaluated at $x$ (and may differ from the usual depth $\dep(\pi)$).

This result specializes, for hyperspecial parahorics of split groups, to Kazhdan's foundational
theorem \cite{kazhdan1986}. Its Iwahori-level analogue (for $\mathrm{GL}_n$, see \cite{Lem01}) has been used
to transport the local Langlands correspondence from characteristic~$0$ to positive characteristic,
notably in \cite{Gan2015} for $\mathrm{GSp}_4$ and \cite{GanVar17} for split classical groups. The special
parahoric case for $m\in\Z_{>1}$ was established by Ganapathy \cite{ganapathy2022}.

However, this parahoric-level equivalence is vertex-dependent. Moreover, it transfers overlapping
subcategories of $\cR(G(F))$, which typically meet several Bernstein blocks at once. This makes it
too coarse for a uniform, blockwise treatment of $\cR(G(F))$.

Our next theorem remedies this by producing a support-preserving, block-by-block canonical
isomorphism of Hecke algebras, at the cost of excluding a finite set of small residue characteristics.

In recent joint works of the author with Jeffrey D. Adler, Jessica Fintzen and Kazuma Ohara \cite{AFMO24a}\cite{AFMO24b}, it was established that the Hecke algebras attached to types constructed by Kim and Yu are isomorphic to the Hecke algebras attached to depth-zero types of some twisted Levi subgroup. We use that result in combination with the above-developed results to prove the following (Theorem \ref{thm:ell-close-arbitrary-blocks}).

\begin{theoremA}[Hecke-algebra isomorphism for Bernstein blocks over $\ell$--close fields]\label{thm:F}
Let $F$ be a non-archimedean local field and $G/F$ a connected reductive group. Assume that the residue characteristic  $p$ is odd and not a torsion prime for $G$. Fix a Bernstein block $\,[L,\pi]_G$ of $G(F)$. Then there exists $\ell_0=\ell_0(G,[L,\pi]_G)$ with the following property.

For every local field $F'$ that is $\ell$--close to $F$ with $\ell\ge\ell_0$, there are canonically defined transfers
\[
(G,F)\rightsquigarrow (G',F'),\qquad
L\rightsquigarrow L'\subset G',\qquad
\pi\rightsquigarrow \pi'\;\text{ on }L'(F'),
\]
and a canonical isomorphism of Hecke algebras attached to the two blocks:
\[
\kappa_{\ell,[L,\pi]_G}\;:\;
\mathcal{H}\!\bigl(G(F),[L,\pi]_G\bigr)\ \xrightarrow{\;\sim\;}
\mathcal{H}\!\bigl(G'(F'),[L',\pi']_{G'}\bigr).
\]
Here $\mathcal{H}(G(F),[L,\pi]_G)$ denotes Bushnell--Kutzko Hecke algebra attached to \([L,\pi]_G\). Moreover, this isomorphism preserves the presentation of these Hecke algebras as obtained in \cite[5.3.6]{AFMO24a}.
\end{theoremA}
\paragraph{Local Langlands correspondence in positive characteristic.}
For regular supercuspidal representations, the input data are toral (a tame regular elliptic pair
$(S,\theta)$), and this allows a direct close--field transport of Kaletha's construction from
characteristic~$0$ to positive characteristic using only the truncated transfer for tori/parahorics
developed earlier (Theorems~\ref{thm:B} and~\ref{thm:C}).
Concretely, we obtain the following (Theorem~\ref{thm:llc-for-regular}).

\begin{theoremA}\label{thm:G}
Let  \(G\) be a tamely ramified connected reductive group over a non-archimedean local field \(F\) of characteristic \(p\) . Assume that  \(p\) is odd, not bad for \(G\), and coprime to the order of \(\pi_0(Z(G))\). Then:  
\begin{enumerate}
\item The set of regular supercuspidal parameters (up to conjugacy) is in bijection with the set of isomorphism classes of regular supercuspidal data.  
\item For each such parameter $\varphi$ there is a finite $L$--packet $\Pi_\varphi(G)$ of regular supercuspidal representations of $G(F)$. Its internal labeling agrees with the characteristic--$0$ theory: after fixing a Whittaker datum \(\mathbf{w}\) for the quasi--split inner form of $G$, there is a canonical finite map from a natural set of \emph{enhancements} of $\varphi$ onto
$\Pi_\varphi(G)$, and the only ambiguity is the action of a finite abelian group determined canonically by $G$ and $\varphi$.
\item If $G$ is quasi--split, then for any chosen Whittaker datum $\mathbf{w}$, the packet
$\Pi_\varphi(G)$ contains a $\mathbf{w}$--generic member.
\end{enumerate}  
\end{theoremA}
Theorem \ref{thm:G} is proved using only the simplicity of the input data on both sides of the LLC---entirely without Theorem \ref{thm:F}, which treats the general framework. Our next result, Theorem \ref{thm:H}, leverages Theorem \ref{thm:F} to pass to full generality.
\begin{theoremA}\label{thm:H}
Let \(F\) be a non-archimedean local field of characteristic \(p > 0\), and \(G\) a connected reductive \(F\)-group. Assume that a depth preserving Local Langlands Correspondence is established for all such pairs \((G'/F')\) arising from transfers of \(G\) to \(\ell\)-close characteristic-\(0\) fields for suitably large \(\ell\) . Then there exists a map, 
\[
\sigma \mapsto \varphi(\sigma) \quad \forall \sigma \in \operatorname{Irr}(G(F)),
\]defined by :  \(\sigma':=\kappa_{\ell,[L,\pi]_G}^{\sharp}(\sigma)\) via the blockwise Hecke--algebra isomorphism (Theorem \ref{thm:F}) and pull back of parameter  \(\varphi(\sigma')\) along Deligne's isomorphism. This assignment is independent of choices. 
\end{theoremA}
In the course of the proofs, we describe the double coset space of an arbitrary maximal parahoric \(G(F)_{x,0}\) in terms of a \emph{maximally unramified} elliptic maximal torus \(S\) describing the vertex \(x\) (see \S \ref{subsec:maxunramified} for the notation of maximally unramified and its correspondence to a vertex). This theorem generalizes the previously known result due to Haines and Rostami \cite{HR10} for the case where \(x\) is special. 
\paragraph{Cartan decomposition and Hecke algebras}
\begin{theoremA}[Generalised Cartan decomposition]\label{thm:D}
For any vertex $x\in\mathcal{B}^{\mathrm{red}}(G,F)$ let $K=G(F)_{x,0}$ be the associated parahoric. Then there exists a canonical bijection
\[
  W(G,S)_{x,0}\backslash X_{*}\!\bigl(S\bigr)_{I_F}^{\Frob} \xrightarrow{\;\sim\;} K\backslash G(F)/K,
\]
where $W(G,S)_{x,0}:=N_{G(F)_{x,0}}(S)(F)/S(F)_0$ and \(S\) a maximally unramified elliptic maximal \(F\)-torus of \(G\) whose associated vertex is \(x\). Here \(I_{F}=\Gal(\overline{F}/F^{\mathrm{ur}}), \Frob\in\Gal(\overline{\kappa}/\kappa), \) where \(\kappa\) denotes the residue field. 

\end{theoremA}
Theorem \ref{thm:D} is independent of the choice of \(S\).
\smallskip
\renewcommand{\thetheoremA}{A}

\subsection*{Notations}
Throughout, $F$ is a non-archimedean local field with residue characteristic $p$. For a real number $r$, we write $\lceil r\rceil$ for the least integer $\ge r$. The upper-numbered ramification subgroups of $W_F$ are $W_F^{r}$ for $r\ge0$, and inertia is $I_F=W_F^0$.
In literature, \(G(F)_1\) is sometimes used to denote the kernel of the Kottwitz homomorphism. In the case of a torus \(T\), this can create ambiguity. By \(T(F)_1\), we always mean the first congruent filtration.

A prime $p$ is \emph{bad} for a connected reductive group $G$ if it is bad for some simple factor
of its absolute root system $\Phi(G)$. For an irreducible reduced root system $\Phi$, the bad
primes are:
\[
\begin{array}{c|l}
\Phi & \text{bad primes} \\ \hline
A_r\ (r\ge 1) & \text{none} \\
B_r,\ C_r,\ D_r\ (r\ge 2) & 2 \\
G_2,\ F_4,\ E_6,\ E_7 & 2,\ 3 \\
E_8 & 2,\ 3,\ 5
\end{array}
\]
Equivalently, $p$ is bad iff it divides some coefficient of the highest root written as a
$\mathbb{Z}_{\ge0}$--linear combination of simple roots of $\Phi$.

Let $\pi_0(Z(G))$ denote the component group of the center of \(G\). The possible prime divisors of \(|\pi_0(Z(G))|\) are:

\[
\begin{array}{c|l}
\text{simple type} & \text{possible prime divisors of }|\pi_0(Z(G))| \\ \hline
A_r\ (r\ge1) & \text{primes dividing } r+1 \\
B_r,\ C_r,\ D_r & 2 \\
E_6 & 3 \\
E_7 & 2 \\
E_8,\ F_4,\ G_2 & \text{none}
\end{array}
\]

For a general $G$ with derived subgroup a product of simple factors, the set of prime divisors of
$|\pi_0(Z(G))|$ is a (possibly proper) subset of the union of the entries for those factors

\section{Filtrations on a torus}
\label{sec:filtrations}

Let $F$ be a non-archimedean local field with ring of integers $\scrO_{F}$, maximal ideal $\fracp_{F}$,  normalized valuation $\operatorname{val}_{F}$ and residue field \(\kappa\) . Let $T$ be an $F$-torus. We recall three filtrations that appear in the literature.

\subsection{Na\"{i}ve valuation filtration}

\begin{definition} 
For $r\ge0$, set 
\[
T(F)_{r}^{\text{na\"{i}ve}}:=\left\{ t\in T(F)_b\;\middle|\;\operatorname{val}_{F}\bigl(\chi(t)-1\bigr)\ge r\text{ for all }\chi\in X^{\!*}(T)\right\} ,
\]
where $T(F)_b$ is the unique maximal bounded subgroup of $T(F)$.
\end{definition}

\subsection{Moy-Prasad filtration}

\begin{definition} 
For $r\ge0$, define 
\[
T(F)_{x,r}^{\mathrm{MP}}:=\bigl\{ t\in T(F)_{0}\;\big|\;\operatorname{val}_{F}\bigl(\chi(t)-1\bigr)\ge r\text{ for all }\chi\bigr\},
\]
where $T(F)_{0}$ is the Iwahori subgroup of $T(F)$ (see \cite[B.5.1]{KalethaPrasad2016}). 
\end{definition}

\subsection{Minimal congruence filtration on tori}
\label{sub:min-congruence}
For a discrete valuation ring \(\mathcal{O}\) with uniformizer \(\varpi\) and for  any smooth $\cO$-group scheme $X$, the group
\[
   \Gamma(\varpi^n,X) :=\ker\!\bigl(X(\cO)\longrightarrow X(\cO/\varpi^n\cO)\bigr),\qquad n \geq 1
\]
is called the \(n\)-th congruence subgroup of \(X\).  Each $\Gamma(\varpi^n,X)$ is an open normal subgroup of $X(\cO)$.
We will now recall the notion of minimal congruence filtrations on tori, which is due to Jiu-Kang Yu. 
The contents of this section are taken from \cite{yu2002smooth} after incorporating the correction from \cite{kaletha-prasad2024errata}.

\subsubsection{Definition: Minimal congruence filtration}\label{subsub:mincongdef}
Let $\widetilde{F}$ denote the completion of the maximal unramified extension of $F$ and let $\widetilde{\mathcal{O}}$  be its ring of integers.
Let \(\mathcal{O}^{\sh}\) be the strict henselization of \(\mathcal{O}=\mathcal O_F\), with fraction field \(F^{\mathrm{sh}}\) and residue field \(\kappa^{\mathrm{sh}}\) (a separable closure of \(\kappa\)). Let $T/F$ be a torus. Define a family of smooth models $\{\mathscr{T}^{\mathrm{mc}}_r\}_{r\geq 0}$ as follows.

\begin{enumerate}[leftmargin=1.8em, label=\textup{(\arabic*)}]
    \item Let $\mathscr{T}^{\mathrm{mc}}_0$ be the connected N\'{e}ron model of $T$.
    
    \item \textbf{$0<r<1$.}
    Let $T(F^{\mathrm{sh}})^{\mathrm{mc}}_r$ be the subgroup generated by 
    \begin{itemize}
        \item[(a)] $\Gamma(\varpi, \mathscr{T}^{\mathrm{mc}}_0 \otimes \mathcal{O}^{\sh})$, and
        \item[(b)] the images of $R(F^{\mathrm{sh}})^{\mathrm{MP}}_r \to T(F^{\mathrm{sh}})$ for all morphisms $f: R \to T$ defined over $F^{\sh}$, 
              where $R$ is an induced $F^{\mathrm{sh}}$-torus (i.e., $R \cong \prod_i \operatorname{Res}_{E_i/F^{\sh}} \mathbb{G}_{m}$ 
              for finite separable extensions $E_i/F^{\mathrm{sh}}$).
    \end{itemize}
    There exists a smooth model $\mathscr{T}^{\mathrm{mc}}_r/\mathcal{O}_F$ such that $\mathscr{T}^{\mathrm{mc}}_r(\mathcal{O}^{\sh}) = T(F^{\sh})^{\mathrm{mc}}_r$ 
    (obtained by taking the dilatation on $\mathscr{T}^{\mathrm{mc}}_0$ of the closed subgroup $T(F^{\sh})^{\mathrm{mc}}_r/\Gamma(\varpi, \mathscr{T}^{\mathrm{mc}}_0 \otimes \mathcal{O}^{\sh})$ 
    in $\mathscr{T}^{\mathrm{mc}}_0(\kappa^{\sh})$).
    
    \item \textbf{$r\ge1$.}
    Write $r=r_0+n$ with $n\in\mathbb{Z}_{>0}$ and $r_0\in[0,1)$. Then put $\mathscr{T}^{\mathrm{mc}}_r = (\mathscr{T}^{\mathrm{mc}}_{r_0})^{(n)}$, 
    the $n$-th congruence subgroup model.
    
    \item Finally, put $T(F)^{\mathrm{mc}}_r = \mathscr{T}^{\mathrm{mc}}_r(\mathcal{O}_F)$.
\end{enumerate}
\[
\boxed{\text{For the rest of the article, we abbreviate } T(F)_r := T(F)_{r}^{\mathrm{mc}}.}
\]

\subsubsection*{Basic properties}

\begin{itemize}[leftmargin=1.8em]
    \item Functoriality.\label{lem:functorial}
    Any morphism of $F$-tori $f:T\to T'$ satisfies $f\bigl(T(F)_r\bigr)\subset T'(F)_r$ for all $r\ge0$.
    
    \item Left-continuity.
    $T(F)_r=\bigcap_{s<r}T(F)_s$.
    
    \item Compatibility with Moy-Prasad.
    If $T$ is induced or splits over a tamely ramified extension, then the minimal congruence filtration coincides with the Moy-Prasad filtration.
    
    \item $\Gamma$-stability.\label{prop:Gamma-stable}  
    The filtration subgroup $T(F^{\mathrm{sh}})^{\mathrm{mc}}_r$ is stable under $\Gamma = \operatorname{Gal}(F^{\mathrm{sh}}/F)$.
\end{itemize}

 For an \(\mathcal{O}\)-scheme \(\mathscr{X}\), we denote its base change to \(\mathcal{O}^{\mathrm{sh}}\) by $\mathscr{X}^{\mathrm{sh}}$.

\begin{proposition}\label{P:depth-shift}
Let \(F\) be a non-archimedean local field with ring of integers \(\mathcal{O}_{F}\) and let \(T\) be an \(F\)-torus. Given \(r>0\), write it uniquely as,
\[r\ =\ n+r_{0},\qquad n\in\mathbb{Z}_{\geq 0},\ 0\leq r_{0}<1.\]
Let \(\mathscr{T}_{r}\) be the minimal congruent model as in Definition \ref{subsub:mincongdef}. Then,
\[
T(F^{\textnormal{sh}})_{r}\ =\ \Big{\langle}\,T(F^{\textnormal{sh}})_{n+1},\ f\big{(}R(F^{\textnormal{sh}})_{r}\big{)}\ \Big{|}\ \ f:R\to T \text{ defined over } F^{\mathrm{sh}}, \text{ with } R \text{ an induced } F^{\mathrm{sh}}\text{-torus}\Big{\rangle}.
\]
\end{proposition}

\begin{proof}
By definition of the minimal congruent filtration (applied over \(F^{\mathrm{sh}}\)), we have:
\[
T(F^{\mathrm{sh}})_r = \Gamma(\varpi^n, \mathscr{T}_{r_0}^{\mathrm{sh}}),
\]
where \(\mathscr{T}_{r_0}\) is the smooth model over \(\mathcal{O}\) associated to depth \(r_0\).
By Definition \ref{subsub:mincongdef}, \(T(F^{\mathrm{sh}})_{r_0}\) is generated by:
 \(\Gamma(\pi, \mathscr{T}_0^{\mathrm{sh}})) = T(F^{\mathrm{sh}})_1\) and the collection
\(f(R(F^{\mathrm{sh}})_{r_0})\) consisting of all morphisms \(f \colon R \to T\) with \(R\) induced over \(F^{\sh}\).
So
\[
T(F^{\mathrm{sh}})_{r_0} = \left\langle T(F^{\mathrm{sh}})_1,\; f(R(F^{\mathrm{sh}})_{r_0}) \mid \text{induced } R/F^{\sh},\ f \right\rangle.
\]

By \cite[Proposition 2.8]{yu2002smooth}, \(\Gamma(\pi^n, \mathscr{T}_0^{\mathrm{sh}}) = T(F^{\mathrm{sh}})_{n}\). Also, by the extension principle in \cite[\S 2.3]{yu2002smooth}, each $F^{\sh}$ morphism \(f \colon R \to T\) with \(R\) induced $F^{\sh}$-torus, the map \(f\) extends to a morphism of smooth models \(\mathscr{R}_{r_0} \to \mathscr{T}_{r_0}\) over $F^{\sh}$. Apply \(\Gamma(\varpi^n, -)\) to \(f_{\mathscr{T}}: \mathscr{R}_{r_0} \to \mathscr{T}_{r_0}\) to obtain,
  \[
  f_{\mathscr{T}}: \Gamma(\varpi^n, \mathscr{R}_{r_0}) \longrightarrow \Gamma(\varpi^n, \mathscr{T}_{r_0}).
  \]
By functoriality of congruence subgroups, the morphism \(f_{\mathscr{T}}: \mathscr{R}_{r_0} \to \mathscr{T}_{r_0}\) gives,  
\[
f_{\mathscr{T}}\left(\Gamma(\varpi^n, \mathscr{R}_{r_0}^{\text{sh}})\right) = \Gamma\left(\varpi^n, f_{\mathscr{T}}(\mathscr{R}_{r_0}^{\text{sh}})\right) = \Gamma\left(\varpi^n, f(\mathscr{R}_{r_0}^{\text{sh}})\right).
\]
For \(R\), the minimal congruence filtration gives \(\Gamma(\varpi^n, \mathscr{R}_{r_0}^{\text{sh}}) = R(F^{\text{sh}})_r\), so,
\[
f_{\mathscr{T}}(\Gamma(\varpi^n, \mathscr{R}_{r_0}^{\text{sh}})) = f(R(F^{\text{sh}})_r).
\]
The model \(\mathscr{T}_1\) is the first congruence subgroup model of \(\mathscr{T}_0\), so,
\[
\Gamma(\varpi^k, \mathscr{T}_1^{\text{sh}}) = T(F^{\text{sh}})_{k+1}.
\]
Since \(\mathscr{T}_{r_0}\) is generated by the subgroup schemes \(\mathscr{T}_1\) and \(f_{\mathscr{T}}(\mathscr{R}_{r_0})\) for all \(f\) and induced \(R\), the congruence subgroup \(\Gamma(\varpi^n, \mathscr{T}_{r_0}^{\text{sh}})\) is generated as a group by:
 \(\Gamma(\varpi^n, \mathscr{T}_1^{\text{sh}}) = T(F^{\text{sh}})_{n+1}\),
and \(f_{\mathscr{T}}(\Gamma(\varpi^n, \mathscr{R}_{r_0}^{\text{sh}})) = f(R(F^{\text{sh}})_r)\) for all \(f\) and induced \(R\).
Combining the above, we obtain,
\[
T(F^{\text{sh}})_r = \Gamma(\varpi^n, \mathscr{T}_{r_0}^{\text{sh}}) = \left\langle T(F^{\text{sh}})_{n+1}, f(R(F^{\text{sh}})_r) \mid f: R \to T \text{ over } F^{\text{sh}}, R \text{ induced} \right\rangle.
\]
\end{proof}
\section{The Hasse-Herbrand function}
\label{sec:hasse-herbrand}

Let $E/F$ be a finite Galois extension with ramification index $e$. Denote by $G^{u}$ ($u\ge-1$) the upper-numbered ramification groups of $\Gal(E/F)$. 

\begin{definition}[Hasse--Herbrand function]\label{def:hasse-herbrand}
Let $E/F$ be a finite Galois extension with Galois group $G=\Gal(E/F)$.
For $t\ge 0$, let $G_t$ denote the \emph{lower-numbered} ramification subgroup of $G$
(in the sense of Serre).
Define the \emph{Hasse--Herbrand function} $\varphi_{E/F}:[0,\infty)\to[0,\infty)$ by
\[
\varphi_{E/F}(t)\ :=\ \int_{0}^{t}\frac{du}{[\,G_0:G_u\,]}\,.
\]
The function $\varphi_{E/F}$ is continuous, increasing, and piecewise linear, hence admits
a (two-sided) inverse $\psi_{E/F}=\varphi_{E/F}^{-1}$.
The \emph{upper-numbered} ramification subgroups are then defined by
\[
G^{r}\ :=\ G_{\psi_{E/F}(r)}\qquad (r\ge 0).
\]
\end{definition}

We record its basic properties.

\begin{lemma}\label{lem:HH} 
$\phi_{E/F}$ is continuous, piecewise linear, increasing, concave and satisfies $\phi_{E/F}(0)=0$. For \(r\geq 0\), \(\phi_{E/F}(er)/e \leq r\)  and if  \(E/F\)  is tame, then $\phi_{E/F}(er)/e=r$. 
\end{lemma}

\begin{lemma}
\label{lem:HH-tower}
For every tower of finite Galois extensions $F\subset E_1\subset E_2$ and every real number $u\ge0$,
\[
   \phi_{E_2/F}(u) = \phi_{E_1/F}\!\bigl(\,\phi_{E_2/E_1}(u)\bigr).
\tag{$\ast$}
\]
\end{lemma}

\begin{proof}
This is \cite[IV.\S 3, Prop.~15]{SerreLocal}.
\end{proof}

\section{Depth comparison under LLC for tori}
\label{sec:depth-theorem}
\subsection{Standard depth of a parameter and its generalization}
\begin{definition} 
Let $\chi\in\Hom(T(F),\CC^{\times})$,  \(\chi\)  \emph{ramified} (i.e.,  \(\chi|_{T(F)_0}\neq1\) ). Define,
\[
\depth_{T}(\chi):=\inf\bigl\{ r\ge0\mid\chi|_{T(F)_{r+}}=1 \text{ and } \chi|_{T(F)_{r}}\neq1\bigr\}.
\]
For $\varphi \in H^1(W_F,\widehat{T})$, let
\[\depthstd(\varphi)=\inf\bigl\{r \mid \varphi|_{W_F^r }\neq 1 \text{ and } \varphi|_{W_F^{r+}}=1 \bigr\}.\]
\end{definition}

We now define a new notion of depth of a Langlands parameter which is better aligned for local Langlands correspondence. 
\begin{definition}[Depth of a Langlands Parameter for a Torus]\label{def:depth-parameter}
Let \( T \) be an arbitrary \( F \)-torus, and let \( \varphi: W_F \to \widehat{T} \) be a Langlands parameter. The \emph{depth} of \( \varphi \) is defined as
\[
\operatorname{depth}(\varphi) = \sup_{\substack{f: R \to T \\ R\text{ induced }F\text{-torus}}} \operatorname{depth}_R(\widehat{f} \circ \varphi),
\]
where:
\begin{itemize}
  \item The supremum is over all induced \( F \)-tori \( R \) and all morphisms \( f: R \to T \) of \( F \)-tori.
  \item \( \widehat{f}: \widehat{T} \to \widehat{R} \) is the dual morphism of \( f \).
  \item \( \operatorname{depth}_R(\widehat{f} \circ \varphi) \) is the depth of the parameter \( \widehat{f} \circ \varphi: W_F \to \widehat{R} \), defined as follows:
\end{itemize}

Suppose \( R = \prod_{j=1}^{k} \operatorname{Res}_{E_j/F} \mathbb{G}_m \) is an induced torus. Then for a parameter \( \psi: W_F \to \widehat{R} \), we define:
\[
\operatorname{dep}_R(\psi) = \max_{1 \leq j \leq k} \left( \frac{1}{e_j} \cdot \phi_{E_j/F}^{-1}\left( \operatorname{dep}^{\std}(\psi_j) \right) \right),
\]
where:
\begin{itemize}
  \item \( \psi_j: W_F \to \widehat{R}_j \) is the \( j \)-th component of \( \psi \), corresponding to \( R_j = \operatorname{Res}_{E_j/F} \mathbb{G}_m \).
  \item \( \operatorname{dep}^{\std}(\psi_j) = \inf \{ t \geq 0 \mid \psi_j(W_F^{t+}) = \{1\} \} \) is the standard depth of the parameter \( \psi_j \).
  \item \( e_j = e(E_j/F) \) is the ramification index of the extension \( E_j/F \).
  \item \( \phi_{E_j/F} \) is the Herbrand function of the extension \( E_j/F \).
\end{itemize}
\end{definition}
\begin{lemma}
    If \(T\) is tamely ramified, then \(\depthstd(\varphi)=\depth(\varphi)\).
\end{lemma}
\begin{proof}
Assume \(T = \prod_{i=1}^{m} \operatorname{Res}_{E_i/F} \mathbb{G}_m\), where each \(E_i/F\) is a finite tamely ramified extension with ramification index \(e_i\). For a Langlands parameter \(\varphi: W_F \to \widehat{T}\), denote by \(\varphi_i: W_F \to \widehat{T}_i\) the \(i\)-th component.
Since \(T\) is  tamely ramified,  by Lemma \ref{lem:HH},  \(\phi_{E_i/F}(t) = t / e_i\) for all \(t \geq 0\). 
Therefore  \(\phi_{E_i/F}^{-1}(r_i) = e_i r_i\), and so:
\[
\operatorname{depth}_T(\varphi) = \max_i \left( \frac{1}{e_i} \cdot e_i r_i \right) = \max_i r_i = \max_i s_i = \operatorname{depth}^{\text{std}}(\varphi).
\]
Thus, \(\operatorname{depth}(\varphi) = \operatorname{depth}^{\text{std}}(\varphi)\) for induced tamely ramified tori.
Now let $T$ be an arbitrary tamely ramified $F$--torus and let $\varphi:W_F\to \widehat T$
be a Langlands parameter. Set $r=\depthstd(\varphi)$.
Choose a finite tamely ramified extension $E/F$ which splits $T$.

By standard structure theory of $F$--tori split by $E$ (equivalently, of $\Z[\Gal(E/F)]$--lattices),
there exists a \emph{quasi-trivial} (hence induced) $F$--torus
\[
R \;=\; \prod_{i=1}^m \Res_{E_i/F}\G_m
\]
(with each $E_i/F$ a finite subextension of $E/F$, in particular tamely ramified) together with a
surjective $F$--morphism of tori
\[
f:R\longrightarrow T.
\]
Since $f$ is surjective, the dual morphism $\widehat f:\widehat T\hookrightarrow \widehat R$
is injective. Hence for every $s\ge 0$,
\[
(\widehat f\circ \varphi)\vert_{W_F^{s}}=1 \quad\Longleftrightarrow\quad
\varphi\vert_{W_F^{s}}=1,
\]
and therefore
\[
\depthstd(\widehat f\circ \varphi)=\depthstd(\varphi)=r.
\]
Since $R$ is induced and tamely ramified, we have
$\dep_R(\widehat f\circ \varphi)=\depthstd(\widehat f\circ \varphi)=r$
(by the induced/tame case proved above). In particular, by Definition~\ref{def:depth-parameter},
\[
\depth(\varphi)
=\sup_{\substack{f':R'\to T\\ R'\ \mathrm{induced}}}\dep_{R'}(\widehat{f'}\circ\varphi)
\;\ge\; \dep_R(\widehat f\circ\varphi)\;=\; r.
\]

On the other hand, for any $F$--morphism $f':R'\to T$ with $R'$ induced, the composite
$\widehat{f'}\circ\varphi$ is also trivial on $W_F^{r+}$, hence
$\depthstd(\widehat{f'}\circ\varphi)\le r$. By Lemma~\ref{lem:HH} we have
$\dep_{R'}(\widehat{f'}\circ\varphi)\le \depthstd(\widehat{f'}\circ\varphi)\le r$.
Taking the supremum over all such $f'$ shows $\depth(\varphi)\le r$.

Combining both inequalities, we obtain $\depth(\varphi)=r=\depthstd(\varphi)$.

\end{proof}

\begin{definition}\label{def:PhiT}
For $r>0$ define
\[
\Phi_T(r)\;:=\;\sup_{\substack{f:R\to T\\ R\ \text{induced}}}\ \Phi_R(r),
\]
where for an induced torus $R=\prod_i\Res_{E_i/F}\G_m$ we set
$\Phi_R(r)=\max_i \varphi_{E_i/F}(e_ir)$.
\end{definition}

\begin{lemma}\label{lem:PhiT-monotone}
With $\Phi_T$ defined by \eqref{def:PhiT}, the function $\Phi_T$ is non-decreasing.
\end{lemma}

\begin{proof}
Each $\Phi_R$ is non-decreasing, and $\Phi_T$ is a supremum of such functions.
\end{proof}

\begin{proposition}\label{prop:Phi-bounds-std}
Let $T/F$ be an arbitrary torus, let
\(
\varphi:W_F\longrightarrow\widehat T
\)
be a Langlands parameter.    
Then
\[
   \Phi_T\!\bigl(\dep(\varphi)\bigr)\;\;\ge\;\;\dep^{\std}(\varphi).
\]
\end{proposition}
\begin{proof}
Put
\[
s \;:=\; \dep^{\std}(\varphi)
\qquad\text{and}\qquad
r \;:=\; \dep(\varphi).
\]
By definition of $s$, there exists $w\in W_F^{\,s}$ with $\varphi(w)\neq 1$ and
$\varphi(W_F^{s+})=1$.

Choose an algebraic character $\lambda:\widehat T\to\CC^\times$ such that
$\lambda(\varphi(w))\ne 1$.
Let $E/F$ be the finite extension fixed by the stabilizer of $\lambda$ in $W_F$; then
$\lambda$ is $W_E$--invariant.  Set $R_\lambda:=\Res_{E/F}\Gm$.
The dual torus is $\widehat{R_\lambda}\simeq \Ind_{W_E}^{W_F}\CC^\times$.
By Frobenius reciprocity there is a canonical bijection
\[
\Hom_{W_F}\!\bigl(\widehat T,\ \Ind_{W_E}^{W_F}\CC^\times\bigr)
\ \cong\
\Hom_{W_E}\!\bigl(\widehat T,\ \CC^\times\bigr),
\]
so $\lambda$ corresponds to a unique $W_F$--equivariant morphism
$\widehat f_\lambda:\widehat T\to \widehat{R_\lambda}$ such that
$\pr\circ \widehat f_\lambda=\lambda$, where $\pr:\widehat{R_\lambda}\to\CC^\times$
is dual to the tautological $F$--morphism of tori
$\iota:\Gm\to \Res_{E/F}\Gm$.
Define
\[
\varphi_\lambda \;:=\; \widehat f_\lambda\circ \varphi : W_F\to \widehat{R_\lambda}.
\]
Since $w\in W_F^{\,s}$ and $\lambda(\varphi(w))\neq 1$, we have $\varphi_\lambda(w)\neq 1$.
Moreover $\varphi_\lambda(W_F^{s+})=1$ because $\varphi(W_F^{s+})=1$.
Hence
\[
\dep^{\std}(\varphi_\lambda)\;\ge\; s=\dep^{\std}(\varphi). \tag{I}
\]

Now let $f:R\to T$ be any $F$--morphism with $R$ induced, and put
\(
\varphi_R:=\widehat f\circ \varphi.
\)
By definition of $r=\dep(\varphi)$ as a supremum over induced $R$ and $f$, one has
\(
\dep(\varphi_R)\le r.
\)
For an induced torus $R$, Lemma~\ref{lem:HH} gives
\[
\dep^{\std}(\varphi_R)
\;\le\;
\Phi_R\!\bigl(\dep(\varphi_R)\bigr)
\;\le\;
\Phi_R(r),
\]
where the last inequality uses that $\Phi_R$ is non-decreasing (Lemma~\ref{lem:PhiT-monotone}).
Taking the supremum over all induced $R$ and all $f:R\to T$ yields
\[
\sup_{f:R\to T}\ \dep^{\std}(\varphi_R)
\;\le\;
\sup_{f:R\to T}\ \Phi_R(r). \tag{II}
\]

Applying (II) to the particular pair $(R_\lambda,f_\lambda)$ constructed above and combining with (I),
we obtain
\[
\dep^{\std}(\varphi)
\;\le\;
\Phi_{R_\lambda}(r)
\;\le\;
\sup_{f:R\to T}\ \Phi_R(r).
\]
Finally, by Definition~\ref{def:PhiT}, the right-hand side
is exactly $\Phi_T(r)=\Phi_T\bigl(\dep(\varphi)\bigr)$. Thus
\[
\dep_{\std}(\varphi)\ \le\ \Phi_T\!\bigl(\dep(\varphi)\bigr),
\]
as desired.
\end{proof}

\subsubsection*{\textbf{The inequality in Proposition \ref{prop:Phi-bounds-std} could be strict}} Let $F$ be a non-archimedean local field and set  
\[
T_1=\Res_{E_1/F}\G_m,\qquad 
T_2=\Res_{E_2/F}\G_m,\qquad 
T=T_1\times T_2,
\]
where  

\smallskip
\begin{itemize}[leftmargin=*]
\item $E_1/F$ is unramified, so $e_1=1$ and $\varphi_{E_1/F}(t)=t$;
\item $E_2/F$ is totally wildly ramified of degree $e_2=p$ with lower break $u_0\ge1$, giving  
      $\displaystyle 
      \varphi_{E_2/F}(t)=
      \frac{u_0}{p}+t-u_0
      \ (t\ge u_0).
      $
\end{itemize}

\smallskip
Choose $r>\tfrac{u_0}{p}$ (hence $e_2r\ge u_0$) and define the character  
\[
\chi=(\chi_1,1)\in\Hom\!\bigl(T(F),\CC^\times\bigr),\qquad
\dep_{T_1}(\chi_1)=r.
\]
Let $\varphi=L_T(\chi)=(\varphi_1,1)$ be its LLC parameter.

\medskip\noindent
Then
\[
\dep_T(\chi)=\dep(\varphi)=r,\qquad
\dep_{\mathrm{std}}(\varphi)=\varphi_{E_1/F}(e_1r)=r,
\]
while
\[
\Phi_T\!\bigl(\dep(\varphi)\bigr)=
\max\!\{\varphi_{E_1/F}(r),\varphi_{E_2/F}(e_2r)\}
=\varphi_{E_2/F}(e_2r)>
r=\dep_{\mathrm{std}}(\varphi).
\]
Hence $\Phi_T(\dep(\varphi))>\dep_{\mathrm{std}}(\varphi)$, so the inequality in Proposition \ref{prop:Phi-bounds-std} is strict.

\subsection{Depth preservation for induced torus}
The following result is proved in \cite[Theorem 6]{MishraPattanayak2015}. We include the full proof to make the proof of general case self contained. 
\begin{lemma}\label{lem:depth-r-simple-induced}
Let \(F\) be a non-archimedean local field and let \(E/F\) be a finite separable
extension of ramification index \(e\).
Put
\[
T \;=\; \Res_{E/F}\,\Gm.
\]
For every real number \(r\ge0\) the local Langlands correspondence for tori
gives a canonical isomorphism
\[
\Hom\!\bigl(T(F)/T(F)_r,\CC^\times\bigr)
\;\xrightarrow{\;\;\cong\;\;}
H^1\!\Bigl(W_F/W_F^{\phi_{E/F}(e r)},\,\widehat{T}^{\,W_F^{\phi_{E/F}(e r)}}\Bigr).
\]
\end{lemma}

\begin{proof}
The case \(r=0\) is a special case of \cite[Theorem 1]{MM15}, so assume  \(r>0\).
Since \(T(F)=E^\times\),
\[
T(F)/T(F)_r
\;=\;
E^\times/E^\times_{e r}.
\]Local class-field theory isomorphism gives,
\[
\Hom\!\bigl(E^\times/E^\times_{e r},\CC^\times\bigr)
\;\cong\;
\Hom\!\bigl(W_E/W_E^{e r},\CC^\times\bigr)
\;\cong\;
H^1\!\bigl(W_E/W_E^{e r},\CC^\times\bigr).
\]
By Shapiro's lemma,
\( H^1\!\bigl(W_E,\CC^\times\bigr)\;\cong\;
   H^1\!\bigl(W_F,\Ind_{W_E}^{W_F}\CC^\times\bigr)=H^1(W_F,\widehat{T}) \).
For the truncated groups this gives,
\[
H^1\!\bigl(W_E/W_E^{e r},\CC^\times\bigr)
\;\cong\;
H^1\!\bigl(
   W_F/W_F^{\phi_{E/F}(e r)},
   \widehat{T}^{\,W_F^{\phi_{E/F}(e r)}}
\bigr),
\]
since \(W_E^{e r}=W_F^{\phi_{E/F}(e r)}\) by the compatibility of
upper-numbering filtrations with finite extensions.
Putting these together yields
\[
\Hom\!\bigl(T(F)/T(F)_r,\CC^\times\bigr)
\;\xrightarrow{\;\sim\;}
H^1\!\bigl(
  W_F/W_F^{\phi_{E/F}(e r)},
  \widehat{T}^{\,W_F^{\phi_{E/F}(e r)}}
\bigr),
\]
completing the proof.
\end{proof}
\begin{lemma}\label{lem:depth-pres-ind}
Let \(T\) be an induced torus and let \(\mathcal{L}_T\) be local Langlands correspondence isomorphism. 
Then for any character \(\chi\in\Hom\!\bigl(T(F),\CC^\times\bigr)\)  
\[
\dep_T(\chi)\;=\;\dep(\mathcal{L}_T(\chi)).
\]
\end{lemma}

\begin{proof}
Write \(T=\prod_{i=1}^{m}T_i\) with \(T_i\coloneqq\Res_{E_i/F}\Gm\).  
For each \(i\) put \(\chi_i:=\chi\!\restriction_{T_i(F)}\) and 
\(\varphi_i:=\pr_i\circ\varphi\) where \(\pr_i:\widehat{T}\twoheadrightarrow\widehat{T_i}\) is the
projection.  By Lemma~\ref{lem:depth-r-simple-induced} we get,
\[
\dep_{T_i}(\chi_i)=\dep(\varphi_i)
\quad\text{for each }i.
\tag{\(\dagger\)}
\]
By Definition~\ref{def:depth-parameter}, 
\[
\dep(\varphi)=\max_i\dep(\varphi_i).
\]
The result follows. 
\end{proof}
\subsection{Depth preservation for arbitrary torus}

\begin{lemma}
\label{lem:gen-by-induced}
Let \(F\) be a non-archimedean local field, let \(T/F\) be a torus, and fix \(r>0\).
Then
\[
   T(F)_r / T(F)_{r+}
   \;=\;
   \Bigl\langle\,
      f\bigl(R(F)_r/R(F)_{r+}\bigr)\,
      \;\Bigm|\;
      f:R\rightarrow T,\;
      R/F\text{ induced}
   \Bigr\rangle .
\]
\end{lemma}
\begin{proof}
For the strict henselization \(F^{\text{sh}}\), Proposition \ref{P:depth-shift} gives:
   \[
   T(F^{\text{sh}})_r = \langle T(F^{\text{sh}})_{n+1},  f\left(R(F^{\text{sh}})_r\right) \mid \ \ f:R\to T \text{ defined over } F^{\mathrm{sh}}, \text{ with } R \text{ an induced } F^{\mathrm{sh}}\text{-torus}\Big{\rangle}. \tag{I}
   \]
where \(n = \lfloor r \rfloor\). Put \(V = T(F^{\text{sh}})_r / T(F^{\text{sh}})_{r+}\), a vector group over the residue field \(\kappa^{\text{sh}}\) and put 
   \[
   W = \left\langle f\left(R(F^{\text{sh}})_r / R(F^{\text{sh}})_{r+}\right) \mid F^{\sh}\text{-morphism } f: R \to T,  R \text{ induced }F^{\sh}\text{-torus} \right\rangle \subseteq V
   \]
 a \(\Gamma:=\Gal(F^{\sh}/F)\)-stable subspace. Since \(T(F^{\text{sh}})_{n+1} \subseteq T(F^{\text{sh}})_{r+}\), quotienting by \(T(F^{\text{sh}})_{r+}\) yields:
 \[
 V=W.\tag{II}
 \]
 Taking \(\Gamma\)-invariants gives \(W^\Gamma = V^\Gamma\). We wish to show:
   \[
 W^\Gamma = \left\langle f\left(R(F)_r / R(F)_{r+}\right) \mid f: R \to T,  R \text{ induced }F \text{-torus}\right\rangle.
   \]
To show this, define \( U_R = R(F^{\text{sh}})_r / R(F^{\text{sh}})_{r+} \) for each induced $F^{\sh}$-torus \( R \) and homomorphism \( f: R \to T \) defined over \( F^{\sh} \). Let \( S = \left\langle f\left(R(F)_r / R(F)_{r+}\right) \mid f: R \to T,  R \text{ induced }F\text{-torus } \right\rangle \). Each \( f\left(R(F)_r / R(F)_{r+}\right) = f\left((U_R)^\Gamma\right) \) is a \( \kappa \)-vector space (since the residue field of \( F \) is \( \kappa \)), and \( S \) is the \( \kappa \)-span of these images. For each \( f \) and \( R \), consider the extension of scalars \( f\left((U_R)^\Gamma\right) \otimes_\kappa \kappa^{\text{sh}} \). Since \( U_R \) is defined over \( \kappa \), Galois descent gives \( (U_R)^\Gamma \otimes_\kappa \kappa^{\text{sh}} \cong U_R \).
The map \( f: U_R \to V \) is \( \Gamma \)-equivariant and defined over \( F \), so it induces \( (U_R)^\Gamma \to V^\Gamma \) and after base change, \(  (U_R)^\Gamma \otimes_\kappa \kappa^{\text{sh}} \to V^\Gamma \otimes_\kappa \kappa^{\text{sh}} \) recovers the original \( f: U_R \to V \). In other words:
     \[
     f\left((U_R)^\Gamma\right) \otimes_\kappa \kappa^{\text{sh}} \cong f(U_R).
     \]
Taking the sum over all \( f \) and \( R \):
     \[
     S \otimes_\kappa \kappa^{\text{sh}} = \sum_{f,R} \left[ f\left((U_R)^\Gamma\right) \otimes_\kappa \kappa^{\text{sh}} \right] = \sum_{f,R} f(U_R) = W = V.
     \]
Since \( S \subseteq V^\Gamma \) and \( V^\Gamma \otimes_\kappa \kappa^{\text{sh}} = V \) by Galois descent, \[
     S \otimes_\kappa \kappa^{\text{sh}} = V^\Gamma \otimes_\kappa \kappa^{\text{sh}}.
     \]
implies \( S = V^\Gamma \), i.e.,  \( W^\Gamma = V^\Gamma = S \), which gives the required equality:
     \[
     W^\Gamma = \left\langle f\left(R(F)_r / R(F)_{r+}\right) \mid f: R \to T,  R \text{ induced} \right\rangle.
     \]

 This completes the proof of the lemma. 
 \end{proof}
\begin{theorem}\label{thm:dep-r-LLC-forTori}
    Let \(F\) be a non-archimedean local field and \(T/F\) an \(F\)-torus. Denote its complex dual torus by \(\widehat{T}\). For every real number \(r \geq 0\), the following holds:
For every character \(\chi \in \operatorname{Hom}(T(F), \mathbb{C}^\times)\) with corresponding parameter \(\varphi = \mathcal{L}_T(\chi)\), the depth of \(\chi\) equals the depth of \(\varphi\):
   \[
   \operatorname{dep}(\chi) = \operatorname{dep}(\varphi),
   \]
where \(\operatorname{dep}(\varphi)\) is defined as above.
\end{theorem}
\begin{proof}
  For any induced torus \(R\) and morphism \(f: R \to T\), let \(\chi_f = \chi \circ f\) be the character of \(R(F)\) obtained by pullback. By functoriality of LLC, the parameter of \(\chi_f\) is \(\hat{f} \circ \varphi\). The depth of \(\chi_f\) satisfies \(\operatorname{dep}(\chi_f) \leq \operatorname{dep}(\chi)\) since \(f\) is a morphism of tori, and the minimal congruence filtration is functorial. For induced tori, \(\operatorname{dep}_R(\hat{f} \circ \varphi) = \operatorname{dep}(\chi_f)\) by Lemma \ref{lem:depth-pres-ind}. Thus:
  \[
  \operatorname{dep}_R(\hat{f} \circ \varphi) = \operatorname{dep}(\chi_f) \leq \operatorname{dep}(\chi).
  \]
  This holds for all \(f\), so:
  \[
  \sup_{\substack{f: R \to T \\ R \text{ induced}}} \operatorname{dep}_R(\hat{f} \circ \varphi) \leq \operatorname{dep}(\chi). \tag{I}
  \]

We now show the reverse inequality. 
Suppose $\varphi$ has depth \(r\).  Then for all $f:R\rightarrow T$, since \(\mathcal{L}_R(\chi \circ f) = \widehat{f}(\mathcal{L}_T(\chi)) \), by Proposition \ref{prop:Phi-bounds-std},  \( \chi \circ f \text{ trivial on } R(F)_{r+}\). Equivalently, $\chi$ is trivial on $f(R(F)_{r+})$ for all $f$. The reverse inequality now follows from Lemma \ref{lem:gen-by-induced}. To see this, assume on the contrary that \(\text{dep}(\chi) = d > r\).  Then there exists \(t \in T(F)_d \setminus T(F)_{d+}\) such that \(\chi(t) \neq 1\) (since \(\chi\) is trivial on \(T(F)_{d+}\) but not on \(T(F)_d\)).  By Lemma \ref{lem:gen-by-induced}, there exists an induced torus \(R\), a morphism \(f: R \to T\), and \(u \in R(F)_d \setminus R(F)_{d+}\) such that:  
  \[
  f(u) = t \cdot v \quad \text{for some} \quad v \in T(F)_{d+}.
  \] Since \(\chi\) is trivial on \(T(F)_{d+}\), \(\chi(f(u)) = \chi(t) \cdot \chi(v) = \chi(t) \neq 1\).  Let \(\chi_f = \chi \circ f\). Then \(\chi_f(u) = \chi(f(u)) \neq 1\), and \(u \in R(F)_d \setminus R(F)_{d+}\), so \(\text{dep}(\chi_f) \geq d\).  But \(\text{dep}(\chi_f) = \text{dep}_R(\widehat{f} \circ \varphi) \leq r\) (since \(\text{dep}(\varphi) = r\)).  This contradicts \(d > r\). Hence, \(\text{dep}(\chi) \leq r\), completing the proof.
\end{proof}
\begin{remark}
For tori that are not tamely ramified, it is known that depth preservation fails
for the \emph{standard} notion of depth of Langlands parameters
(cf.\ \cite{aubert2017,MishraPattanayak2015} and the discussion in \cite{DeBackerChenTsai}).
The revised notion of depth introduced here is designed precisely to account for
this phenomenon and restore depth preservation for arbitrary tori.
\end{remark}

\section{Truncated torus isomorphisms for $\ell$-close fields}
\label{sec:ell-close}

We recall Deligne's notion of $\ell$-close fields and show that the depth-comparison theorem (Theorem~\ref{thm:dep-r-LLC-forTori}) together with Deligne's Galois isomorphism yield a canonical truncated isomorphism tori over $\ell$-close fields. This result generalizes the congruence isomorphisms of \cite[Th.~1.1.1]{AV24}.

\begin{definition}[Deligne]\label{def:ell-close} 
Two non-archimedean local fields $F$ and $F'$ are \emph{$\ell$-close} ($\ell\in\mathbf{Z}_{>0}$) if there exists an isomorphism of topological rings 
\[
\scrO_{F}/\fracp_{F}^{\ell} \cong \scrO_{F'}/\fracp_{F'}^{\ell}.
\]
Deligne \cite{Deligne84} constructs a canonical isomorphism,
\[
\operatorname{Del}_{\ell}:\Gamma_{F}/I_{F}^{\ell} \xrightarrow{\;\sim\;} \Gamma_{F'}/I_{F'}^{\ell},
\]
functorial for finite $\ell$-ramified $\Gamma_{F}$-modules. 
\end{definition}
\begin{proposition}\label{prop:H1-dep-preserve}
    Let \(F\) and \(F'\) be \(\ell\)-close fields, \(T\) and \(T'\) be tori defined over \(F\) and \(F'\) such that \(T'\) is a transfer of  \(T\) along \(\operatorname{Del}_{\ell}\). Then for any $0\leq r\in \mathbb{R}$ such that $\Phi_T(r)\leq \ell$, the canonical isomorphism 
    \[
   \operatorname{Del}_{T,\ell}: H^1(W_F/W_F^r,\widehat{T}^{W_F^r})\cong  H^1(W_{F'}/W_{F'}^r,\widehat{T'}^{W_{F'}^r}).
    \]
    preserves depth. 
\end{proposition}
\begin{proof}
Put \(r:=\dep(\varphi)\) and \(s:=\Phi_T(r)\). By hypothesis \(s\le \ell\).
By Proposition~\ref{prop:Phi-bounds-std} we have
\[
\dep^{\std}(\varphi)\ \le\ \Phi_T\!\bigl(\dep(\varphi)\bigr)\ =\ s,
\]
hence \(\varphi(W_F^{s+})=1\). In particular, \(\varphi\) factors through \(W_F/W_F^{s+}\),
and therefore also through \(W_F/W_F^{s}\).
Via the \(\ell\)-closeness isomorphism
\(\Del_{\ell}:\Gamma_F/I_F^{\ell}\xrightarrow{\sim}\Gamma_{F'}/I_{F'}^{\ell}\),
Deligne's theory yields a canonical identification of the truncated Weil data up to level \(s\)
(and hence an isomorphism on the corresponding cohomology sets)
\[
H^1(W_F/W_F^{s},\ \widehat{T}^{W_F^{s}})
\ \cong\
H^1(W_{F'}/W_{F'}^{s},\ \widehat{T'}^{W_{F'}^{s}}),
\]
which transfers \(\varphi\) to a parameter \(\varphi'\).

Now let \(f:R\to T\) be any \(F\)-morphism with \(R\) an induced \(F\)-torus. Let \(R'\) be the
transfer of \(R\) to \(F'\) and \(f':R'\to T'\) the transferred morphism.
Write \(\varphi_R:=\widehat f\circ\varphi:W_F\to\widehat R\) and
\(\varphi_{R'}':=\widehat f'\circ\varphi':W_{F'}\to\widehat{R'}\).
By functoriality of Deligne transfer, \(\varphi_{R'}'\) corresponds to \(\varphi_R\) under the same
truncation isomorphism.

Write \(R=\prod_{j=1}^k \Res_{E_j/F}\Gm\) and let \(\psi_j:W_F\to \widehat R_j\) be the \(j\)-th component
of \(\varphi_R\) (so \(\widehat R_j\simeq \Ind_{W_{E_j}}^{W_F}\CC^\times\)).
Recall that
\[
\dep_R(\varphi_R)
=\max_{1\le j\le k}\left(\frac{1}{e_j}\,\varphi_{E_j/F}^{-1}\bigl(\dep^{\std}(\psi_j)\bigr)\right),
\]
where \(e_j=e(E_j/F)\).
Since \(r=\dep(\varphi)=\sup_{(R,f)}\dep_R(\widehat f\circ\varphi)\), we have \(\dep_R(\varphi_R)\le r\) for every such \(f\).
Fix \(j\). From the defining formula above we obtain
\[
\frac{1}{e_j}\,\varphi_{E_j/F}^{-1}\bigl(\dep^{\std}(\psi_j)\bigr)\ \le\ \dep_R(\varphi_R)\ \le\ r,
\]
hence (using monotonicity of \(\varphi_{E_j/F}\))
\[
\dep^{\std}(\psi_j)\ \le\ \varphi_{E_j/F}(e_j r)\ \le\ \Phi_R(r)\ \le\ \Phi_T(r)\ =\ s,
\]
where \(\Phi_R(r)=\max_j\varphi_{E_j/F}(e_j r)\) and the inequality \(\Phi_R(r)\le \Phi_T(r)\) uses
Definition~\ref{def:PhiT}.

Thus \(\dep^{\std}(\psi_j)\le s\) for every \(j\), so all ramification information relevant to
\(\varphi_{E_j/F}^{-1}(\dep^{\std}(\psi_j))\) lies \emph{within level \(s\)}.
Since \(s\le \ell\), Deligne's \(\ell\)-closeness comparison identifies the upper ramification filtrations
up to level \(s\) and, for each \(j\), identifies the truncated extensions \(E_j/F\) and their transfers \(E_j'/F'\)
up to that level. In particular:
\begin{itemize}
\item \(\dep^{\std}(\psi_j')=\dep^{\std}(\psi_j)\), where \(\psi_j'\) is the \(j\)-th component of \(\varphi_{R'}'\);
\item \(e(E_j/F)=e(E_j'/F')\) and the Hasse--Herbrand functions agree on \([0,s]\), hence
\(\varphi_{E_j/F}^{-1}(t)=\varphi_{E_j'/F'}^{-1}(t)\) for all \(t\in[0,s]\).
\end{itemize}
Therefore, for each \(j\),
\[
\frac{1}{e_j}\,\varphi_{E_j/F}^{-1}\bigl(\dep^{\std}(\psi_j)\bigr)
=
\frac{1}{e_j'}\,\varphi_{E_j'/F'}^{-1}\bigl(\dep^{\std}(\psi_j')\bigr),
\]
and taking the maximum over \(j\) gives
\[
\dep_R(\widehat f\circ\varphi)\ =\ \dep_{R'}(\widehat f'\circ\varphi').
\]

Finally, as \(f:R\to T\) ranges over all morphisms from induced \(F\)-tori to \(T\), the transferred maps
\(f':R'\to T'\) range over the corresponding class for \(T'\). Hence
\[
\dep(\varphi)
=\sup_{f:R\to T}\dep_R(\widehat f\circ\varphi)
=\sup_{f':R'\to T'}\dep_{R'}(\widehat f'\circ\varphi')
=\dep(\varphi'),
\]
as claimed.
\end{proof}

\begin{theorem}\label{thm:congruent-isom}
Let $F$ and $F'$ be $\ell$-close. Let $T/F$ be a torus and $T'/F'$ its transfer via the isomorphism $\Gamma_{F}/I_{F}^{\ell}\cong\Gamma_{F'}/I_{F'}^{\ell}$. For any $0\leq r\in \mathbb{R}$ such that $\Phi_T(r)\leq \ell$, we have,
\[
T(F)/T(F)_{r} \xrightarrow{\;\cong\;} T'(F')/T'(F')_{r},
\]
canonically and functorially, and this isomorphism intertwines the local Langlands correspondences for $T$ and $T'$. 
\end{theorem}

\begin{proof}
Put \(s:=\Phi_T(r)\). By hypothesis \(s\le \ell\).

\medskip\noindent
Since \(T'\) is the Deligne transfer of \(T\) along \(\Del_\ell:\Gamma_F/I_F^\ell\simeq \Gamma_{F'}/I_{F'}^\ell\),
the induced \(F\)-tori \(R\) and morphisms \(f:R\to T\) correspond
functorially to induced \(F'\)-tori \(R'\) and morphisms \(f':R'\to T'\), and for such a pair one has
\(\Phi_{R'}(t)=\Phi_R(t)\) for all \(t\) with \(\Phi_R(t)\le \ell\) (since the ramification data and Hasse--Herbrand
functions agree up to level \(\ell\)).
Using Definition~\ref{def:PhiT} ,
\[
\Phi_{T'}(r)
=\sup_{f':R'\to T'}\Phi_{R'}(r)
=\sup_{f:R\to T}\Phi_R(r)
=\Phi_T(r)
=s.
\]

\medskip\noindent
Let
\[
\mathscr X_r(T)\ :=\ \Hom\!\bigl(T(F)/T(F)_r,\C^\times\bigr).
\]
A character \(\chi\in\Hom(T(F),\C^\times)\) lies in \(\mathscr X_r(T)\) if and only if \(\dep(\chi)<r\)

Let \(\mathcal L_T:\Hom(T(F),\C^\times)\xrightarrow{\sim} H^1(W_F,\widehat T)\) be the LLC for tori.
By Theorem~\ref{thm:dep-r-LLC-forTori},
\[
\dep\bigl(\mathcal L_T(\chi)\bigr)=\dep(\chi)\quad \text{for all }\chi.
\]
Hence \(\mathcal L_T\) induces a bijection between \(\mathscr X_r(T)\) and the subset
\[
\mathscr Y_r(T)\ :=\ \{\ \alpha\in H^1(W_F,\widehat T)\ \mid\ \dep(\alpha)<r\ \}.
\]
Moreover, if \(\alpha\in \mathscr Y_r(T)\), then \(\dep(\alpha)<r\) implies (by Proposition~\ref{prop:Phi-bounds-std})
\[
\dep^{\std}(\alpha)\ \le\ \Phi_T\!\bigl(\dep(\alpha)\bigr)\ \le\ \Phi_T(r)=s,
\]
so \(\alpha\) is trivial on \(W_F^{s+}\) and therefore comes from a class in
\(H^1(W_F/W_F^{s},\widehat T^{W_F^{s}})\) by inflation. Thus we may (and do) view
\(\mathscr Y_r(T)\) as a subset of \(H^1(W_F/W_F^{s},\widehat T^{W_F^{s}})\), and the map
\[
\mathcal L_{T,r}:\mathscr X_r(T)\hookrightarrow H^1(W_F/W_F^{s},\widehat T^{W_F^{s}})
\]
is a bijection onto \(\mathscr Y_r(T)\) (not necessarily onto all of \(H^1\)).

The same discussion applies to \(T'\) over \(F'\): we obtain \(\mathscr X_r(T')\) and the depth-\(<r\)
subset \(\mathscr Y_r(T')\subset H^1(W_{F'}/W_{F'}^{s},\widehat T'^{W_{F'}^{s}})\), with
\(\mathcal L_{T',r}:\mathscr X_r(T')\xrightarrow{\sim}\mathscr Y_r(T')\).

\medskip\noindent
Since \(s\le \ell\), Deligne's \(\ell\)-closeness comparison yields a canonical isomorphism
\[
\Del_{T,\ell}:\ H^1(W_F/W_F^{s},\widehat T^{W_F^{s}})
\ \xrightarrow{\ \sim\ }\
H^1(W_{F'}/W_{F'}^{s},\widehat T'^{W_{F'}^{s}}),
\]
and by Proposition~\ref{prop:H1-dep-preserve} it preserves the depth \(\dep\) of parameters. Therefore
\(\Del_{T,\ell}\) restricts to a bijection
\[
\Del_{T,\ell}:\ \mathscr Y_r(T)\ \xrightarrow{\ \sim\ }\ \mathscr Y_r(T').
\]
Composing with \(\mathcal L_{T,r}\) and \(\mathcal L_{T',r}\) gives a canonical isomorphism
\[
\mathscr X_r(T)\ \xrightarrow{\ \sim\ }\ \mathscr X_r(T').
\]

\medskip\noindent
Since \(T(F)/T(F)_r\) and \(T'(F')/T'(F')_r\) are finite abelian groups, Pontryagin duality yields a canonical
isomorphism
\[
T(F)/T(F)_r\ \xrightarrow{\ \sim\ }\ T'(F')/T'(F')_r,
\]
and by construction this intertwines the local Langlands correspondences for \(T\) and \(T'\) at depth \(<r\).
\end{proof}

\begin{remark} 
When $T$ is tamely ramified, $s=r$ and the bound simplifies to $\ell\ge r$.
\end{remark} 

\section{Truncated parahoric isomorphisms for close fields}
\label{sec:wild-yu}
\subsection{Vertices and maximally unramified tori}\label{subsec:maxunramified}
\begin{definition}\cite[Fact 3.4.1 and Def. 3.4.2]{KalethaRegSC}.
Let \(G\) be a connected reductive group over a non-archimedean local field \(F\).
A maximal torus \(S\subset G\) is said to be \emph{maximally unramified} if, writing
\(S'\subset S\) for its maximal unramified subtorus, any (hence all) of the
following equivalent conditions hold:
\begin{enumerate}
  \item \(S'\) has maximal dimension among the unramified subtori of \(G\);
  \item \(S'\) is not properly contained in a larger unramified subtorus of \(G\);
  \item \(S=\operatorname{Cent}_G(S')\);
  \item \(S\times_F F^{\mathrm u}\) is a minimal Levi subgroup of \(G\times_F F^{\mathrm u}\);
  \item the inertia group \(I_F\) acts on the root system \(R(S,G)\) preserving
        some set of positive roots.
\end{enumerate}

\end{definition}

We now recall a correspondence between vertices and maximally unramified elliptic maximal tori. 
\begin{lemma}
\label{lem:vertex-torus-bijection}
Let $G$ be a connected reductive group over a non-archimedean local field $F$ with Bruhat-Tits reduced building $\mathcal{B}^{\mathrm{red}}(G,F)$. For a point $x\in\mathcal{B}^{\mathrm{red}}(G,F)$, the following are equivalent:
\begin{enumerate}
  \item $x$ is a vertex of the building;
  \item there exists a maximally unramified elliptic maximal $F$-torus $S\subset G$ whose unique Frobenius-fixed point satisfies $x_S = x$.
\end{enumerate}

\end{lemma}

\begin{proof}
The first part is \cite[Lemma 3.4.3]{KalethaRegSC}. The second part is \cite[Lemma~3.4.4(2)]{KalethaRegSC}.
\end{proof}

\subsection{Close field transfer of reductive groups}
\label{sub:vertex}

\begin{proposition}\label{prop:vertex-transfer} 
Let $G/F$ be a connected reductive group, and let $S \subset G$ be a maximally unramified elliptic maximal $F$-torus with splitting field $E/F$. There exists an integer \( \ell_0 \) depending on \( G\) such that for any integer \(\ell \geq \ell_0 \), and $F'$ a local field $\ell$-close to $F$ via $\iota_\ell: \mathscr{O}_F/\mathfrak{p}_F^\ell \to \mathscr{O}_{F'}/\mathfrak{p}_{F'}^\ell$,  there exists a reductive $F'$-group $G'$ and a maximally unramified elliptic maximal $F'$-torus $S' \subset G'$ such that:
\begin{enumerate}
    \item There exists a canonical isomorphism
    \[
    S(F)/S(F)_m \xrightarrow{\simeq} S'(F')/S'(F')_m
    \]
    for all $m$ such that $\Phi_S(m)\leq \ell$.
    \item There exists a simplicial isomorphism of the Bruhat-Tits buildings $\mathcal{B}_{\ell}:\mathcal{B}^{\mathrm{red}}(G,F)\cong \mathcal{B}^{\mathrm{red}}(G', F')$, with the toral vertices $x_S$ and $x_{S'}$ corresponding under this isomorphism.
\end{enumerate}
\end{proposition}

\begin{proof}
The existence of $G'$ and (ii) is proved in \cite[\S 1C2 and \S 1C3]{ganapathy2022}. It comes equipped with a \(\Del_\ell\) equivariant identification $\Phi_\ell$ of absolute root data of $(G,T)$ and $(G',T')$. Hence $\Phi_\ell$ also provides a \(\Del_\ell\) equivariant isomorphism $W(G,T)\xrightarrow{\sim}W(G',T')$ of Weyl groups and consequently an isomorphism \(\Phi_{\ell,H^1}:H^1(\Gamma_{F}/I_{F}^\ell,\ W(G,T))\cong H^1(\Gamma_{F'}/I_{F'}^\ell,\ W(G',T'))\). Let $S'$ be the transfer of $S$ to $G'$ along the isomorphism $\Phi_{\ell,H^1}$. 
Then (i) follows from Theorem \ref{thm:congruent-isom}. 
\end{proof}

\begin{definition}[Level-$\ell$ congruence datum]
\label{def:congruence-datum}
Let $F$ and $F'$ be $\ell$-close local fields. A \emph{congruence datum of level $\ell$} is a triple
\[
  D_\ell = \bigl( \psi_\ell,\; \Del_\ell,\; \Phi_\ell \bigr)
\]
consisting of,
\begin{enumerate}
\item an isomorphism of truncated valuation rings
\[
  \psi_\ell: \mathcal{O}_F/\pi_F^{\,\ell} \xrightarrow{\;\sim\;} \mathcal{O}_{F'}/\pi_{F'}^{\,\ell},
\]
\item Deligne's canonical isomorphism of Galois quotients
\[
  \Del_\ell: \Gamma_F/I_F^{\ell} \xrightarrow{\;\sim\;} \Gamma_{F'}/I_{F'}^{\ell},
\]
\item a bijection $\Phi_\ell:\Phi(G,T)\xrightarrow{\sim}\Phi(G',T')$ compatible with $\Del_\ell$ and preserving coroots.
\end{enumerate}
All constructions below depend only on the fixed choice of $D_\ell$.
\end{definition}
\subsection{Root groups and their splitting fields}\label{subsec:root-splitting-field}
We will recall some standard theory from \cite{BT84}.
Let \(F\) be a non-archimedean local field and let \(G\) be a connected
reductive \(F\)-group.  
Choose a maximal \(F\)-split torus \(A\subset G\) and write
\(\Phi(G,A)\) for the corresponding relative root system. Choose a base point $x_{0}$ in the apartment $\mathcal{A}(A,F)$ of ${A}$. For any point $x\in\mathcal{A}(A,F)$, define,
\[
  f_{x} : \Phi(G,A)\longrightarrow\mathbb{R}, \qquad a\longmapsto -a\bigl(x-x_{0}\bigr).
\]

Now let \(T\) be the centralizer of a maximally \(F^{\mathrm{un}}\)-split \(F\)-torus containing
\(A\); then \(T/F\) is a maximal torus.
Every absolute root \(\alpha\in\Phi(G,T)\) determines a
one-parameter unipotent subgroup \(U_\alpha\subset G_{F^{\mathrm s}}\).
For each absolute root \(\alpha\) let
\[
   L_\alpha \;:=\; F^{\mathrm {sep}}{}^{\Gamma_\alpha}, 
   \quad\text{where}\;
   \Gamma_\alpha=\{\sigma\in\Gal(F^{\mathrm {sep}}\!/F)\mid \sigma(U_\alpha)=U_\alpha\}.
\]
The field \(L_\alpha\) is the smallest extension of \(F\) over
which \(U_\alpha\) is defined; over \(L_\alpha\) one has an
\(L_\alpha\)-isomorphism \(x_\alpha:\mathbf G_{a,L_\alpha}\!\xrightarrow\sim\!U_\alpha\),
unique up to scaling by \(L_\alpha^{\times}\).
Now fix a relative root \(a\in\Phi(G,A)\).
Choose any absolute root \(\alpha\) restricting to \(a\) on \(A\). Define the minimal field of definition \(F_a\) as follows.
\begin{enumerate}[label=(\alph*)]
    \item \emph{Short case (\(2a\notin\Phi(G,A)\)).}
          Define
          \[
              F_a \;:=\; L_\alpha .
          \]
    \item \emph{Long--short case (\(2a\in\Phi(G,A)\)).}
          Pick absolute roots \(\alpha,\alpha'\) with
          \(\alpha|_A=\alpha'|_A=a/2\) and \(\alpha+\alpha'=2a\).  
          Then
          \[
              F_{a}:=L_\alpha=L_{\alpha'}.
          \] 
In either case, the field \(F_a\) is independent of the chosen lift \(\alpha\).
\end{enumerate}
Write \(e_a := e(F_a/F)\) for the ramification index.

\begin{lemma}
\label{lem:root-filtration-HH}
Let $F$ and $F'$ be $\ell$-close local fields and let $G/F$ be a connected reductive group with maximal $F$-split torus $A\subset G$. Fix a facet $x\in\mathcal{A}(A,F)$ and put $x':=\mathcal{B}_\ell(x)$. For each root $a\in\Phi(G,A)$, denote by $F_a/F$ (resp. $F_{a'}/F'$) the minimal field of definition of the root subgroup $U_a$ (resp. $U_{a'}$ with $a'=\Phi_\ell(a)$), and let $\phi_{F_a/F}$ be the Hasse-Herbrand function of that extension. Assume $r>0$ satisfies the uniform bound
\[
\; \ell \;\ge\; \phi_{F_a/F}\!\bigl(e_a\,r\bigr) \quad\text{for every }a\in\Phi(G,A). \tag{$\ast$}
\]
Then for every root $a\in\Phi(G,A)$, there is a canonical bijection
\[
  U_a(F)_{x,0}\big/ U_a(F)_{x,r} \xrightarrow{\;\sim\;} U_{a'}(F')_{x',0}\big/ U_{a'}(F')_{x',r}.
\]
\end{lemma}

\begin{proof}
Let $a\in\Phi(G,A)$ and $a'\in\Phi(G',A')$ be the corresponding roots under the root datum
identification of Proposition~\ref{prop:vertex-transfer}. If $x\in\mathcal{A}(A,F)$ and
$x'=\mathcal{B}_\ell(x)\in\mathcal{A}(A',F')$, then by \cite[Proposition 4.4]{ganapathy2019}, the building isomorphism preserves affine root functions. In particular, for every $a\in\Phi(G,A)$ we have
\[
   f_x(a) \;=\; f_{x'}(a'),
\]
where $f_x$ (resp. $f_{x'}$) denotes the affine functional attached to $a$ at $x$
(resp. to $a'$ at $x'$).
Put $k:=f_x(a)=f_{x'}(a')$ and $k':=k+r$. By \cite[\S 4.3.2 and \S 4.3.5]{BT84}, $U_{a,k}/U_{a,k'}\cong \varpi^{\lceil e_a k\rceil}\mathcal{O}_{F_a}\big/ \varpi^{\lceil e_a k'\rceil}\mathcal{O}_{F_a}$. Hence the quotient depends only on the truncated ring $\mathcal{O}_{F_a}\big/\varpi_{F_a}^{\,\lceil e_a r\rceil}$.
For the finite Galois extension $F_a/F$, the jump $e_a r$ in lower numbering corresponds, under upper numbering, to $\phi_{F_a/F}(e_a r)$. Hypothesis $(\ast)$ gives $\phi_{F_a/F}(e_a r)\le\ell$. Deligne's $\ell$-closeness provides an isomorphism of truncated valuation rings
\[
   \mathcal{O}_{F_a}\big/\varpi^{\,\lceil e_a r\rceil} \;\xrightarrow{\;\sim\;}\;
   \mathcal{O}_{F_{a'}}\big/\varpi'^{\,\lceil e_a r\rceil}.
\]
From this we deduce the following equalities of local invariants:
\begin{enumerate}[label=(\roman*), leftmargin=2em]
   \item the residue fields are canonically identified, 
   \(\kappa_{F_a}\;\cong\;\kappa_{F_{a'}}\);
   \item the ramification indices agree, so that
   \([F_a:F]=[F_{a'}:F']\);
   \item the value groups attached to the root filtrations coincide,
   \[
      \Gamma_a \;:=\; \nu_a\!\bigl(U_a(F)\setminus\{1\}\bigr)
      \;=\; \nu_{a'}\!\bigl(U_{a'}(F')\setminus\{1\}\bigr)
      \;=\; \Gamma_{a'},
   \]
where $\nu_a$ (resp. $\nu_{a'}$) denotes the valuation on the root subgroup $U_a$ (resp. $U_{a'}$).
\end{enumerate}
With these invariants equal, the quotients
\(U_{a,k}/U_{a,k'}\) and \(U_{a',k}/U_{a',k'}\) are naturally
$\kappa$-vector spaces of the same dimension over isomorphic residue fields. 
Choosing any $\kappa$-basis on one side transfers canonically to the other,
yielding the desired bijection. 
Moreover, this bijection depends only on the congruence datum $D_\ell$,
and is therefore functorial and independent of auxiliary choices.
\end{proof}

\subsection{Definition:  depth-transfer function for reductive groups.}\label{subsec:defglobaldepcomp}
Let $G/F$ be a connected reductive group and fix a vertex $x\in\mathcal{B}^{\mathrm{red}}(G,F)$.
\begin{itemize}
\item Write $\mathscr{E}_x$ for the set of maximally unramified elliptic maximal $F$-tori of $G$ that fix $x$. For every $S\in\mathscr{E}_x$, we have the depth-transfer function $\Phi_{S} : \mathbb{R}_{\ge0} \to \mathbb{R}_{\ge0}$ introduced in Definition~\ref{def:PhiT}.
\item Choose a maximal split torus $A\subset G$ whose apartment contains $x$; let $\Phi(G,A)$ be the associated root system. For each $a\in\Phi(G,A)$, denote by $F_{a}/F$ the root-splitting field recalled in Section \ref{subsec:root-splitting-field} and by $e_a$ its ramification index; let $\phi_{F_a/F}$ be the Hasse-Herbrand function of this extension.
\end{itemize}
Define
\[
\boxed{\; \Phi_{G,x}(r) := \max~ \! \bigl\{ \max_{S\in\mathscr{E}_x}\Phi_S(r), \, \max_{a\in\Phi(G,A)}\phi_{F_a/F}(e_a\,r) \bigr\}, \qquad r\ge0. }
\]
The value $\Phi_{G,x}(r)$ depends only on $(G,x)$ (not on the auxiliary split torus $A$) and will serve as the \emph{global depth-comparison gauge} at the vertex $x$.

\begin{remark}[Depth--transfer as a \emph{normalised} Hasse--Herbrand function]
\label{R:Phi-vs-HH}
For a finite Galois extension $E/F$ of non-archimedean local fields with ramification
index $e$, consider the normalised form of the Hasse--Herbrand function \(\phi_{E/F}\):
\[
   \phi_{E/F}^{\mathrm{norm}}(r)\;:=\;
   \phi_{E/F}(e\,r), 
   \qquad r\in\mathbb R_{\ge 0},
\]
so that $\phi_{E/F}^{\mathrm{norm}}(r)=r$ whenever $E/F$ is tamely ramified.
For an induced torus $T=\Res_{E/F}\mathbb G_{m}$
\[
   \Phi_{T}(r)\;=\;\phi_{E/F}(e\,r)
                 \;=\;\phi_{E/F}^{\mathrm{norm}}(r) 
                 \qquad (r\ge 0) , 
\] 
so that $\Phi_{T}$ coincides with the normalised Hasse--Herbrand.

If $G=T$ is any torus, the ``root'' terms in $\Phi_{G,x}$ disappears and it reduces to $\Phi_{T}$ and
hence to $\phi_{E/F}^{\mathrm{norm}}$ when $T=\Res_{E/F}\mathbb G_{m}$. Thus, the depth--transfer function $\Phi_{G,x}$ is a generalisation of the normalised Hasse--Herbrand function that captures the combined ramification behaviour of all maximal tori and root data inside an arbitrary reductive group. 

\end{remark}
\subsection{Close field isomorphism of truncated parahorics}
\begin{theorem}
\label{thm:depth-comp-HH}
Let $F$ and $F'$ be $\ell$-close non-archimedean local fields and let $G/F$ be a connected reductive group. Fix a vertex $x\in\mathcal{B}^{\mathrm{red}}(G,F)$ and let $x':=\mathcal{B}_\ell(x)\in\mathcal{B}^{\mathrm{red}}(G',F')$ be its image under the building isomorphism as in Proposition \ref{prop:vertex-transfer}. For a real depth $r\ge0$, assume the truncation level satisfies
\(
\ell \;\ge\; \Phi_{G,x}(r). 
\) Then there is a canonical  isomorphism,
\[
  \Psi_{x,r}: \dfrac{G(F)_{x,\mathrm{b}}}{G(F)_{x,r}} \xrightarrow{\;\sim\;} \dfrac{G'(F')_{x',\mathrm{b}}}{G'(F')_{x',r}},
\]
depending only on the level-$\ell$ congruence datum $D_\ell$, where \(G(F)_{x,\mathrm{b}}\) (resp. \(G'(F')_{x',\mathrm{b}}\)) denotes the maximal bounded subgroup of \(G(F)_x\) (resp. \(G'(F')_{x'}\)). The isomorphism is respectful of restriction to minimal congruence filtration subgroups.
\end{theorem}

\begin{proof}
Choose any $S\in\mathscr{E}_x$. Since $\ell\ge\Phi_{G,x}(r)\ge\Phi_S(r)$, Theorem \ref{thm:congruent-isom} supplies an isomorphism
\[
  S(F)/S(F)_r \xrightarrow{\;\sim\;} S'(F')/S'(F')_r,
\tag{1}
\]
where $S'$ is the transfer of $S$ to $G'$.
Now fix a maximal split torus $A\subset G$ whose apartment contains $x$ and let $A'\subset G'$ be its transfer. For every root $a\in\Phi(G,A)$, the bound $\ell\ge\Phi_{G,x}(r)\ge\phi_{F_a/F}(e_a r)$ implies the hypothesis of Lemma \ref{lem:root-filtration-HH}; hence we have canonical bijections
\[
  U_a(F)_{x,0}/U_a(F)_{x,r} \xrightarrow{\;\sim\;} U_{a'}(F')_{x',0}/U_{a'}(F')_{x',r},
\tag{2}
\]
with $a'=\Phi_\ell(a)$.
The group law in a parahoric (hence in the quotients $G(F)_{x,\mathrm{b}}/G(F)_{x,r}$)
is generated by these graded pieces in equations (1) and (2) subject to the integral
Chevalley relations:
\[
[x_a(u),x_b(v)]
 \;=\; \prod_{i,j>0} x_{i a + j b}\!\big(C_{i,j}\,u^{\,i}v^{\,j}\big),
\qquad
t\,x_a(u)\,t^{-1}=x_a\!\big(a(t)\,u\big),
\]
with structure constants $C_{i,j}\in\mathbb{Z}$ (independent of the field) and
$\alpha$ the corresponding root.  The maps (1)--(2) are defined by
transport of parameters through isomorphisms of truncated valuation rings;
hence they send $u\mapsto u'$, $v\mapsto v'$, $t\mapsto t'$ and preserve
valuations/levels. Because the coefficients $C_{i,j}$ are integral, their
reductions agree on both sides, so the images of the Chevalley polynomials
computed after applying (1)--(2) coincide with the Chevalley polynomials of
the images. In particular, all defining commutator and conjugation relations
are respected in the quotients.

Thus the assignments \emph{(1)} and \emph{(2)} define maps on a generating set of
$G(F)_{x,\mathrm{b}}/G(F)_{x,r}$ and preserve the defining relations; hence they extend uniquely
to a group homomorphism
\[
  \Psi \colon G(F)_{x,\mathrm{b}}/G(F)_{x,r}\longrightarrow G'(F')_{x',\mathrm{b}}/G'(F')_{x',r}.
\]

Applying the same construction with $(F',G',x')$ in place of $(F,G,x)$ yields a
homomorphism $\Psi' \colon G'(F')_{x',\mathrm{b}}/G'(F')_{x',r}\to G(F)_{x,\mathrm{b}}/G(F)_{x,r}$.
By construction, both $\Psi'\circ\Psi$ and $\Psi\circ\Psi'$ act as the identity on
the generating subquotients \emph{(1)} and \emph{(2)}, hence are the identity on the
whole group. Therefore $\Psi$ is a group isomorphism. Moreover, since \emph{(1)} and
\emph{(2)} depend only on the congruence datum $D_\ell$, the resulting isomorphism is
functorial in $D_\ell$ and independent of auxiliary choices.

\end{proof}

\section{$\ell$-close Hecke algebra isomorphism for Bernstein blocks}
\label{sec:hecke_iso_bernstein}

\subsection{Depth-zero Bernstein blocks}
Let $G$ be a connected reductive group defined over a non-archimedean local field $F$.
Recall from the theory of Bushnell--Kutzko that the category $\mathcal{R}(G(F))$ of smooth
representations of $G(F)$ decomposes into a product of indecomposable subcategories:
\[
\mathcal{R}(G(F))=\prod_{[L,\pi]_G\in \mathfrak{B}(G)} \mathcal{R}([L,\pi]_G),
\]
where $\mathfrak{B}(G)$ denotes the Bernstein spectrum of $G$.
Denote by $\cH(\mathfrak{s})$ the Hecke algebra attached to $\mathfrak{s}=[L,\pi]_G$
via an $\mathfrak{s}$-type in the sense of Bushnell--Kutzko.

Recall that a depth-zero supercuspidal representation $\pi$ of $L(F)$ is of the form
$\cind^{L(F)}_{L(F)_x}\rho_L$  where $x\in \mathcal{B}^{\red}(L,F)$ is a vertex.
We call (any such) $x$ a vertex associated to $\pi$; it is determined by $\pi$ up to $L(F)$-conjugacy.

\medskip
\noindent\textbf{Notation (depth-zero transfer).}
Assume now that $F$ and $F'$ are $\ell$-close, fix a congruence datum $D_\ell$, and let $G'$ be the
transfer of $G$ to $F'$. If $L\subset G$ is an $F$-Levi subgroup, write $L'\subset G'$ for its transfer.
Let $\pi$ be a depth-zero supercuspidal representation of $L(F)$ with associated vertex $x$, and put
$x':=B_\ell(x)\in \mathcal{B}^{\red}(L',F')$. For $\ell$ sufficiently large, Theorem~\ref{thm:depth-comp-HH} (with $r=0$)
provides a canonical identification of the reductive quotients at $x$ and $x'$. Transporting the
depth-zero type of $\pi$ at $x$ along this identification yields a depth-zero supercuspidal representation
$\pi'\in \Irr(L'(F'))$, which we denote by
\[
\pi'\ :=\ \Kaz^{(L,x)}_\ell(\pi).
\]
(Our preference for the notation $\Kaz^{(L,x)}_\ell$ and its compatibility with the Kazhdan correspondence
are discussed in Appendix~\ref{app:kazhdan}.)

\begin{theorem}\label{thm:ell-close-depth-zero-block}
Let $F$ and $F'$ be $\ell$-close nonarchimedean local fields with residue characteristic $p$, and let
$G$ be a connected reductive group over $F$. Let $G'$ be the transfer of $G$ to $F'$ via the congruence
datum $D_\ell$ (Proposition~\ref{prop:vertex-transfer}). Let $L\subset G$ be an $F$-Levi subgroup and
$L'\subset G'$ its transfer. Let $\pi$ be a depth-zero supercuspidal representation of $L(F)$ with
associated vertex $x$, and let $\pi'=\Kaz^{(L,x)}_\ell(\pi)$.

Then there exists $\ell_\star=\ell_\star(G,L,[L,\pi]_G)$ such that for all $\ell \ge \ell_\star$ there
is a canonical isomorphism of Hecke algebras preserving the anti-involution:
\[
\cH([L,\pi]_G)\ \xrightarrow{\ \sim\ }\ \cH([L',\pi']_{G'}).
\]
\end{theorem}

\subsection*{Notations for the Proof of Theorem \ref{thm:ell-close-depth-zero-block}}

The following notations are specific to the structure of depth zero Hecke algebra.
Let  $x \in \mathcal{B}(L,F)$ correspond to a vertex in the reduced building, and $x' = \mathscr{B}_\ell(x) \in \mathcal{B}(L',F')$ its image. Fix a \(0\)-generic embedding \(\iota:\mathcal{B}(L,F)\hookrightarrow \mathcal{B}(G,F)\) in the sense of \cite{KimYu2018}.
\vspace{1em}

\noindent \textbf{Affine Root System and Hyperplanes}
\begin{enumerate}
\item 
\(\Phi_{\mathrm{aff}}(G, A_L)\): Set of affine functionals on the apartment \(\mathcal{A}(A_L, F)\) defined by
\[
\Phi_{\mathrm{aff}}(G, A_L) := \left\{ a|_{\mathcal{A}_{x}} \mid a \in \Phi_{\mathrm{aff}}(G, T, F) \setminus \Phi_{\mathrm{aff}}(L, T, F) \right\},
\]
where \(T\) is a maximal \(F\)-split torus of \(L\) such that \(x \in \mathcal{A}(L, T, F)\) and \(A_L\) denotes the maximal split torus in the center \(Z(L)\) of \(L\) and \(\mathcal{A}_{x}=x+(X_*(A_L)\otimes \mathbb{R})\).
 
  \item \( \mathfrak{H} \): Set of affine hyperplanes \( \{H_a \mid a \in \Phi_{\mathrm{aff}}(G, A_L)\} \) where \( H_a = \{u \in \mathcal{A}(A_L, F) \mid a(u) = 0\} \).
  \item \( \mathfrak{H}_{u,v} \): Subset of \( \mathfrak{H} \) for which \( u \) and \( v \) lie on the opposite sides. 
  \item \( H_s \): Hyperplane associated to a simple affine root \( s \).
\end{enumerate}

\vspace{1em}
\noindent \textbf{Reductive Quotients}
\begin{enumerate}
  \item \( \mathsf{G}_x^\circ \): Reductive group over \( \kappa \) with \( \mathsf{G}_x^\circ(\kappa) = G(F)_{x,0} / G(F)_{x,0+} \).
  \item \( \mathsf{L}_x^\circ \): Analogous quotient for the Levi subgroup \( L \).
  \item $\mathsf{S}$: Analogous quotient for the elliptic torus $S$ which corresponds to the vertex $x$.
\end{enumerate}
\vspace{1em}
\noindent \textbf{Subgroups and Representations} 
 Let $(L(F)_x,\rho_L)$ be a depth-zero type with $\rho_L$ inflating an irreducible representation $\widetilde{\rho}$ of the (possibly disconnected) reductive quotient $\mathsf{L}_x(\kappa)$ whose restriction to $\mathsf{L}_x^\circ(\kappa)$ is a cuspidal representation $\bar{\rho}$. 
\begin{enumerate}
  \item \( K_y = L(F)_{x} \cdot G(F)_{y,0} \) and \( K_{y,0+} = L(F)_{x} \cdot G(F)_{y,0+} \): compact open subgroup for \( y \in \mathcal{A}_x \).
  \item  \(\mathcal{A}_{\mathrm{gen}}\) is the set of points in the affine space \(\mathcal{A}_{x}\) that do not lie on any affine hyperplane \(H\) belonging to the set \(\mathfrak{H}\).
  \item \( \rho_y \): Irreducible representation of \( K_y \) attached to \( \rho_L \) via the isomorphism 
    \[
        K_y/K_{y,0+} \cong L(F)_{x}/L(F)_{x,0+}.
    \]
(defined for \( y \) in the generic set \( \mathcal{A}_{\mathrm{gen}} \)).

\end{enumerate}

\vspace{1em}
\noindent \textbf{\( q \)-Parameters}
\begin{enumerate}
  \item For a representation \( \pi \) of finite length \( \leq 2 \):
  \[
    q(\pi) = 
    \begin{cases} 
      1 & \text{if } \pi \text{ is irreducible} \\
      \dim \pi_1 / \dim \pi_2 & \text{if } \pi = \pi_1 \oplus \pi_2, \ \dim \pi_1 \geq \dim \pi_2
    \end{cases}
  \]
  \item \( q_s = q\left( \ind_{K_u}^{K_h} \rho_u \right) \): Parameter for an affine reflection \( s \), where:
  \begin{itemize}
    \item \( u, v \in \mathcal{A}_\mathrm{gen} \) satisfy \( \mathfrak{H}_{u,v} = \{H_s\} \) (unique separating hyperplane),
    \item \( h \in H_s \cap [u,v] \) is the unique intersection point.
  \end{itemize}
\end{enumerate}

\vspace{1em}
\noindent \textbf{Stabilizer Group and Cocycle}
For a representation \(\rho_L\) extending a cuspidal representation \(\bar{\rho}\) of the reductive quotient \(\mathsf{L}_{x}^\circ(\kappa)\).
\begin{enumerate}
  \item \( \Omega = N_{G(F)}(\rho_L)^{\heartsuit}_{x} / L(F)_{x} \), where:
  \begin{itemize}
    \item \( N_{G(F)}(\rho_L) = \{g \in G(F) \mid {}^g\rho_L \simeq \rho_L\} \),
    \item \( N_{G(F)}(\rho_L)^{\heartsuit}_{x} = N_{G(F)}(\rho_L) \cap N_G(L)(F)_{x} \),
    \item \( N_G(L)(F)_{x} \): Stabilizer of \( x \).
  \end{itemize}
  \item \( \mu \): 2-cocycle on \( \Omega \) encoding intertwining operators between \( g \)-conjugates of \( \rho_L \).
\end{enumerate}

\vspace{1em}

\subsection{Proof of Theorem \ref{thm:ell-close-depth-zero-block}} 

\begin{proof}
Let \(\pi = \text{c-ind}_{L(F)_{x}}^{L(F)} \rho_L\), where \(\rho_L\) extends a cuspidal representation \(\bar{\rho}\) of \(\mathsf{L}_{x}^\circ(\kappa)\). By Theorem \ref{thm:depth-comp-HH}, for \(\ell \geq \Phi_{G,x}(0)\), there is a canonical isomorphism:
\[
\Psi_{x,0} \colon \mathsf{G}_{x}(\kappa) \xrightarrow{\sim} \mathsf{G}_{x'}^{\prime}(\kappa'),
\]
restricting to:
\[
\Psi_{x,0}|\mathsf{L}_{x}(\kappa) \colon \mathsf{L}_{x}(\kappa) \xrightarrow{\sim} \mathsf{L}_{x'}^{\prime}(\kappa').
\]
This transfers \(\bar{\rho}\) to a cuspidal representation \(\bar{\rho}'\) of \(\mathsf{L}_{x'}^{\prime\circ}(\kappa')\), and \(\rho_L\) to \(\rho_{L'}'\) of \(L'(F')_{x'}\). Define \(\pi' = \text{c-ind}_{L'(F')_{x'}}^{L'(F')} \rho_{L'}'\), so \(\pi' = \mathrm{Kaz}_\ell(\pi)\), by the definition of the Kazhdan transfer functor (applied to $L$ at  $m=0$) together with the transport of idempotents and the Hecke algebra isomorphism at $m=0$.

By \cite[Theorem 5.3.6]{AFMO24a}, the Hecke algebra decomposes as:
\[
\mathcal{H}([L, \pi]_G) \simeq \mathbb{C}[\Omega(\rho_L), \mu] \ltimes \mathcal{H}_{\mathbb{C}}(W(\rho_L)_{\text{aff}}, q),
\]
where \(W(\rho_L)_{\text{aff}}\) is the affine Weyl group for the \(\mathcal{K}\)-relevant hyperplanes in \(\Phi(G, A_L)\) (in the sense of \cite[Definition 3.5.6]{AFMO24a} and \(q = (q_s)_{s \in \Delta_{\text{aff}}}\) is the parameter function for simple reflections for \(q_s\) as in the notations above.

Similarly, \[
\mathcal{H}([L', \pi']_{G'}) \simeq 
\mathbb{C}[\Omega(\rho_{L'}'), \mu'] \ltimes 
\mathcal{H}_{\mathbb{C}}\bigl(W(\rho_{L'}')_{\text{aff}}, q'\bigr).
\]
 Let $H\in\mathfrak H$,  $s_H$ be the reflection across \(H\) and $u\in\mathcal{A}_{\mathrm{gen}}$ is such that $H_{u,s_Hu}=\{H\}$. Let $h$ be the unique point $h\in H\cap [u,s_Hu]$. We take $0$-generic embedding mapping $x$ to $u$. For suitable  $r>0$, chosen such that $G(F)_{h,r}\subset G(F)_{x,0+}$, we have by Theorem \ref{thm:depth-comp-HH}, for $\ell >\mathrm{max}\{\Phi_{G,x}(r),\Phi_{G,h}(r)\}$,
$$
\Psi_{x,r}:\ \frac{G(F)_{x,0+}}{G(F)_{x,r}}\ \xrightarrow{\sim}\ \frac{G'(F')_{x',0+}}{G'(F')_{x',r}}
\quad\text{and}\quad
\Psi_{h,r}:\ \frac{G(F)_{h,0}}{G(F)_{h,r}}\ \xrightarrow{\sim}\ \frac{G'(F')_{h',0}}{G'(F')_{h',r}},
$$
which gives, 
\[
\frac{G(F)_{h,0}}{G(F)_{x,0+}}\xrightarrow{\sim}\frac{G'(F')_{h',0}}{G'(F')_{x',0+}}, 
\]
or equivalently, 
\[
\frac{K_h}{K_{x,0+}}\xrightarrow{\sim}\frac{K_h'}{K_{x',0+}}.
\]
Recall that \(\rho_{x'}\) was the transport of \(\rho_x\) along the isomorphism \(\frac{K_x}{K_{x,0+}}\xrightarrow{\sim}\frac{K_{x'}}{K_{x',0+}}\). We therefore have a \(\frac{K_h}{K_{x,0+}}\xrightarrow{\sim}\frac{K_{h'}}{K_{x',0+}}\) equivariant isomorphism,
\[ \tag{I}
\operatorname{ind}^{K_h}_{K_x}\rho_x \ \cong\ \operatorname{ind}^{K'_{h'}}_{K'_{x'}}\rho'_{x'}.
\]
Thus, 
\[
q_s = q(\text{ind}_{K_x}^{K_h} \rho_u) = q(\text{ind}_{K_{x'}'}^{K_{h'}'} \rho_{x'}') = q_{s'}'.
\]
The hyperplane $H\in\mathfrak H$, is $\mathcal{K}$ -relevant in the sense of \cite[Definition 3.5.6]{AFMO24a} iff \(\operatorname{ind}^{K_h}_{K_x}\rho_x \) is length two with distinct dimensional factors \cite[Lemma 4.2.4]{ohara2025}. Thus \(\mathscr{B}_{\ell}\) takes  $\mathcal{K}$ -relevant hyperplane to  $\mathcal{K}$ -relevant ones. The group \(W(\rho_L)_{\text{aff}}\) is defined as the affine Weyl group generated by reflections over  $\mathcal{K}$ -relevant hyperplanes \(\mathfrak{H}_{\mathcal{K}\text{-rel}}\) \cite[Prop. 5.3.5]{AFMO24a}. Consequently, it induces an isomorphism,
\[
W(\rho_L)_{\text{aff}} \simeq W(\rho_{L'}')_{\text{aff}}.
\]
The reduction map $G(F)_x \twoheadrightarrow \mathsf{G}_x(\kappa)$ induces a natural isomorphism
\[
  \iota_x:\quad
  \Omega(\rho_L)\ \xrightarrow{\ \sim\ }\ 
  \Stab_{N_{\mathsf G_x(\kappa)}(\mathsf L_x(\kappa))/\mathsf L_x(\kappa)}\!\big(\widetilde{\rho}\big).
\]
There is a canonical isomorphism of finite groups
\[
  \theta_x:\quad
  \dfrac{N_{\mathsf G_x(\kappa)}(\mathsf L_x(\kappa))}{\mathsf L_x(\kappa)}
  \xrightarrow{\ \sim\ }
  \dfrac{N_{\mathsf G'_{x'}(\kappa)}(\mathsf L'_{x'}(\kappa))}{\mathsf L'_{x'}(\kappa)}
\]
induced by \(\Psi_{x,0}\). Consequently,
\[
  \theta_x\Big(\Stab(\widetilde{\rho})\Big) \;=\; \Stab(\widetilde{\rho}').
\]

We obtain therefore
\[
\Phi \colon \Omega(\rho_L) \xrightarrow{\sim} \Omega(\rho_{L'}'),
\]
which is the unique isomorphism satisfying
\[
  \iota_{x'}\circ \Phi\;=\;\theta_x\circ \iota_x.
\]

The $2$-cocycle $\mu$ transfers to a cocycle $\mu'$ with
\(
  [\mu'] \;=\; \Phi_*[\mu]\in H^2\!\big(\Omega(\rho_{L'}'),\CC^\times\big).
\)
Thus, we have algebra isomorphisms:
\[
\mathbb{C}[\Omega(\rho_L), \mu] \simeq \mathbb{C}[\Omega(\rho_{L'}'), \mu'], \quad \mathcal{H}_{\mathbb{C}}(W(\rho_L)_{\text{aff}}, q) \simeq \mathcal{H}_{\mathbb{C}}(W(\rho_{L'}')_{\text{aff}}, q'),
\]
so \(\mathcal{H}([L, \pi]_G) \simeq \mathcal{H}([L', \pi']_{G'})\).  

\noindent

\end{proof}

\subsection{Arbitrary depth Bernstein blocks}
Assume the residue characteristic \(p\) of $F$ satisfies:
\begin{itemize}
    \item \(p \neq 2\) and
    \item \(p \) is not a torsion prime for the dual root datum. 
\end{itemize}
Yu's construction \cite{Yu2001} builds supercuspidal representations $\pi$ from a datum $\Sigma = (\pi_0, (G^i)_{i=0}^d, (\phi_i)_{i=0}^d)$, where:
\begin{enumerate}
\item $G^0 \subset \cdots \subset G^d = G$ is a tamely ramified twisted Levi sequence.
\item $\phi_i$ are characters of $G^i(F)$ satisfying genericity conditions (\cite[\S 8]{Yu2001}).
\item  $\pi_0$ is a depth-zero supercuspidal representation of $G^0(F)$.
\end{enumerate}
The resulting representation is $\pi(\Sigma) = \mathrm{c\text{-}ind}_{K^d}^{G(F)} \tau$, where $K^d$ is a compact mod center open subgroup and $(\tau,V)  $ is an irreducible representation of $K^d$. The Deligne transfer of $\Sigma$ to $F'$ yields a datum $\Sigma'$ for $G'$, giving a representation $\pi(\Sigma')$ of $G'(F')$.
\begin{remark}
    One should be able to drop the residue characteristic hypothesis in Theorem \ref{thm:ell-close-kaz-yu-compatability} by invoking the recent work of Fintzen-Schwein \cite{FS2025}.
\end{remark}
\begin{theorem}\cite[Theorem 4.4.1]{AFMO24b}\label{thm:AFMO24b}
Let $\Sigma$ be a Kim-Yu-datum producing an $[L,\pi]_G$-type and let $(K^i,\rho_i)$ for $0\leq i\leq d$ be the tower of types constructed out of it, $(K^0,\rho_{-1})$ being the  $[L^0,\pi_{-1}]_{G^0}$ -type where $\pi_{-1}$ is depth-zero . Then,
\[
\mathcal{H}([L,\pi]_G)\cong \mathcal{H}([L^0,\pi_{-1}]_{G^0}),  
\]
where $(K^d,\rho_d)=:(K,\rho)$ is the $[L,\pi]_G$-type.
\end{theorem}
\begin{remark}\label{rem:AFMO-filtrations}
The reduction-to-depth-zero results of Adler--Fintzen--Mishra--Ohara \cite{AFMO24a,AFMO24b}
are formulated using Moy--Prasad filtration subgroups. In the present paper we work instead with
Yu's minimal congruence filtration (Definition~\ref{subsub:mincongdef}) and the associated subgroups $G(F)_{x,r}$.

We use \cite[Theorem~4.4.1]{AFMO24a} and \cite[Theorem~5.3.6]{AFMO24b} only as statements about the
Hecke algebras attached to the (Bushnell--Kutzko) types produced from Yu data; in particular, our
arguments do not require any identification between Moy--Prasad and minimal congruence filtration
subgroups.

Note also that even when $G$ is tamely ramified, a maximally unramified elliptic maximal torus
$S$ attached to a vertex $x$ can be wildly ramified; this is precisely why our depth-transfer
function $\Phi_{G,x}$ includes the terms $\Phi_S$ as $S$ ranges over such tori.
Finally, when the relevant tori are tamely ramified, Yu's filtration on tori agrees with the
Moy--Prasad filtration, so the depth bounds in this section coincide with the
customary ones.
\end{remark}

\begin{theorem}\label{thm:ell-close-arbitrary-blocks}
Let $F$ and $F'$ be $\ell$-close non-archimedean local fields with odd residue
characteristic $p$, and let $G$ be a connected reductive group over $F$ with
transfer $G'$ over $F'$ via a congruence datum $D_\ell$, where $p$ is not a torsion
prime for $G$. Let $L\subset G$ be an $F$-Levi subgroup, and $L'\subset G'$ its transfer.
Fix an integer $m\ge 0$. Then there exists $\ell^\ast=\ell^\ast(G,[L,\pi]_G,m)\gg 0$
such that for every $\ell\ge \ell^\ast$ the following holds.

Let $\pi$ be an irreducible supercuspidal representation of $L(F)$ of depth $\le m$,
and fix a Yu datum $\Sigma$ for $L$ such that $\pi \cong \pi(\Sigma)$.
Let $\Sigma'$ be the (Deligne) transfer of $\Sigma$ to a Yu datum for $L'$ over $F'$
(with respect to $D_\ell$), and set $\pi' := \pi(\Sigma')\in \Irr(L'(F'))$.
Then there is a canonical isomorphism of Hecke algebras
\[
\kappa_{\ell,[L,\pi]_G}:\ \mathcal{H}([L,\pi]_G)\ \xrightarrow{\ \sim\ }\ \mathcal{H}([L',\pi']_{G'})\, .
\]
Moreover, this isomorphism preserves the presentation of these Hecke algebras as
obtained in \cite[\S 5.3.6]{AFMO24b}.

\end{theorem}
\begin{proof}
    The proof follows from Theorem \ref{thm:AFMO24b} and Theorem \ref{thm:ell-close-depth-zero-block}. 
\end{proof}
\begin{definition}\label{def:kappa-sharp}
Define the isomorphism induced by \(\kappa_{\ell,[L,\pi]_G}\) on the Bernstein blocks to be,
\[
\kappa_{\ell,[L,\pi]_G}^{\sharp}
:= (M'_{[L',\pi']_{G'}})^{-1}\circ
(\kappa_{\ell,[L,\pi]_G}^{-1})_{*}\circ
M_{[L,\pi]_G}:
\mathcal R([L,\pi]_G)\xrightarrow{\sim}\mathcal R([L',\pi']_{G'}) ,
\]
where \(M_{[L,\pi]_G}\) is as in Section \ref{sec:def-Kaz} and \((\kappa_{\ell,[L,\pi]_G})_{*}\) denotes the isomorphism induced on module categories. 
\end{definition}
\section{Local Langlands correspondence in positive characteristic via close--field transfer}

\subsection{Internal structure of $L$-packets for inner forms}\label{sec:l-packet-internal-struct}

Let $F$ be a non-archimedean local field with Weil group $W_F$, and let ${}^{L}G=\widehat G\rtimes W_F$ be the $L$-group of a connected reductive $F$-group $G$ with quasi-split inner form $G^{\ast}$. Set \(Z=Z(G_\mathrm{der})\) and \(\overline{G}=G/Z\). For a tempered Langlands parameter $\varphi:W_F'\to{}^{L}G$, where \(W_F'\) is the Weil-Deligne group, write $S_\varphi:=Z_{\widehat G}(\varphi)$ for the centralizer of $\varphi(W_F')$ in $\widehat G$ and set
\[
  S_\varphi^{+}\text{ for the preimage of }S_\varphi \text{ in }\widehat{\overline{G}}.  
\]

\paragraph{Kaletha's parametrization via rigid inner twists}
A \emph{rigid inner twist} \cite[\S 5]{Kal16} of $G^{\ast}$ is a triple $(G',\iota,z)$ consisting of an inner isomorphism $\iota:G'_{\overline F}\xrightarrow{\sim}G^{\ast}_{\overline F}$ and a cohomology class $z\in H^{1}(u\!\to\!W,\,Z\!\to\!G)$ in Kaletha's enlarged set. For $p$-adic $F$ there is a canonical pairing \cite[Cor. 5.4]{Kal16}.
\[
  \langle\ ,\ \rangle:\ H^{1}(u\!\to\!W,\,Z\!\to\!G)\ \times\ \pi_0\!\bigl(Z(\widehat {\overline{G}})^{+}\bigr)\ \longrightarrow\ \mathbb{C}^{\times},
\]
which associates to $z$ a character $\omega_z$ of $\pi_0\!\bigl(Z(\widehat {\overline{G}})^{+}\bigr)$, where $Z(\widehat {\overline{G}})^{+}$ is the preimage of $Z(\widehat{G})^{W_F}$ in $Z(\widehat {\overline{G}})$. For tempered $\varphi$, the packet on $(G',\iota,z)$ is conjectured to be internally labeled by
\[
  \Pi_\varphi(G',\iota,z)\ \longleftrightarrow\ 
  \Bigl\{\rho\in\mathrm{Irr}\bigl(\pi_0(S_\varphi^{+})\bigr):\ \rho\!\restriction_{\pi_0\!\bigl(Z(\widehat {\overline{G}})^{+}\bigr)}=\omega_z\Bigr\}.
\tag{A}\]

\begin{lemma}
There is a natural embedding, 
\[
Z(\widehat G_{\mathrm{sc}})^{W_F}\ \hookrightarrow\ \pi_0\!\bigl(Z(\widehat{\overline G})^{+}\bigr).
\]
\end{lemma}
\begin{proof}
Let $q:\widehat{\overline G}\to \widehat G$ denote the dual isogeny. Its kernel is $K:=Z(\widehat G_{\mathrm{sc}})$. We have an exact sequence

\[
1\ \longrightarrow\ K\ \longrightarrow\ Z(\widehat{\overline G})^{+}\ \xrightarrow{\,q\,}\ Z(\widehat G)^{W_F}\ \longrightarrow\ 1.
\]
The inclusion $K\hookrightarrow Z(\widehat{\overline G})^{+}$ induces a map on components
$$
\iota:\ K^{W_F}\ \longrightarrow\ \pi_0\!\bigl(Z(\widehat{\overline G})^{+}\bigr).
$$
Now, every torsion character of $Z(\widehat{\overline G})^{+}$ vanishes on the identity component $Z(\widehat{\overline G})^{{+}\circ}$. Since every character of $K$ is the restriction of a torsion character of $H$, they each vanish on \(K\cap Z(\widehat{\overline G})^{{+}\circ}\). Therefore $K^{W_F}\cap H^\circ=\{1\}$. This proves, 
\[
\iota:\ Z(\widehat G_{\mathrm{sc}})^{W_F}\ \longrightarrow\ \pi_0\!\bigl(Z(\widehat{\overline G})^{+}\bigr),
\]
is an injection. 
\end{proof}
\paragraph{Formulation of the parametrisation in arbitrary characteristic}
Let $G$ be an inner form of $G^{\ast}$ with class in $H^{1}(F,G_{\mathrm{ad}})$. By Kottwitz (\cite[Proposition 6.4]{Kott84}; \cite[Theorem 2.1]{Thang11} for positive characteristic), there is a natural group isomorphism,
\[
  H^{1}(F,G_{\mathrm{ad}})\cong \mathrm{Irr}(\,Z(\widehat{G}_{\mathrm{sc}})^{W_F});
\]
Via this isomorphism, let 
\[
  \zeta_{G}\ \in\ \mathrm{Irr}\big(Z(\widehat G_{\mathrm{sc}})^{W_F}\big)
\]
represent the inner class $G$. Given a Langlands parameter \(\varphi\) of \(G\), following \cite{ABPS18}, define the $G$-relevant enhancements for $\varphi$ by the central-character condition,
\[
  \mathrm{Irr}_{\,\zeta_G}\bigl(S_\varphi^{+}\bigr):=\Bigl\{\rho\in\mathrm{Irr}\bigl(\pi_0(S_\varphi^{+})\bigr):\ 
  \rho\!\restriction_{Z(\widehat G_{\mathrm{sc}})^{W_F}}=\zeta_G\Bigr\}.
\]
We have a decomposition,
\[
\Irr_{\zeta_G}\!\bigl(\pi_0(S_\varphi^{+})\bigr)
\;=\;
\bigsqcup_{\ \omega\in\Ext(\zeta_G)}\ 
\underbrace{\Bigl\{\rho\in\Irr(\pi_0(S_\varphi^{+}))\;:\;
\rho\!\restriction_{\pi_0(Z(\widehat{\overline G})^{+})}=\omega\Bigr\}}_{=:~\Irr_\omega(\pi_0(S_\varphi^{+}))},
\]
where
\[
\Ext(\zeta_G)\;:=\;\Bigl\{\ \omega\in\Irr\!\bigl(\pi_0(Z(\widehat{\overline G})^{+})\bigr)\ :\
\omega\!\restriction_{Z(\widehat G_{\mathrm{sc}})^{W_F}}=\zeta_G\ \Bigr\}.
\]
Set

$$
C:=\pi_0\!\bigl(Z(\widehat{\overline G})^{+}\bigr),\qquad
C_{\mathrm{sc}}:=Z(\widehat G_{\mathrm{sc}})^{W_F},\qquad
Q:=C/C_{\mathrm{sc}},
$$
Then $\Ext(\zeta_G)$ is a nonempty finite torsor under $\widehat Q=\Hom(Q,\C^\times)$.

\medskip
In the absence of a fully developed \emph{local} theory of rigid inner twists over
local function fields in the form needed here, we record the following conjectural
description of the internal parametrization of $L$-packets.
(For a gerbe-theoretic framework organizing rigid inner forms over \emph{global}
function fields and relating their localizations, see Dillery~\cite{Dillery}.)

\begin{conjecture}
Let \(G\) be a connected reductive group defined over a non-archimedean local field \(F\) of arbitrary characteristic. Fix a Whittaker datum for the quasi-split inner form $G^{\ast}$. Then there exists a canonical surjection
\[
\operatorname{pr}:\ \Irr_{\zeta_G}\!\bigl(S_\varphi^{+}\bigr)\ \twoheadrightarrow\ \Pi_{\varphi}(G),
\]
whose fibers form a torsor under \(\widehat{Q}\). 
\end{conjecture}
\noindent\textit{Remark.}
The conjecture above matches the expected form of the internal parametrization via
rigid inner twists (compare \cite[Conj.~1.1]{Dillery} and the discussion preceding it).

\medskip

\subsection{LLC for Regular supercuspidals in positive characteristic}
\begin{theorem}\label{thm:llc-for-regular}
Let \(F\) be a non-archimedean local field of positive characteristic with residue characteristic \(p\).
Let \(G\) be a connected reductive \(F\)-group such that  
\begin{enumerate}
\item \(p\) is odd;
\item \(p\) is not bad for \(G\);
\item \(p\nmid \lvert\pi_0\bigl(Z(G)\bigr)\rvert\).
\end{enumerate}
Then there exists a local Langlands correspondence for regular supercuspidal representations of \(G(F)\). Specifically,
\begin{enumerate}[label=\textup{(\alph*)}]
\item The set of $\widehat G$--conjugacy classes of regular supercuspidal parameters $\varphi:W_F\to {^LG}$ is in bijection with the set of isomorphism classes of regular supercuspidal $L$--packet data $(S,\hat\jmath,\chi,\theta)$.\vspace{2pt}

\item For each such $\varphi$, there exists an  \(L\)-packet  \(\Pi_\varphi\) consisting of regular supercuspidal representations of \(G(F)\). Moreover, there is a surjection
\[
\operatorname{pr}:\ \Irr_{\zeta_G}\!\bigl(S_\varphi^{+}\bigr)\ \twoheadrightarrow\ \Pi_{\varphi}(G),
\]
determined by the choice of a Whittaker datum of the quasi-split inner form of \(G\). The fibers of \(\operatorname{pr}\)  form a torsor under \(\widehat{Q}\). 
\medskip
\item If $G$ is quasi--split and \(\mathbf w\) is a Whittaker datum for $G$, there exists an \(\mathbf w\)-generic constituent in \(\Pi_\varphi(G)\).

\end{enumerate}

\end{theorem}
\begin{proof}
Throughout the proof we use the prime ``\('\,\)'' to denote all objects over the
\(\ell\)-close characteristic-\(0\) field. We let \(G^*\) denote the quasi-split inner form of \(G\).

Let $(S,\theta)$ be a tame regular elliptic pair for $G(F)$ \cite[Definition 3.7.5]{KalethaRegSC}, let $x=x_S\in \mathcal B^{\red}(G,F)$ be the associated vertex and \(\pi(S,\theta)\) the associated regular supercuspidal representation. Fix $m\ge \depth(\pi)$, $\ell_{\dagger}\le \ell$ as in Theorem ~\ref{thm:ell-close-kaz-yu-compatability}, and choose a characteristic \(0\) \(\ell\)-close field \(F'\). Choose a  Howe factorization in the sense of \cite[Definition 3.6.2]{KalethaRegSC} of $(S,\theta)$, yielding a Yu datum
\[
  \Sigma=\bigl((G^i)_{i=0}^d,\; (x, \vec r),\; (\phi_i)_{i=-1}^d\bigr).
\]
Choose a characteristic 0 field \(F'\) which is \(\ell\) close to \(F\). Let $\Sigma'$ denote the $\ell$-transfer of $\Sigma$, and let $(S',\theta')$ be the toral transfer (Theorem \ref{thm:congruent-isom}) of $(S,\theta)$. Then, up to $G'(F')$-conjugacy and refactorization (in the sense of \cite{HM08}, also see \cite[Lemma 3.6.6]{KalethaRegSC}), $\Sigma'$ is a  Howe factorization of $(S',\theta')$, and the  representations \(\pi(\Sigma')\) and \(\pi(S',\theta')\) are isomorphic. Also by \ref{thm:ell-close-kaz-yu-compatability}, \(\Kaz_{\ell}^{(G,x;m)}(\pi)=\pi(\Sigma')\).
We now recall the notion of an \(L\)-packet datum. It consists of a tuple \((S,\widehat{j},\chi,\theta)  \) consisting of
\begin{enumerate}
  \item $S$ an $F$-torus of dimension equal to the absolute rank of $G$, splitting over a tame extension of $F$; 
  \item an embedding $\widehat{\jmath}:\widehat{S}\to\widehat{G}$ of complex reductive groups whose $\widehat{G}$-conjugacy class is $\Gamma$-stable, where \(\Gamma \;:=\; \Gal(F^{\mathrm{sep}}/F)\); 
  \item $\chi$ a set of \emph{minimally ramified} $\chi$-data for $R(S,G)$ (in the sense of \cite[Definition 4.6.1]{KalethaRegSC}).
  \item $\theta:S(F)\to\C^\times$ a character. 
\end{enumerate}
In addition, the following conditions are required:
\begin{itemize}
  \item the $\chi$-data are $\Omega(S,G^{0})(F)$-invariant, where \(\Omega(S,G^{0}) \;:=\; N_{G^{0}}(S)/S \); 
  \item $S/Z(G)$ is anisotropic; 
  \item $(S,\theta)$ is a tame extra regular elliptic pair in the sense of \cite[Definition 3.7.5]{KalethaRegSC}.
\end{itemize}
From a regular supercuspidal datum $(S',\hat\jmath',\chi',\theta')$, one obtains a parameter 
\(\varphi'=\!{^{L}\!\hat\jmath'}\circ \varphi'_{S',\theta',\chi'}\) and the regularity hypothesis ensures that its $\widehat G$--centralizer is a torus:
\[
S_{\varphi'}\;=\;Z_{\widehat G}(\varphi')\;=\;\hat\jmath'(\widehat S').
\]
We then have
\[
  \Pi_{\varphi'}(G')\ \twoheadleftarrow\ \mathrm{Irr}_{\,\zeta_{G'}}\bigl(S_{\varphi'}^{+}\bigr).
\]
as explained in Section \ref{sec:l-packet-internal-struct}. This surjection is determined by the choice of a Whittaker datum \(\mathbf{w'}=(B',\psi')\). or equivalently, \(\Pi_{\varphi'}(G'^*)\) contains a unique \(\mathbf{w}\)-generic constituent \(\pi'_{\mathrm{gen}}\). For toral supercuspidal \(L\)-packets, the last fact was already established in original \cite[Lemma 6.2.2]{KalethaRegSC}. For regular supercuspidal $L$-packets, the same argument extends verbatim once \cite[Lemma 6.2.1]{KalethaRegSC} is replaced by \cite[Proposition 4.4.1]{FKS23}.

Recall that a set \(\chi=\{\chi_\alpha\}_{\alpha\in R(S,G)}\) is minimally ramified if \(\chi_\alpha=1\) for asymmetric roots, \(\chi_\alpha\) is unramified for unramified symmetric roots, and \(\chi_\alpha\) is tamely ramified for ramified symmetric roots; any two such choices can differ only at ramified symmetric roots by the unramified sign character of \(F_\alpha^\times\). Here $R=R(S,G)$ is the absolute root system, $F_{\alpha}$ (resp.\ $F_{\pm\alpha}$) denote the usual root (resp.\ folded root) field and \(\chi_\alpha: F_{\alpha}^{\times}\to \C^{\times}\) are characters satisfying Langlands--Shelstad axioms. 

Now assume \((S,\theta)\) is also extra regular. By \cite[Proposition 5.2.7]{KalethaRegSC}, the isomorphism class of \(\pi(S',\theta')\) associates to the \(\widehat{G}\)-conjugacy class of an \(L\)-packet datum \((S',\widehat{j'},\chi',\theta')  \). We will now describe a canonical transfer of the datum $(S',\widehat{j}',\chi',\theta')$ to a datum over $F$:
\[
(S,\widehat{j},\chi,\theta).
\]
\emph{Minimally ramified $\chi$--data.} Identify roots $R(S,G)\xrightarrow{\sim} R(S',G')$, $\alpha\leftrightarrow \alpha'$, together with the associated root/folded fields $F_{\alpha}\leftrightarrow F'_{\alpha'}$, $F_{\pm\alpha}\leftrightarrow F'_{\pm\alpha'}$, preserving ramification type. For each $\alpha$:
  \begin{itemize}
    \item if $\alpha$ is asymmetric, set $\chi_{\alpha}=1$;
    \item if $\alpha$ is symmetric unramified, take $\chi_{\alpha}$ the unramified quadratic;
    \item if $\alpha$ is symmetric ramified, transport the \emph{tame} character $\chi'_{\alpha'}$ via the canonical bijection of tame character groups
    \[
      d_{\alpha}^{-1}:\ \Hom_{\mathrm{tame}}(F'_{\alpha'}{}^\times,\C^\times)\xrightarrow{\ \sim\ }\Hom_{\mathrm{tame}}(F_{\alpha}^\times,\C^\times),
    \]
characterized via local class field theory using the isomorphism of tame Weil quotients. This yields $\chi_{\alpha}$ with the same restriction to $F_{\pm\alpha}^\times$ as prescribed by the Langlands--Shelstad axioms. 
\end{itemize}
Then by construction, $\chi=\{\chi_{\alpha}\}$ is again minimally ramified. 

\emph{Admissible embedding.} Fix pinned splittings on the dual groups to identify the Galois actions. The $\Gamma'$--stability of the $\widehat{G}'$--conjugacy class of $\widehat{j}'$ descends along Deligne's identification $W_F/I_F^{\ell+}\!\cong W_{F'}/I_{F'}^{\ell+}$ to a $\Gamma$--stable $\widehat{G}$--conjugacy class, yielding an admissible embedding
  \[
    \widehat{j}:\widehat{S}\longrightarrow \widehat{G}.
  \]
Equivalently on cocharacters, the $\Gamma$--equivariant inclusion $X_\ast(\widehat{S}')\hookrightarrow X_\ast(\widehat{T}')$ transports to a $\Gamma$--equivariant inclusion $X_\ast(\widehat{S})\hookrightarrow X_\ast(\widehat{T})$, which integrates to $\widehat{j}$.

We have thus obtained a canonical transfer of the datum $(S',\widehat{j}',\chi',\theta')$ to a datum  \((S,\widehat{j},\chi,\theta)\) over $F$.

This transfer produces identifications 
\[
\hat\jmath'(\widehat S')\;=\;\hat\jmath(\widehat S)\  \Leftrightarrow S_{\varphi}=S_{\varphi'},
\]
and therefore, 
\[
S_{\varphi}^{+}=S_{\varphi'}^{+}\subset\widehat{\overline G}.
\]
Our close-field setup identifies the $W$--action on the based root datum, hence canonically identifies 
\(Z(\widehat G_{\mathrm{sc}})^{W_F}\cong Z(\widehat G_{\mathrm{sc}})^{W_{F'}}\) inside the same $\widehat{\overline G}$. Therefore the restriction condition give a canonical bijection on
\(\Irr_{\zeta_{G}}(S_{\varphi'}^{+})\to \Irr_{\zeta_{G}}(S_{\varphi}^{+})\).

By Lemma \ref{lem:kaz-preserves-genericity}, there is a canonically transferred Whittaker datum \(\mathbf{w}\) for \(G^*\) corresponding to \(\mathbf{w'}\) and the representation \(\pi_{\mathrm{gen}}:=\Kaz_{\ell}(\pi'_{\mathrm{gen}})\) is \(\mathbf{w}\)-generic. Since 
\[
\operatorname{pr}^\prime:\ \mathrm{Irr}_{\zeta_{G'}}\!\bigl(S_{\varphi'}^{+}\bigr)\ \twoheadrightarrow\ \Pi_{\varphi'}(G'),
\]
is determined by \(\mathbf{w'}\), it follows that there is a  surjection, 
\[
\operatorname{pr}:\ \Irr_{\zeta_G}\!\bigl(S_\varphi^{+}\bigr)\ \twoheadrightarrow\ \Pi_{\varphi}(G),
\]
determined by the choice of \(\mathbf{w}\). 
\end{proof}

\subsection{Recipe for the full LLC in positive characteristic via close--field transfer}
\label{sec:LLC_recipe_poschar}

Let \(G\) be a connected reductive group defined over a \emph{characteristic-\(p\)} non-archimedean
local field \(F\).
Fix an integer \(m\ge 0\).
Choose an integer \(\ell\gg 0\) with
\(\ell\ge \ell_{*}\) as in Theorem \ref{thm:ell-close-arbitrary-blocks} 
and pick an \(\ell\)-close characteristic-\(0\) field
\(F'\) together with a level--\(\ell\) congruence datum
\(\mathcal D_{\ell}\colon \OO_F/\mathfrak p_F^{\ell}\;\widetilde{\longrightarrow}\;\OO_{F'}/\mathfrak p_{F'}^{\ell}\).
Via~\(\mathcal D_{\ell}\) we obtain the transferred group \(G'/F'\), and  an equivalence
\[
\kappa_{\ell,[L,\pi]_G}^{\sharp}
:\mathcal R([L,\pi]_G)\xrightarrow{\sim}\mathcal R([L',\pi']_{G'}) .
\]
We assume that LLC $\sigma' \mapsto \varphi'(\sigma')$ is known in characteristic $0$. 

\begin{definition}\label{def:parameter_poschar}
For an irreducible smooth representation
\(\sigma\in\mathcal R([L,\pi]_G)\)
set
\[
  \sigma':=\kappa_{\ell,[L,\pi]_G}^{\sharp}(\sigma)\;\in\;\mathcal R([L',\pi']_{G'}),
  \qquad
  \varphi(\sigma):=\varphi(\sigma'),
\]
where \(\varphi(\pi')\colon W_{F'}\to{}^LG\) is the Langlands parameter
provided by the characteristic-\(0\) LLC for \(G'/F'\). 
Via Deligne's isomorphism
\(W_F/I_F^{\ell}\;\widetilde{\longrightarrow}\;W_{F'}/I_{F'}^{\ell}\)
(this is possible since \(\ell\ge\ell_*\)),
we regard \(\varphi(\sigma)\) as a parameter defined over \(W_F\).
\end{definition}

\begin{theorem}\label{thm:LLC_poschar}
The assignment \(\sigma\longmapsto\varphi(\sigma)\) of
Definition~\ref{def:parameter_poschar} satisfies:

\begin{enumerate}
  \item 
        \(\varphi(\sigma)\) is independent of the auxiliary choices
        of the \(\ell\)-close field \(F'\), the congruence datum
        \(\mathcal D_{\ell}\), and the integer~\(\ell\)
        (provided \(\ell\ge \ell_*)\)).
\end{enumerate}
\end{theorem}

\begin{proof}

\noindent
\textit{(i)~Independence.}
For two choices \(\ell_1,\ell_2\ge\ell_*\) and corresponding
characteristic-\(0\) fields \(F'_1,F'_2\)
there exists \(\ell_3\ge\ell_1,\ell_2\) such that all data embed into the
level-$\ell_3$ situation. Since $\kappa_{\ell,[L,\pi]_G}^{\sharp}$ is canonical 
(Theorem \ref{thm:ell-close-arbitrary-blocks} and Definition \ref{def:kappa-sharp}), 
it follows that the diagram

\[
  \mathcal R([L,\pi]_{G})\;
  \xrightarrow{\;\kappa_{\ell,[L,\pi]_G}^{\sharp}\;}
  \mathcal R([L_1',\pi_1']_{G_1'})
  \longrightarrow
  \mathcal R([L_3',\pi_3']_{G_3'})
\]
commutes with the analogous diagram for~\(\ell_2\).
Since the characteristic-\(0\) LLC is uniquely characterized
by functoriality with respect to these maps, the resulting
parameter~\(\varphi(\sigma)\) is independent of all choices.

\smallskip
\noindent

\end{proof}

\medskip
\noindent\textbf{LLC for local function field.}
For an arbitrary irreducible smooth representation \(\sigma\in \mathcal R([L,\pi]_G)\),
choose \(m\ge\depth(\sigma)\) and apply
Definition~\ref{def:parameter_poschar} with any
\(\ell\ge\ell_*(G,[L,\pi]_G,m)\); by part~\textit{(i)} of Theorem \ref{thm:LLC_poschar},
the resulting parameter is independent of all choices.
Taking limits over \(m\) yields a
Langlands correspondence for all irreducible representations of
\(G(F)\) that preserve depth.
\smallskip
\begin{remark}
Gan--Harris--Sawin--Beuzart--Plessis \cite{GHSBP}
formulate a conjectural compatibility between the Fargues--Scholze and
Genestier--Lafforgue parametrizations via Deligne--Kazhdan close fields
(\cite[Conjecture~11.7]{GHSBP}).
The depth-preserving close-field transfer constructed here may be viewed as
local evidence toward such compatibilities, once depth is measured using the
revised notion introduced in \S \ref{sec:depth-theorem}, which remains well behaved for wildly
ramified tori.
\end{remark}

\appendix

\section{Cartan decomposition for maximal parahorics}
\label{app:Cartan}

In this section, we extend Cartan decomposition from the known case of special maximal parahorics to arbitrary maximal parahorics. We first recall some basics. 
\subsection{Extended affine Weyl group and the Kottwitz homomorphism}

Let $F$ be a non-archimedean local field with residue field $\kappa$, of cardinality \(q\) and characteristic \(p\). Let $L$ be the completion of the maximal unramified extension of $F$ and $\sigma$  the Fr\"obenius automorphism of $L$.

For a connected reductive group $G$ over $F$ and a maximally unramified elliptic maximal $F$-torus $S \subset G$:
\begin{itemize}
    \item The \textit{affine Weyl group} $W_{\text{aff}, S}$, associated to $S$, is the semi-direct product $Q_S^{\vee} \rtimes W(G, S)$, where $Q_S^{\vee}$ is the coroot lattice of $S$, and $W(G, S) = N_G(S)(F) / S(F)$ is the Weyl group, with $N_G(S)(F)$ the normalizer of $S$ in $G(F)$.
    \item The \textit{extended affine Weyl group} $\widetilde{W} = N_G(S)(L) / S(L)_0$, where $S(L)_0$ is the Iwahori subgroup of $S(L)$. 
    \item The extended affine Weyl group $\widetilde{W}$ admits a $\sigma$-equivariant decomposition:
    \[
    \widetilde{W} = W_{\text{aff}, S} \rtimes \Omega_S,
    \]
    where $\Omega_S$ is the subgroup of length-zero elements in $\widetilde{W}$, which stabilizes a fixed alcove in the apartment of the Bruhat-Tits building of $G(L)$. The action of $\sigma$ preserves this decomposition, acting on both $W_{\text{aff}, S}$ and $\Omega_S$.
\end{itemize}

\paragraph{Kottwitz Homomorphism}

As before, let $G$ be a connected reductive group over a non-archimedean local field $F$, with $L$ the completion of the maximal unramified extension of $F$, and $I = \mathrm{Gal}(\overline{F}/L)$ the inertia subgroup. The \emph{Kottwitz homomorphism} \cite[\S 7]{Kot97} is a surjective, functorial group homomorphism
\[
w_G: G(L) \to X^*(Z(\widehat{G})^I),
\]
where $\widehat{G}$ is the dual group of $G$, and $Z(\widehat{G})^I$ denotes the $I$-invariants of its center. For a torus $T$, it simplifies to
\[
w_T: T(L) \to X_*(T)_I,
\]
where $X_*(T)_I$ is the group of $I$-coinvariants in the cocharacter lattice $X_*(T)$, defined via the valuation map.

When $G$ is defined over $F$, with $\Gamma = \mathrm{Gal}(\overline{F}/F)$, the homomorphism restricts to $G(F)$, inducing a surjection
\[
w_G: G(F) \to X^*(Z(\widehat{G}))_I^\sigma.
\]
Let ${}^{\circ}G(F)$ denote the kernel of this map. For a vertex $x$ in the Bruhat-Tits building $\mathscr{B}(G, F)$, let $G(F)_x$ denote its stabilizer. Then, the intersection
\[
G(F)_x \cap {}^{\circ}G(F)
\]
is equal to the parahoric subgroup $G(F)_{x,0}$ at $x$.

\subsection{Cartan decomposition}

\begin{lemma}\label{lem:no-sigma-in-wa}
Let $S\subset G$ be elliptic. Let $W_{\mathrm{aff},S}$ be the affine Weyl group that sits in the semi-direct product $\widetilde{W}=W_{\mathrm{aff},S}\rtimes \Omega_S$. Then
\[
   W_{\mathrm{aff},S}^{\sigma}=1.
\]
\end{lemma}

\begin{proof}
Write $Q^\vee_S\subset X_*(S)_I$ for the coroot lattice. We may realize $W_{\mathrm{aff},S}$ as
\[
   W_{\mathrm{aff},S}=Q^\vee_S\rtimes W(G,S),
\]
where an element $t_\nu w$ acts on the apartment $X_*(S)_I\otimes\mathbb{R}$ by a translation by $\nu\in Q^\vee_S$ followed by the finite Weyl element $w\in W(G,S)$. Let $x=t_\nu w\in W_{\mathrm{aff},S}$ be $\sigma$-fixed.

Applying $\sigma$ to $x$, we obtain
\[
  x=t_{\sigma(\nu)}\,\sigma(w)=t_\nu w.
\]
Comparing translation parts gives \(\sigma(\nu)=\nu\).
Let \(G_{\mathrm{ad}}:=G/Z(G)\) and let \(S_{\mathrm{ad}}\subset G_{\mathrm{ad}}\) be the image of \(S\). Since \(S\) is elliptic, the torus \(S_{\mathrm{ad}}\) is anisotropic over \(F\), hence it contains no nontrivial \(F\)-split subtorus. Equivalently,
\[
X_*(S_{\mathrm{ad}})_I^{\sigma}=0.
\]
The coroot lattice \(Q_S^\vee\) maps injectively into \(X_*(S_{\mathrm{ad}})_I\) (coroots vanish on the center), so the equality \(\sigma(\nu)=\nu\) forces the image of \(\nu\) in \(X_*(S_{\mathrm{ad}})_I^{\sigma}\) to be zero, hence \(\nu=0\).

Thus, \(x = w \in W(G, S)\) with \(w = \sigma(w)\). Fix a $\sigma$-stable positive system of (relative) roots (this exists because over $L$ we may choose a $\sigma$-stable pinned datum). If $w\neq 1$, choose a positive root $\alpha$ with $w(\alpha)<0$.
Since $\sigma(w)=w$ and $\sigma$ preserves positivity, for every $i$ we have
\[
w\bigl(\sigma^i(\alpha)\bigr)=\sigma^i\bigl(w(\alpha)\bigr)<0.
\]
Thus, the entire \(\sigma\)-orbit of \(\alpha\) is mapped to negative roots. The sum of the coroots in this orbit,
\[
\beta = \sum_{i=0}^{k-1} \bigl(\sigma^i(\alpha)\bigr)^\vee,
\]
is \(\sigma\)-fixed, non-zero since it's a sum of positive co-roots, and lies in \(Q_S^\vee\). This contradicts the ellipticity of \(S\), as \(\beta\) spans a non-trivial \(F\)-split torus in the derived group, centralized by \(S\). Hence, \(w = 1\).
Combining the two steps, we conclude $x=1$, so $W_{\mathrm{aff},S}^{\sigma}=1$.
\end{proof}
\begin{proposition}\label{prop:sigma-Cartan}
Maintain the notations of Lemma \ref{lem:no-sigma-in-wa}.  
Let \(S\subset G\) be a maximally un\-ramified elliptic maximal \(F\)-torus and let
\(
  x_S\in\mathscr B^{\mathrm{red}}(G,F)
\)
be the associated vertex. Set
\[
K_S \;:=\; G(F)_{x_S,0}, 
\qquad
\mathcal K \;:=\; \mathcal G_{x_S}^{\circ}(\mathcal O_L)\subset G(L),
\]
with  \(\mathcal G_{x_S}\) being the parahoric group scheme at \(x_S\). Put
\[
  \widetilde W \;:=\; N_G(S)(L)\big/S(L)_0,
  \qquad
  \widetilde W_{\mathcal K}\;:=\;(N_G(S)(L)\cap\mathcal K)\big/S(L)_0 .
\]
Then the natural map
\[
  \phi:\;
  K_S\!\backslash G(F)/K_S
  \; \longrightarrow\;
  \widetilde W_{\mathcal K}^{\sigma}\!\backslash
  \widetilde W^{\sigma}/
  \widetilde W_{\mathcal K}^{\sigma},\qquad 
  K_S g K_S\;\longmapsto\;
  \widetilde W_{\mathcal K}^{\sigma}\,\overline{g}\,\widetilde W_{\mathcal K}^{\sigma},
\]
(where \(\overline{g}\) is the image of \(g\) in
\(\widetilde W^{\sigma}=N_G(S)(F)/S(F)_0\))
is a bijection.
\end{proposition}
Warning: the \(S\) in our notation differs from that in \cite{HR10}.
\begin{proof}
Let \(S'\) be the maximal \(L\)-split subtorus of \(S\).
Haines--Rapoport give a canonical bijection over \(L\):
\[
  \mathcal K\backslash G(L)/\mathcal K
  \;\xrightarrow{\;\sim\;}
  \widetilde W_{\mathcal K,S'}\backslash 
  \widetilde W_{S'}/
  \widetilde W_{\mathcal K,S'}
  \qquad
  \bigl(\text{\cite[Prop.\,8]{HR10}}\bigr).
\]
Since \(S=\Cent_G(S')\),
\(N_G(S)(L)=N_G(S')(L)\); hence
\(
  \widetilde W_{S'} = N_G(S)(L)\!/S'(L)_0
\).
The quotient map
\( \pi:\widetilde W_{S'}\twoheadrightarrow\widetilde W \)
has kernel
\(K:=S(L)_0/S'(L)_0\subset\widetilde W_{\mathcal K,S'}\),
so double--cosets are unchanged when we mod out by \(K\).
Thus
\[
  \mathcal K\backslash G(L)/\mathcal K
  \;\xrightarrow{\;\sim\;}
  \widetilde W_{\mathcal K}\backslash 
  \widetilde W/
  \widetilde W_{\mathcal K}.
\]
The parahoric \(\mathcal K\) has connected special fibre and is \(\sigma\)-stable
(\cite[Prop.\,3\,(ii) \& Rem.\,4]{HR10}).
Hence Remark 9 of \emph{loc.~cit.} gives
\[
   K_S\backslash G(F)/K_S
   \;\xrightarrow{\;\sim\;}
   \bigl(\mathcal K\backslash G(L)/\mathcal K\bigr)^{\sigma}.
\]
Combining the two identities yields,
\[
   K_S\backslash G(F)/K_S
   \;\xrightarrow{\;\sim\;}
   \bigl(\widetilde W_{\mathcal K}\backslash 
          \widetilde W/
          \widetilde W_{\mathcal K}\bigr)^{\sigma}.
\]
Now since \(x_S\) lies in the closure of a \(\sigma\)-invariant alcove, Remark 9 of \cite{HR10} also yields,
\[
   \bigl(\widetilde W_{\mathcal K}\backslash 
          \widetilde W/
          \widetilde W_{\mathcal K}\bigr)^{\sigma}
   \;\xrightarrow{\;\sim\;}
   \widetilde W_{\mathcal K}^{\sigma}\backslash 
   \widetilde W^{\sigma}/
   \widetilde W_{\mathcal K}^{\sigma}.
\]
Composing these bijections obtained above gives the desired bijection \(\phi\).
\end{proof}

\begin{theorem}[Cartan decomposition]
\label{thm:KS-double-cosets}
Let $G$ be a connected reductive group over a non-archimedean local field $F$ with residue characteristic~$p$.
Fix a maximally unramified elliptic maximal $F$-torus $S\subset G$ and write $x_S\in\mathcal{B}^{\mathrm{red}}(G,F)$
for the associated vertex. Put $K_S:=G(F)_{x_S,0}$. Let $I\subset\mathrm{Gal}(\bar{F}/F)$ denote inertia and let
$\sigma$ be (the image of) the geometric Frobenius.

Then the Kottwitz homomorphism
\[
  \omega_S \colon S(F) \longrightarrow X_{*}\!\bigl(S\bigr)_{I}^{\sigma}
\]
induces a canonical bijection of sets
\[
  W(G,S)_{x_S,0}\backslash X_{*}\!\bigl(S\bigr)_{I}^{\sigma} \xrightarrow{\;\sim\;} K_S\backslash G(F)/K_S,
\]
where $W(G,S)_{x_S,0}:=N_{G(F)_{x_S,0}}(S)/S(F)_0$.
\end{theorem}

\begin{proof}
We follow the strategy of Haines--Rostami~\cite{hainesrostami2010} but without requiring the vertex \(x_S\) to be special.
Retain the notation of Proposition~\ref{prop:sigma-Cartan} and \S7.1. Thus
\(\mathcal K:=\mathcal G_{x_S}^{\circ}(\mathcal O_L)\subset G(L)\) is the parahoric subgroup,
\(\widetilde W=N_G(S)(L)/S(L)_0\) is the extended affine (Iwahori--Weyl) group, and
\[
  \widetilde W_{\mathcal K}\;:=\;\bigl(N_G(S)(L)\cap \mathcal K\bigr)\big/S(L)_0\;\subset\;\widetilde W.
\]
By Proposition~\ref{prop:sigma-Cartan}, there is a canonical bijection
\begin{equation}\label{eq:BTdouble}
  \mathcal K\backslash G(L)/\mathcal K\;\xrightarrow{\;\sim\;}\;
  \widetilde W_{\mathcal K}\backslash \widetilde W/\widetilde W_{\mathcal K},
\end{equation}
which is $\sigma$--equivariant.

Write \(K:=\mathcal K^{\sigma}=G(F)\cap\mathcal K=\mathcal G_{x_S}^{\circ}(\mathcal O_F)\).
Using \(G(F)=G(L)^{\sigma}\), \(K=\mathcal K^{\sigma}\), and the $\sigma$--equivariance of~\eqref{eq:BTdouble}
(cf.\ the fixed-point passage in~\cite{hainesrostami2010}), we obtain a canonical bijection
\begin{equation}\label{eq:sigmadouble}
  K\backslash G(F)/K\;\xrightarrow{\;\sim\;}\;
  \bigl(\widetilde W_{\mathcal K}\bigr)^{\sigma}\backslash \widetilde W^{\sigma}\big/\bigl(\widetilde W_{\mathcal K}\bigr)^{\sigma}.
\end{equation}

Let \(S'\subset S_L\) denote the maximal \(L\)-split subtorus of \(S_L\).
Since \(S\) is maximally unramified, \(S_L\) is maximally \(L\)-split in \(G_L\), and hence \(S'\) is a maximal
\(L\)-split torus of \(G_L\). Write \(\mathrm{G}_{x_S}^\circ\) for the identity component of the reductive quotient
of the ft-N\'eron model of \(G\) at \(x_S\), and \(\mathrm S\) for the maximal torus in \(\mathrm{G}_{x_S}^\circ\)
corresponding to \(S'\). By Haines--Rapoport~\cite[Prop.~12]{HR10}, applied over \(L\) to the \(\sigma\)-stable facet
\(x_S\),
\[
  \widetilde W_{\mathcal K}
  \;=\; (N_G(S)(L)\cap \mathcal K)\big/S(L)_0
  \;\xrightarrow{\;\sim\;}\; W\!\big(\mathrm{G}_{x_S}^\circ,\,\mathrm S\big),
\]
hence, taking \(\sigma\)--fixed points,
\[
  \bigl(\widetilde W_{\mathcal K}\bigr)^{\sigma}\;\cong\;W\!\big(\mathrm G_{x_S}^\circ,\,\mathrm S\big)(\kappa_F).
\]
Moreover, by~\cite[Lem.~3.4.10(2)]{KalethaRegSC}, the natural map
\[
  N_{G(F)_{x_S,0}}\bigl(S\bigr)\big/S(F)_0\;\longrightarrow\;
  N_{\mathrm G_{x_S}^\circ}\bigl(\mathrm{S}\bigr)(\kappa_F)\big/\mathrm{S}(\kappa_F)
\]
is bijective, so that
\begin{equation}\label{eq:WK=WGS}
  W\!\big(\mathrm G_{x_S}^\circ,\,\mathrm S\big)\;\cong\;W(G,S)_{x_S,0}.
\end{equation}

Write \(W_{\mathrm{aff},S}\) for the affine Weyl subgroup attached to \((G,S)\). By Lemma~\ref{lem:no-sigma-in-wa},
we have \(W_{\mathrm{aff},S}^{\sigma}=1\).
Since the decomposition \(\widetilde W = W_{\mathrm{aff},S}\rtimes \Omega_S\) is \(\sigma\)-stable and the factorization
\(w\omega\) is unique (because \(W_{\mathrm{aff},S}\cap \Omega_S=\{1\}\)), it follows that
\[
\widetilde W^\sigma = W_{\mathrm{aff},S}^\sigma \rtimes \Omega_S^\sigma=\Omega_S^\sigma.
\]
The Kottwitz homomorphism identifies \(\Omega_S^{\sigma}\) with \(X_{*}\!\bigl(S\bigr)_{I}^{\sigma}\)
(cf.\ \cite[\S7]{Kot97}).

Using \eqref{eq:WK=WGS} and the identification \(\widetilde W^\sigma=\Omega_S^\sigma\) in~\eqref{eq:sigmadouble}, we obtain
\[
  K\backslash G(F)/K \;\cong\; W(G,S)_{x_S,0}\backslash \Omega_S^{\sigma}
  \;\cong\; W(G,S)_{x_S,0}\backslash X_{*}\!\bigl(S\bigr)_{I}^{\sigma}.
\]
Finally, \(K=K_S\) by definition, and this is the desired bijection.
\end{proof}

\begin{remark}
Haines pointed out to the author that if \(x_S\) is not special, the associated Hecke algebra may not be commutative.
\end{remark}

\section{Close field Hecke algebra isomorphism}
\label{app:kazhdan}

This appendix provides background on parahoric Hecke algebras and explains the origin of the
residue-characteristic hypotheses. It is logically independent of the main results in
Section~\ref{sec:hecke_iso_bernstein}.

Let $G$ be a connected reductive group defined over $F$ and let $S$ be a maximally unramified elliptic maximal torus $S$ of $G$. Let $x=x_S$ denote the point associated to $S$ in the reduced Bruhat-Tits building of $G$. Let $G(F)_{x,0}=:K$ the associated parahoric and for any $m\in \mathbb{Q}_{\geq 0}$, let $K_m $ denote the $m$-th congruence filtration subgroup of $K$. For $\ell \geq m$, let $F'$ be an $\ell$-close field and let $G'$ and $S'$ be as constructed in Proposition \ref{prop:vertex-transfer}. Let $x'=x_{S'}$ and define  $K'$, and $K'_m$ to be the $F'$ counterparts of  $K$, and $K_m$.

Let $t$ denote a non-negative integer. Write
\[
\id_t=(K_t,1), \qquad \id_t'=(K_t',1)
\]
for the trivial $1$-dimensional representations.

\paragraph{The Hecke algebra.}
Define
\[
\cH(G(F),\id_t) := \End_{G(F)}\!\bigl(\cind_{K_t}^{G(F)}\id_t\bigr), \qquad \cH(G(F'),\id_t') := \End_{G(F')}\!\bigl(\cind_{K'_t}^{G(F')}\id_t'\bigr).
\]
Then there is a natural identification: $\cH(G(F),\id_t)=\mathbf{1}_{K_t}* \cH(G(F))\!*\mathbf{1}_{K_t}$. Here $\cH(G(F))$ denotes the $\mathbb{C}$-algebra of compactly supported smooth functions from $G(F)$ to $\mathbb{C}$ with convolution computed against Haar measure normalized by $\vol(K_t)=1$, and similarly for $F'$ and \(\mathbf{1}_{K_t}\) the idempotent supported on \(K_t\). For the following theorem, we will write $\cH(G(F),K_t)$ instead of $\cH(G(F),\id_t)$ to better align with Ganapathy's theorem which we extend. The proof of the following theorem closely follows the proof of \cite[Theorem 4.1]{ganapathy2022}.

\begin{theorem}\label{thm:deligne-kazhdan-Hecke-isomorphism}
Fix a rational number $m \geq 0$. There exists an integer $\ell=\ell_m \geq m$ such that for any nonarchimedean local field $F'$ that is $\ell$-close to $F$, there is a canonical isomorphism of Hecke algebras:
\[
\kappa_{\ell,x}: \mathscr{H}(G(F), K_m) \xrightarrow{\sim} \mathscr{H}(G'(F'), K'_m).
\]
\end{theorem}
We will omit \(x\) and simply write \(\kappa_\ell\) in place of \(\kappa_{\ell,x}\) whenever there is no possibility of ambiguity. 

\begin{proof}
By Theorem~\ref{thm:KS-double-cosets}, the map $S(F)\twoheadrightarrow\Omega_S^{\sigma}$ induces a canonical bijection
\[
   K\backslash G(F)/K \xrightarrow{\;\sim\;} W(G,S)_{x_S,0}\backslash{\Omega_S^{\sigma}}.
\]

\paragraph{\textit{Compatible sections of the Kottwitz map.}}
For \(\ell\)-close fields \(F\) and \(F'\), by Theorem~\ref{thm:congruent-isom} (equivalently Theorem~5.3 for tori),
we have a canonical isomorphism \(\mathrm{Tr}_0:S(F)/S(F)_0 \xrightarrow{\sim} S'(F')/S'(F')_0\).
Via the Kottwitz homomorphisms \(S(F)/S(F)_0\simeq \Omega_S^\sigma\) and \(S'(F')/S'(F')_0\simeq \Omega_{S'}^\sigma\),
pick sections
\[
p:\Omega_S^\sigma\to S(F),\qquad p':\Omega_{S'}^\sigma\to S'(F')
\]
as in \cite[Lem.~2.4]{ganapathy2022}, compatible with \(\mathrm{Tr}_0\).
Identifying \(\Omega_S^\sigma\simeq \Omega_{S'}^\sigma\), set \(g_\lambda:=p(\lambda)\) and \(g'_\lambda:=p'(\lambda)\).

\paragraph{\textit{Truncated parahoric isomorphisms.}}
Fix \(m\ge 0\). Choose \(\ell_m\ge m\) large enough so that for every \(\ell\ge \ell_m\) we have,
by Theorem~\ref{thm:depth-comp-HH} 
canonical isomorphisms of truncated parahorics at the levels needed below (in particular at levels \(m\) and \(\ell\)):
\[
  \alpha_t:K/K_t \xrightarrow{\sim} K'/K'_t\qquad (t\in\{m,\ell\}),
\]
compatible with the chosen transfer data (and hence with the chosen normalizer representatives used to index
the relevant double cosets; cf.\ \cite[\S4]{ganapathy2022}).

\paragraph{\textit{Finite presentation of \(\Hc(G(F),K_m)\).}}
Let \({\mathscr{H}}=\Hc(G(F),K_m)\). By \cite[Thm.~4.1]{ganapathy2022} (and the discussion preceding it), we may choose a finite subset
\(\Lambda=\{\lambda_1,\dots,\lambda_p\}\subset \Omega_S^{\sigma,+}\) such that the elements
\[
f_i:=\one_{K_m g_{\lambda_i}K_m}\qquad(1\le i\le p)
\]
generate \({\mathscr{H}}\) as a \(\C\)-algebra, and such that there is a finite presentation
\[
   {\mathscr{H}} \ \cong\ \frac{\C\langle X_1,\dots,X_p\rangle}{\langle R_1,\dots,R_q\rangle},
   \qquad X_i\mapsto f_i,
\]
where each relation \(R_j\) is a \(\C\)-linear combination of monomials of bounded total length.
Let \(N_0\) be the maximal length of any monomial appearing among the \(R_j\).

For \(d\ge 1\), set
\[
  B_d(\Lambda)\;:=\;\Bigl\{\nu\in\Omega_S^{\sigma}:\ \nu
  \text{ is a word of length }\le d\text{ in }\Lambda\cup\Lambda^{-1}\Bigr\},
\qquad
\tilde{\Lambda}:=B_{N_0}(\Lambda).
\]
Then for any word \((\lambda_{i_1},\dots,\lambda_{i_d})\) with \(d\le N_0\),
\[
\operatorname{supp}\bigl(f_{i_1}*\cdots * f_{i_d}\bigr)
\ \subset\
\bigcup_{\nu\in B_d(\Lambda)} K g_\nu K
\ \subset\
\bigcup_{\nu\in \tilde{\Lambda}} K g_\nu K,
\]
so only the finitely many elements \(\{g_\nu\}_{\nu\in\tilde\Lambda}\) occur in the support
considerations needed to verify the relations \(R_j\).

\paragraph{\textit{Conjugation control.}}
Enlarge \(\ell_m\) if necessary so that for all \(\ell\ge \ell_m\) and all \(\nu\in\tilde\Lambda\),
the elements \(g_\nu\) and \(g'_\nu\) normalize the depth-\(\ell\) congruence subgroups \(K_\ell\) and \(K'_\ell\).
(This is exactly the normalizer-stability input used in \cite[\S4]{ganapathy2022}; since \(\tilde\Lambda\) is finite,
one can and does choose \(\ell_m\) uniformly for all \(\nu\in\tilde\Lambda\).)
Thus we have:

\begin{lemma}\label{lem:conj}
For every \(\nu\in\tilde{\Lambda}\) and every \(\ell\ge \ell_m\), one has
\[
g_\nu K_\ell g_\nu^{-1}=K_\ell,
\qquad
g'_\nu K'_\ell (g'_\nu)^{-1}=K'_\ell.
\]
\end{lemma}

\paragraph{\textit{Volume comparison.}}
To transport the convolution structure constants appearing in the relations \(R_j\), it suffices (by the
support bound above) to compare, for \(\lambda,\kappa,\mu\in\tilde{\Lambda}\),
\[
   \vol\!\bigl(K_m g_\lambda K_m \;\cap\; g_\kappa K_m g_\mu K_m\bigr).
\]
This is achieved by the following standard stability statement (cf.\ \cite[Lem.~4.6 and the argument on p.~321]{ganapathy2022}),
which we state in a form tailored to our notation:

\begin{lemma}[Volume Stability]\label{lem:volume}
Fix \(\ell\ge \ell_m\) and normalize Haar measures so that \(\vol(K_\ell)=\vol(K'_\ell)=1\).
Let \(g,h,k\in\{g_\nu\}_{\nu\in\tilde\Lambda}\) and let \(g',h',k'\) be the corresponding elements in \(G'(F')\).
Then
\[
   \vol\!\bigl(K_m g K_m \cap k h K_m\bigr)
   \;=\;
   \vol\!\bigl(K'_m g' K'_m \cap k' h' K'_m\bigr).
\]
\end{lemma}

\begin{proof}
By Lemma~\ref{lem:conj}, the subgroup \(K_\ell\) is normalized by \(g,h,k\), and \(K'_\ell\) is normalized by \(g',h',k'\).
Since \(\ell\ge m\), we have \(K_\ell\triangleleft K_m\) and \(K'_\ell\triangleleft K'_m\). Hence both sets
\(K_m g K_m\) and \(k h K_m\) are unions of left cosets of \(K_\ell\), and therefore so is their intersection
\(I:=K_m g K_m\cap k h K_m\). Thus
\[
I=\bigsqcup_{i=1}^t x_i K_\ell
\quad\text{for some }t\in\Z_{\ge 0},
\qquad
\vol(I)=t\cdot \vol(K_\ell)=t,
\]
and similarly on the \(F'\)-side with \(I'\) and \(t'\).

Consider the images of these sets in the finite quotient by \(K_\ell\). Because \(K_\ell\triangleleft K_m\) and
\(g,h,k\) normalize \(K_\ell\), the double cosets \(K_m g K_m\) and \(k h K_m\) project to well-defined subsets of the
finite set \(K_m\backslash G(F)/K_\ell\), and \(|I/K_\ell|=t\).
The truncation isomorphism at level \(\ell\) (coming from \(\alpha_\ell\) and the compatibility of the chosen representatives
\(g_\nu\leftrightarrow g'_\nu\)) identifies the corresponding finite quotients on the \(F\) and \(F'\) sides and carries
the projected intersection to the projected intersection. Hence \(t=t'\), and since \(\vol(K_\ell)=\vol(K'_\ell)=1\),
we get \(\vol(I)=\vol(I')\).
\end{proof}

\paragraph{\textit{Transport of relations and construction of \(\kappa_\ell\).}}
Define \(f_i':=\one_{K'_m g'_{\lambda_i}K'_m}\) and a \(\C\)-algebra map on generators
\[
   \kappa_{\ell}: {\mathscr{H}} \longrightarrow {\mathscr{H}}',
   \qquad f_i\longmapsto f_i' \ (1\le i\le p).
\]
By Lemma~\ref{lem:volume} (applied to all convolution products appearing in the relations \(R_j\), which only involve
elements from \(\tilde\Lambda\)), the structure constants needed to evaluate each \(R_j\) are identical on the \(F\) and \(F'\) sides.
Thus \(R_j(f_1',\dots,f_p')=0\) for every \(j\), so the universal property of the presentation shows that \(\kappa_\ell\) is
well-defined.

By symmetry (swapping \(F\) and \(F'\)) we obtain a homomorphism \(\kappa_\ell':{\mathscr{H}}'\to {\mathscr{H}}\) sending
\(f_i'\mapsto f_i\). The composites \(\kappa_\ell'\circ \kappa_\ell\) and \(\kappa_\ell\circ\kappa_\ell'\) fix the chosen generators,
hence are the identity maps. Therefore \(\kappa_\ell\) is an algebra isomorphism. By construction it is canonical (it depends only
on the fixed congruence datum and the compatible choices of sections as in \cite[Lem.~2.4]{ganapathy2022}).
\end{proof}

\begin{lemma}
\label{lem:idempotents-transport}
For every non-negative rational $r\leq m$, the Deligne-Kazhdan isomorphism
\[
\kappa_{\ell} : \cH\!\bigl(G(F),K_m\bigr) \xrightarrow{\;\sim\;} \cH\!\bigl(G(F'),K_m'\bigr)
\]
sends the central idempotent $e_{r}=\mathbf{1}_{K_{r}}/\!\vol(K_{r})$ to the idempotent $e'_{r}=\mathbf{1}_{K'_{r}}/\!\vol(K'_{r})$ where \(\Phi_{G,x}(m)\leq \ell\) and $F$ and $F'$ are $\ell$-close.
\end{lemma}

\begin{proof}
Put $K:=G(F)_{x,0}$ and $K_r:=G(F)_{x,r}$, likewise $K':=G'(F')_{x',0}$ and $K'_r:=G'(F')_{x',r}$.
Since $r\le m$, we have $K_m\subset K_r$ and $K'_m\subset K'_r$. Since $\ell\ge \Phi_{G,x}(m)\ge \Phi_{G,x}(r)$, Theorem~\ref{thm:depth-comp-HH} provides an isomorphism
\[
\Psi_{x,r}:\ K/K_r \xrightarrow{\ \sim\ } K'/K'_r,
\]
functorial in $(G,x)$ and compatible with the identifications of normalizer
representatives used to index $K_m$--double cosets.
\smallskip
Write the characteristic function of $K_r$ as a finite disjoint sum over
$K_m$--double cosets inside $K_r$:
\[
\mathbf{1}_{K_r}\;=\;\sum_{D\in\,K_m\backslash K_r/K_m} \mathbf{1}_D
\qquad\text{in }\cH\!\bigl(G(F),K_m\bigr).
\]
The natural map
\[
K_r/K_m \;\longrightarrow\; K_m\backslash K_r/K_m,\qquad gK_m\mapsto K_m g K_m
\]

is a bijection, i.e.,
\[
K_m\backslash K_r/K_m \;\cong\; K_r/K_m \;\cong\; K_m\backslash K_r.
\]
Therefore, employing the isomorphism 
\[
\varphi_\ell: K_r / K_m \xrightarrow{\sim} K_r' / K_m' ,
\]
obtained out of \(\Psi_{x,r}\), we get,
\[
\kappa_\ell\!\left(\mathbf{1}_{K_r}\right)
\;=\;\sum_{D'\in\,K'_m\backslash K'_r/K'_m} \mathbf{1}_{D'}=\;\sum_{D'\in\,K'_r/K'_m} \mathbf{1}_{D'}
\;=\;\mathbf{1}_{K'_r},
\]
where $D'$ denotes the transferred $K'_m$--double coset.
Since $\vol K_{\ell}=\vol K'_{\ell}=1$, $\varphi_\ell$ implies $\vol K_{r}=\vol K'_{r}$. Hence $\kappa_\ell(e_r)=e'_r$.

\end{proof}

\subsection{The Hecke-module realization and the transfer functor $\Kaz_{\ell}$}\label{sec:def-Kaz}

Let $J\subset G(F)$ be a compact open subgroup and let $\tau$ be a smooth finite--dimensional
representation of $J$. Put
\[
  \cH := \cH(G(F),\tau) \;=\; \End_{G(F)}\!\bigl(\mathrm{c\text{-}ind}_{J}^{G(F)} \tau\bigr),
  \qquad
  \cH' := \cH\bigl(G'(F'),\tau'\bigr).
\]
Define the full subcategory of smooth $G(F)$--representations
\[
  \cR_\tau\bigl(G(F)\bigr)
  \;:=\; \{\text{smooth } G(F)\text{-representations generated by their }\tau\text{-isotypic subspace}\}.
\]
The functor
\[
  \mathbf{M}_\ell:\ \cR_\tau\bigl(G(F)\bigr) \longrightarrow \cH\text{-}\mathrm{Mod},\qquad
  \mathbf{M}_\ell(\pi):=\Hom_{J}(\tau,\pi),
\]
endowed with the natural (left) $\cH$--action by precomposition, is exact and faithful and induces a
bijection on isomorphism classes of irreducible objects. In general $\mathbf{M}_\ell$ need not be
full nor essentially surjective; it is an equivalence precisely when $(J,\tau)$ is a Bushnell--Kutzko type.

\medskip
Now fix a vertex $x\in\mathcal{B}^{\mathrm{red}}(G,F)$ and let $x'=\mathcal{B}_\ell(x)\in\mathcal{B}^{\mathrm{red}}(G',F')$ as in
\S \ref{sec:hecke_iso_bernstein}. For $r\ge0$ with $\Phi_{G,x}(r)\le \ell$, set
\[
  K:=G(F)_{x,0},\quad K_r:=G(F)_{x,r}, \qquad
  K':=G'(F')_{x',0},\quad K'_r:=G'(F')_{x',r}.
\]
In this application we take $J=K_r$, $\tau=\mathbf{1}_{K_r}$ and $J'=K'_r$, $\tau'=\mathbf{1}_{K'_r}$.
Let
\[
  \kappa_\ell:\ \cH\bigl(G(F),\mathbf{1}_{K_r}\bigr)\xrightarrow{\ \sim\ } \cH\bigl(G'(F'),\mathbf{1}_{K'_r}\bigr)
\]
be the Hecke--algebra isomorphism of Theorem~\ref{thm:deligne-kazhdan-Hecke-isomorphism}.

\begin{definition}[Kazhdan transfer at level $r$]
\label{def:Kaz-transfer}
For $\pi\in\cR_{\mathbf{1}_{K_r}}\!\bigl(G(F)\bigr)$, regard $\mathbf{M}_\ell(\pi)$ as a left $\cH'$--module via $\kappa_\ell$
(\,$h'\!\cdot m := (\kappa_\ell^{-1}(h'))m$\,). Define
\[
  \Kaz_\ell^{(G,x;r)}(\pi)
  \;:=\;
  \bigl(\mathrm{c\text{-}ind}_{K'_r}^{G'(F')}\mathbf{1}_{K'_r}\bigr)\ \otimes_{\ \cH'}\ \bigl(\mathbf{M}_\ell(\pi)\bigr)_{\kappa_\ell}
  \ \in\ \cR_{\mathbf{1}_{K'_r}}\!\bigl(G'(F')\bigr).
\]
Then $\Kaz_\ell^{(G,x;r)}$ is exact and satisfies a natural isomorphism
\[
  \mathbf{M}'_\ell\bigl(\Kaz_\ell^{(G,x;r)}(\pi)\bigr)\ \cong\ \kappa_\ell\!\bigl(\mathbf{M}_\ell(\pi)\bigr)
  \qquad(\pi\in\cR_{\mathbf{1}_{K_r}}\!(G(F))).
\]
In particular, on irreducibles it realizes the bijection induced by $\kappa_\ell$ on simple
modules.
\end{definition}
Define the depth \emph{relative to $x$} by
\[
  \dep_{x}(\pi)\;:=\;\inf\{\, r\ge 0 \mid \pi^{K_{x,r}}\neq 0\ \text{and}\ \pi^{K_{x,r+}}=0\,\},
\]
where, as before, $K_{x,r}$ denotes the minimal--congruence filtration at $x$ and
$\pi^{K_{x,r+}}:=\bigcup_{t>r}\pi^{K_{x,t}}$.
\begin{proposition}
\label{prop:depth-preserved-Kaz}
For an irreducible smooth representation $\pi$ of $G(F)$ set
\[
  \pi' \;:=\; \Kaz^{(G,x;\,\cdot)}_{\ell}(\pi).
\]
Then
\[
   \dep_{x}(\pi)\;=\;\dep_{x'}(\pi').
\]
\end{proposition}

\begin{proof}
By the (level--$\ell$) Hecke--algebra isomorphism attached to $x$ and $x'$, for every $s\ge 0$
in the truncation range the central idempotent $e_{x,s}$ of $\cH\!\bigl(G(F),K_{x,s}\bigr)$
transfers to $e'_{x',s}$ in $\cH\!\bigl(G'(F'),K'_{x',s}\bigr)$.
Hence, for all such $s$,
\begin{equation*}
   \dim \pi^{K_{x,s}} \;=\; \dim \pi'^{K'_{x',s}}.
\end{equation*}
Since $K_{x,s+}=\bigcup_{t>s}K_{x,t}$ and likewise for $x'$, we have
\[
   \pi^{K_{x,s+}}=\bigcup_{t>s}\pi^{K_{x,t}}
   \quad\text{and}\quad
   \pi'^{K'_{x',s+}}=\bigcup_{t>s}\pi'^{K'_{x',t}},
\]
and the preceding equality at each $t$ implies
\[
   \bigl(\pi^{K_{x,s}}\neq 0\bigr)\iff\bigl(\pi'^{K'_{x',s}}\neq 0\bigr),
   \qquad
   \bigl(\pi^{K_{x,s+}}=0\bigr)\iff\bigl(\pi'^{K'_{x',s+}}=0\bigr).
\]
Taking the infimum over $s$ in the defining condition yields
$\dep_{x}(\pi)=\dep_{x'}(\pi')$.
\end{proof}
Note that since there always exists a vertex $x$ with $\dep(\pi)=\dep_{x}(\pi)$. By choosing this vertex for Kazhdan transfer, we obtain
\[
   \dep(\pi')\;=\;\dep(\pi).
\]

\begin{proposition}
\label{prop:Kaz-preserves-SC}
Let $F$ and $F'$ be $\ell$-close local fields. For an irreducible representation $\pi$ of $G(F)$ of \(\dep_x(\pi)\leq m\), where \(\Phi_{G,x}(m)\leq \ell\), write $\pi':=\Kaz_{\ell}(\pi)$. Then
\[
     \pi \text{ is supercuspidal } \Longleftrightarrow \pi' \text{ is supercuspidal}.
\]
\end{proposition}

\begin{proof}
A smooth irreducible representation \(\pi\) of \(G(F)\) is supercuspidal if and only if all its matrix coefficients are compactly supported modulo the center \(Z(G)(F)\). That is, for any \(v \in \pi\), \(v^\vee \in \pi^\vee\), the function  
\[
h_\pi(g) = \langle v^\vee, \pi(g)v \rangle
\] satisfies:  
\[
\mathrm{Supp}(h_\pi) \subset Z(G)^{\circ}(F) \cdot C \quad \text{for some compact set } C \subset G(F).
\]
Now, choose non-zero \(K_{m}\)-fixed vectors:
 \(v \in \pi^{K_{m}}\),\qquad  \(v^\vee \in (\pi^\vee)^{K_{m}}\),
and define the matrix coefficient:
\[
h_\pi(g) = \langle v^\vee, \pi(g)v \rangle.
\]
Under the Kazhdan transfer, we obtain:
 \(v^{\prime} = \mathrm{Kaz}_{m}(v) \in (\pi^{\prime})^{K_{m}^{\prime}}\), \qquad  \(v^{\vee\prime} = \mathrm{Kaz}_{\ell}(v^\vee) \in (\pi^{\prime\vee})^{K_{\ell}^{\prime}}\),
and the matrix coefficient:
  \[
  h_{\pi^{\prime}}(g^{\prime}) = \langle v^{\vee\prime}, \pi^{\prime}(g^{\prime}) v^{\prime} \rangle.
  \]
  Since \(C\) is compact, it is contained in a finite union of double cosets:
\[
C \subset \bigcup_{j=1}^M K_{m} h_j K_{m}.
\]
Fix \(\ell\ge \Phi_{G,x}(m)\) and the isomorphisms
\[
\alpha_m:K/K_m\xrightarrow{\ \sim\ }K'/K'_m,\qquad
\alpha_\ell:K/K_\ell\xrightarrow{\ \sim\ }K'/K'_\ell,
\]
with the natural projection square commuting. Choose Cartan representatives
\(g_{\lambda_j}\in G(F)\) for the \(K_m\)-double cosets meeting \(C\), and let
\(g'_{\lambda_j}\in G'(F')\) be the corresponding representatives provided by
Theorem~\ref{thm:depth-comp-HH} at level \(m\).
Since \(K_m\supset K_\ell\) and \(K_\ell\lhd K\), we can write
\[
K_m\,g_{\lambda_j}\,K_m\ \subset\ \bigsqcup_{i=1}^{t_j} y_{j,i}\,g_{\lambda_j}\,K_\ell
\qquad (y_{j,i}\in K).
\]
Define \(y'_{j,i}\in K'\) by \(\pi'_\ell(y'_{j,i})=\alpha_\ell\bigl(\pi_\ell(y_{j,i})\bigr)\) and set
\(x'_{j,i}:=y'_{j,i}\,g'_{\lambda_j}\).
Then
\[
x'_{j,i}K'_\ell
= y'_{j,i}\,g'_{\lambda_j}K'_\ell
\subset K'_m\,g'_{\lambda_j}\,K'_\ell
\subset K'_m\,g'_{\lambda_j}\,K'_m,
\]
so
\[
\Supp(h_{\pi'})\ \subset\ Z(G')^\circ(F')\cdot
\bigcup_{j=1}^M\ \bigcup_{i=1}^{t_j} x'_{j,i}K'_\ell\ \subset\
Z(G')^\circ(F')\cdot\bigcup_{j=1}^M K'_m\,g'_{\lambda_j}\,K'_m,
\]
which is compact modulo \(Z(G')^\circ(F')\).
The converse follows from symmetry.

\end{proof}
\begin{lemma}\label{lem:kaz-preserves-genericity}
Let \(F\) and \(F^{\prime}\) be \(\ell\)-close non-archimedean local fields, and let \(G\) be a quasi-split connected reductive group over \(F\) and \(D_\ell\) be a congruence datum. Fix a Whittaker datum \(\mathbf{w} = (B, \psi)\) for \(G\), where \(B\) is a Borel subgroup defined over \(F\) and \(\psi: U(F) \to \mathbb{C}^\times\) is a non-degenerate character of the unipotent radical \(U\) of \(B\). Let \(\pi\) be an irreducible \(\mathbf{w}\)-generic supercuspidal representation of \(G(F)\) of depth \(r\). For  \(\ell \geq \max\left( \Phi_{G,x}(r), \max_{a \in \Phi(G,A)} \phi_{F_a/F}(e_a r_a) \right)\), where \(r_a = \operatorname{depth}(\psi_a)\), the following hold:  
1. There is a canonically transferred Whittaker datum \(\mathbf{w}' = (B', \psi')\) for \(G'/F'\) obtained via \(D_\ell\). 
2. The representation \(\pi' := \mathrm{Kaz}_\ell(\pi)\) is \(\mathbf{w}'\)-generic.  
\end{lemma}
\begin{proof}

Fix a vertex \(x\in\mathcal{B}^{\mathrm{red}}(G,F)\) and write
\(K:=G(F)_{x,0}\), \(K_\ell:=G(F)_{x,\ell}\) with
\(\ell\ge \Phi_{G,x}(m)\). Let \(\alpha_\ell:K/K_\ell\xrightarrow{\sim}K'/K'_\ell\) and
\(\alpha_m:K/K_m\xrightarrow{\sim}K'/K'_m\) be the truncated parahoric
isomorphisms with commuting projection square (Theorem~\ref{thm:depth-comp-HH}).
By Proposition~\ref{prop:vertex-transfer}, a pinning
\(\mathcal P=(G,B,T,\{X_\alpha\}_{\alpha\in\Delta})\) transfers to a pinning
\(\mathcal P'=(G',B',T',\{X'_{\alpha'}\}_{\alpha'\in\Delta'})\) compatible with
the identification of \(\Delta\) with \(\Delta'\).

\smallskip

The pinning \(\mathcal P\) fixes for each \(\alpha\in\Delta\) a root subgroup
parameterization \(x_\alpha:\mathbb{G}_a\!\to U_\alpha\) with
\((dx_\alpha)_0(1)=X_\alpha\). Fix a nontrivial additive character
\(\psi_F:F\to\C^\times\) and constants \(c_\alpha\in F^\times\) (determined by the normalization of \(\{X_\alpha\}\)).
For \(u=\prod_{\alpha\in\Delta}x_\alpha(t_\alpha)\in U(F)\) (ordered by any
fixed convex ordering) define
\[
\psi(u):=\psi_F\!\Bigl(\sum_{\alpha\in\Delta} c_\alpha t_\alpha\Bigr).
\]
This is well defined and generic; thus \(\mathbf w=(B,\psi)\) is the
Whittaker datum attached to \(\mathcal P\). Under transfer,
we get \(x'_{\alpha'}:\mathbb{G}_a\!\to U'_{\alpha'}\) with
\((dx'_{\alpha'})_0(1)=X'_{\alpha'}\) and corresponding constants \(c'_{\alpha'}\),
together with an additive character \(\psi_{F'}\) matched to
\(\psi_F\) at the congruence level \(\ell\). Hence the same formula defines
\(\psi'\) on \(U'(F')\), giving \(\mathbf w'=(B',\psi')\).
In particular, for each simple root, the truncated isomorphisms
\(U_\alpha(F)_{x,0}/U_\alpha(F)_{x,r_\alpha}\!\xrightarrow{\sim}\!
U'_{\alpha'}(F')_{x',0}/U'_{\alpha'}(F')_{x',r_\alpha}\)
(Lemma~\ref{lem:root-filtration-HH}, Theorem~\ref{thm:depth-comp-HH})
transport the characters \(\psi_\alpha\) to \(\psi'_{\alpha'}\), hence \(\psi\) to \(\psi'\) on the relevant congruence quotients.

\smallskip

Let \(\omega_\pi\) (resp.\ \(\omega_{\pi'}\)) be the central character of \(\pi\)
(resp.\ \(\pi'\)). By Theorem~\ref{thm:congruent-isom} the restrictions of
\(\omega_\pi\) and \(\omega_{\pi'}\) to \(Z(G)^\circ(F)\) and
\(Z(G')^\circ(F')\) correspond under transfer. Since \(h_\pi\) has support
compact modulo \(Z(G)^\circ(F)\), choose a compact \(C\subset G(F)\) with
\(\Supp(h_\pi)\subset Z(G)^\circ(F)\!\cdot C\).
Define the descent \(\overline h_\pi\) to \(G(F)/Z(G)^\circ(F)\) by
\(\overline h_\pi(\overline g):=h_\pi(g)\,\omega_\pi(z)^{-1}\) for any
decomposition \(g=zg_C\) with \(z\in Z(G)^\circ(F)\), \(g_C\in C\);
this is well defined and compactly supported. Do the same for \(\pi'\) to get
\(\overline h_{\pi'}\) on \(G'(F')/Z(G')^\circ(F')\). Proving the Whittaker
equivariance for \(\overline h_{\pi'}\) implies it for \(h_{\pi'}\), so we work
with compactly supported functions henceforth and drop the bars from the
notation.

\smallskip

Pick \(v\in\pi^{K_\ell}\) and a nonzero \(\psi\)-Whittaker functional
\(\lambda\in\mathrm{Wh}_\psi(\pi^\vee)\); set
\[
W_\pi(g):=\lambda\bigl(\pi(g)v\bigr)\qquad(g\in G(F)).
\]
Then \(W_\pi(ug)=\psi(u)\,W_\pi(g)\) for all \(u\in U(F)\), and
\(W_\pi(gk)=W_\pi(g)\) for all \(k\in K_\ell\).
Thus \(\Supp(W_\pi)\) is a union of finitely many right \(K_\ell\)-cosets:
\[
\Supp(W_\pi)\ \subset\ \bigcup_{j=1}^M \ \bigcup_{i=1}^{t_j} x_{j,i}K_\ell,
\qquad x_{j,i}\in K\,g_{\lambda_j},
\]
where \(g_{\lambda_j}\) are Cartan representatives for the finitely many
\(K_\ell\)-double cosets meeting a fixed compact set \(C\).
For each \(x_{j,i}\) choose \(x'_{j,i}\in G'(F')\) with
\(\pi'_\ell(x'_{j,i}K'_\ell)=\alpha_\ell\bigl(\pi_\ell(x_{j,i}K_\ell)\bigr)\);
then
\[
\Supp(W_{\pi'})\ \subset\ \bigcup_{j=1}^M \ \bigcup_{i=1}^{t_j} x'_{j,i}K'_\ell,
\qquad\text{and}\qquad g_{\lambda_j}\longleftrightarrow g'_{\lambda_j}
\]
by the Cartan/section compatibility (cf.\ the proof of
Proposition~\ref{prop:Kaz-preserves-SC}). 

\smallskip
Write \(x_{j,i}=y_{j,i}\,g_{\lambda_j}\,k_{j,i}\) with \(y_{j,i}\in K\),
\(k_{j,i}\in K_\ell\). For each simple root \(\alpha\) let \(r_\alpha\ge 0\) be
such that \(\psi_\alpha\) is trivial on \(U_\alpha(F)_{x,r_\alpha^+}\);
choose \(\ell\) with
\(\ell\ge \max_\alpha \phi_{F_\alpha/F}(e_\alpha r_\alpha)\).
Then the truncated maps
\[
U_\alpha(F)_{x,0}\big/U_\alpha(F)_{x,r_\alpha^+}
\longrightarrow
U'_{\alpha'}(F')_{x',0}\big/U'_{\alpha'}(F')_{x',r_\alpha^+}
\]
carry \(x_\alpha(t)\) to \(x'_{\alpha'}(t')\) compatibly with the
characters \(\psi_\alpha\) and \(\psi'_{\alpha'}\). Hence for every
\(u\in U(F)_{x,0}\) with image \(\bar u\) in the product of these congruence
quotients and corresponding \(u'\in U'(F')_{x',0}\) with the same image,
we have
\[
\psi(u)=\psi'(u').
\]
Using right \(K_\ell\)-invariance and the above decompositions we get, for the
corresponding primed data,
\[
W_{\pi'}(u'\,x'_{j,i})
= \lambda'\bigl(\pi'(u')\pi'(x'_{j,i})v'\bigr)
= \psi'(u')\,\lambda'\bigl(\pi'(x'_{j,i})v'\bigr)
= \psi(u)\,W_{\pi}(x_{j,i}),
\]
where \(v'\in(\pi')^{K'_\ell}\), \(\lambda'\in\mathrm{Wh}_{\psi'}\bigl((\pi')^\vee\bigr)\) are the
transfers of \(v,\lambda\) under the Hecke-module isomorphisms induced by
\(\alpha_\ell\) (Theorem~\ref{thm:depth-comp-HH}). Since the right cosets
\(x'_{j,i}K'_\ell\) cover \(\Supp(W_{\pi'})\), this proves that
\(W_{\pi'}\) is \((U',\psi')\)-equivariant:
\[
W_{\pi'}(u'g')=\psi'(u')\,W_{\pi'}(g')\qquad
\bigl(u'\in U'(F'),\ g'\in G'(F')\bigr).
\]

\smallskip
We have produced a nonzero right \(K'_\ell\)-invariant Whittaker function
\(W_{\pi'}\) for \((B',\psi')\). Hence \(\pi'\) is \(\mathbf w'\)-generic.
The converse implication follows by symmetry of the construction.
\end{proof}
\begin{theorem}\label{thm:ell-close-kaz-yu-compatability}
    Let $F$ and $F'$ be $\ell$-close non-archimedean local fields with residue characteristic $p$, and let $G$ be a connected reductive group over $F$ with transfer $G'$ over $F'$ via a congruence datum $D_\ell$. Fix an integer $m \geq 0$. There exists $\ell_\dagger = \ell_\dagger(G, \pi,m) \gg 0$ such that for all $\ell \geq \ell_\dagger$, the following holds:  
For any irreducible supercuspidal representation $\pi$ of $G(F)$ of depth $\leq m$ arising from a Yu datum $\Sigma$,  
\[
\mathrm{Kaz}_\ell^{(G,x;m)}(\pi) = \pi(\Sigma')
\]  
where $\Sigma'$ is the Deligne transfer of $\Sigma$ to $G'$, and $\pi(\Sigma')$ is the supercuspidal representation associated to Yu datum $\Sigma'$ and $x$ is the point that is part of the datum $\Sigma$.

\end{theorem}
\begin{proof}
Let $\Sigma = (\pi_0, (G^i)_{i=0}^d, (\phi_i)_{i=0}^d)$ be a Yu datum for $\pi$, and let 
\[
\pi = \pi(\Sigma) = \mathrm{c\text{-}ind}_{K^d}^{G(F)} \tau.
\]
be the associated irreducible supercuspidal representation.
Transfer $G^i$ to $(G^i)'$ over $F'$ via $\mathscr{B}_\ell$ (Proposition~\ref{prop:vertex-transfer}), and transfer $x \in \mathcal{B}(G^0, F)$ to $x' \in \mathcal{B}((G^0)', F')$.
By Propositions~\ref{prop:Kaz-preserves-SC} and \ref{prop:depth-preserved-Kaz}, the transfer 
\[
\pi_0' := \mathrm{Kaz}_\ell(\pi_0) \in \Irr((G^0)'(F')),
\]
is depth-zero supercuspidal. Let \(r_i = \text{depth}(\phi_i)\). By Theorem \ref{thm:depth-comp-HH}, for \(\ell \geq \Phi_{G^i,x_i}(r_i)\) (where \(x_i \in \mathcal{B}(G^i, F)\) is a vertex), there is a canonical isomorphism:
     \[
     \Psi_{x_i, r_i}: \frac{G^i(F)_{x_i,0}}{G^i(F)_{x_i, r_i+}} \stackrel{\sim}{\longrightarrow} \frac{(G^i)'(F')_{x_i',0}}{(G^i)'(F')_{x_i', r_i+}},
     \]
where \(x_i' = \mathscr{B}_\ell(x_i)\). Since \(\phi_i\) is trivial on \(G^i(F)_{x_i, r_i+}\), it factors through \(G^i(F)_{x_i,0}/G^i(F)_{x_i, r_i+}\). Define \(\phi_i'=\phi_i\circ\Psi_{x_i, r_i}^{-1}\). Extend \(\phi_i'\) trivially to \((G^i)'(F')\). We now show that \(\phi_i'\) is \((G^{i+1})'\)-generic.  Recall that genericity is defined for elements in \(\operatorname{Lie}^*(Z^{i\circ})\)  via conditions GE1 and GE2 of \cite[\S 8]{Yu2001} where \(Z^i=Z(G^i)\) . The character \(\phi_i\) corresponds to an element \(X^*_i \in (\operatorname{Lie}^*(Z^i)^\circ)(F)_{-r_i}\) (via the Moy-Prasad isomorphism), and \(\phi_i'\) corresponds to an element \(X^{*\prime}_i \in (\operatorname{Lie}^*((Z^i)')^\circ)(F')_{-r_i}\). We show \(X^{*\prime}_i\) satisfies GE1 and GE2 for \(((G^i)', (G^{i+1})')\).  GE1 requires \(\operatorname{ord}(X_i^{*\prime}(H_{a^\prime})) = -r_i\) for all \(a^\prime \in \Phi((G^{i+1})^\prime, T^\prime, \overline{F^\prime}) \setminus \Phi((G^{i})^\prime, T^\prime, \overline{F^\prime})\), where \(T^\prime\) is a maximal tamely ramified \(F^\prime\)-torus and \(H_a:=da^\vee(1)\). Since \(\phi_i\) is generic, \(X^*_i\) satisfies GE1 for \((G^i, G^{i+1})\). The transfer preserves root systems: the simplicial isomorphism \(\mathscr{B}_\ell\) and congruence datum \(D_\ell\) induce a bijection \(\Phi_\ell : \Phi(G, T, F) \xrightarrow{\sim} \Phi(G', T', F')\). For \(a' \in \Phi(G^{{i+1}\prime}, T^{\prime}, F^\prime) \setminus \Phi((G^i)', T', F')\), let \(a = \Phi_\ell^{-1}(a') \in \Phi(G, T, F) \setminus \Phi(G^i, T, F)\). Then
  \[
  \operatorname{ord}_{F'}\big( X^{*\prime}_i(H_{a'}) \big) = \operatorname{ord}_F \big( X^*_i(H_a) \big) = -r_i,
  \]
since \(\Psi_{x_i, r_i}\) preserves valuations (as it is induced by the ring isomorphism \(\mathcal{O}_F / \mathfrak{p}_F^\ell \cong \mathcal{O}_{F'} / \mathfrak{p}_{F'}^\ell\)). Thus, \(X^{*\prime}_i\) satisfies GE1 for \(((G^i)', (G^{i+1})')\). The condition GE2 holds automatically by \cite[Lemma 8.1]{Yu2001} since \(p\) is odd and not a torsion prime for the dual root datum.
Thus, the transferred datum
\[
\Sigma' = (\pi_0', ((G^i)')_{i=0}^d, (\phi_i')_{i=0}^d),
\]
is a Yu datum.
\vspace{1em}
\\
\textbf{Transfer of the Idempotent and Hecke Algebra Element:}
Let \(({}^{\circ}K^{d},{}^{\circ}\tau)\) (resp. \(({}^{\circ}K^{d\prime},{}^{\circ}\tau^\prime)\)) be the \([G,\pi]_G\)-type (resp. \([G^\prime,\pi^\prime]_G\)-type)  constructed out of \((K^{d},\tau)\) (resp. \((K'^{d},\tau')\)) by using Yu's construction out of \(\Sigma\) (resp. \(\Sigma'\)). 
Let \(m = \lceil r \rceil\), where \(r = \mathrm{depth}(\pi)\). Then \(G(F)_{x,m} \subset K^d\), and the idempotent \(e_{({}^{\circ}K^{d}, {}^{\circ}\tau)} \in \mathscr{H}(G, G(F)_{x,m})\) satisfies:
\[
e_{({}^{\circ}K^{d}, {}^{\circ}\tau)}(g) = 
\begin{cases}
\frac{\dim({}^{\circ}\tau)}{\mathrm{vol}({}^{\circ}K^{d})} \tr_{{}^{\circ}\tau}(g^{-1}) & \text{if } g \in {}^{\circ}K^{d}, \\
0 & \text{otherwise},
\end{cases}
\]
which expands as:
\[
e_{({}^{\circ}K^{d}, {}^{\circ}\tau)} = 
\frac{\dim({}^{\circ}\tau)}{\mathrm{vol}({}^{\circ}K^{d})} 
\sum_{g \in {}^{\circ}K^{d} / G(F)_{x,m}} 
\tr_{{}^{\circ}\tau}(g^{-1}) \cdot \mathbf{1}_{G(F)_{x,m} g G(F)_{x,m}}.
\]

By Theorem \ref{thm:depth-comp-HH}, for \(\ell \geq \Phi_{G,x}(m)\), we have a canonical isomorphism:
\[
\Psi_{x,\ell}: \frac{G(F)_{x}}{G(F)_{x,m}} \stackrel{\sim}{\longrightarrow} \frac{G'(F')_{x'}}{G'(F')_{x',m}}.\tag{A}
\]
which restricts to give:
  \[
  \frac{G^i(F)_{x,r_i}}{G^i(F)_{x,m}} \xrightarrow{\sim} \frac{(G^i)'(F')_{x',r_i}}{(G^i)'(F')_{x',m}}.
  \tag{B}\]
Isomorphisms (A) and (B) imply,
\[
\tilde{\Psi}_{x,\ell}: \frac{K^d}{G^d(F)_{x,m}} \xrightarrow{\sim} \frac{K'^d}{(G^d)'(F')_{x',m}}.
\]
From the inductive construction of Yu out of datum \(\Sigma'\), it readily follows that the representation of \(K'^d\) obtained by composing:
\[
K'^d \to K'^d/G'(F')_{x',m} \xrightarrow{\tilde{\Psi}_{x,\ell}^{-1}} K^d/G(F)_{x,m} \xrightarrow{(\tau,V)} \operatorname{GL}(V),
\]
is the same as \(\tau'\). Set, \(\pi' = \operatorname{c-ind}_{K'^d}^{G'(F')} \tau'\).
By compatibility of transfers:
\[
\tr_{{}^{\circ}\tau}(g^{-1}) = \tr_{{}^{\circ}\tau'}(\tilde{\Psi}_{x,\ell}(g)^{-1}) 
\quad \text{for } g \in {}^{\circ}K^{d} / G(F)_{x,m}. \tag{C}
\]
By Lemma \ref{lem:idempotents-transport}.,
\[
\kappa_\ell(\mathbf{1}_{G(F)_{x,m} g G(F)_{x,m}}) 
= \mathbf{1}_{G'(F')_{x',m} \psi(g) G'(F')_{x',m}}. \tag{D}
\]
Also notice that \(\dim({}^{\circ}\tau) = \dim({}^{\circ}\tau')\) and \(\mathrm{vol}({}^{\circ}K^{d}) = \mathrm{vol}({}^{\circ}K'^{d})\). Combining these with equations (C) and (D), we obtain,
\begin{align*}
\kappa_\ell(e_{({}^{\circ}K^{d}, {}^{\circ}\tau)}) 
&= \frac{\dim({}^{\circ}\tau)}{\mathrm{vol}({}^{\circ}K^{d})} 
   \sum_{g \in {}^{\circ}K^{d} / G(F)_{x,m}} 
   \tr_{{}^{\circ}\tau}(g^{-1}) \cdot 
   \kappa_\ell(\mathbf{1}_{G(F)_{x,m} g G(F)_{x,m}}) \\
&= \frac{\dim({}^{\circ}\tau')}{\mathrm{vol}({}^{\circ}K'^{d})} 
   \sum_{g' \in {}^{\circ}K'^{d} / G'(F')_{x',m}} 
   \tr_{{}^{\circ}\tau'}(g'^{-1}) \cdot 
   \mathbf{1}_{G'(F')_{x',m} g' G'(F')_{x',m}} \\
&= e_{({}^{\circ}K'^{d}, {}^{\circ}\tau')}.
\end{align*}
It follows that \(\kappa_\ell\) induces an isomorphism of Hecke algebras \(\mathcal{H}(G,{}^\circ\tau)\xrightarrow{\sim}\mathcal{H}(G',{}^\circ\tau') \) associated to types \([G,\pi]_G\) and \([G',\pi']_G'\) respectively. 
Thus,
\[
\mathrm{Kaz}_\ell(\pi) = \pi'.
\]

\textbf{Uniformity in \(\ell\):}

Set
\[
\ell_\dagger = \max\left\{ \Phi_{G,x}(r), \ell_{\mathrm{depth-zero}}, \ell_{\mathrm{Hecke}}(m) \right\},
\]
where:
\begin{itemize}
  \item \(\ell_{\mathrm{depth-zero}}\) from Theorem~\ref{thm:ell-close-depth-zero-block},
  \item \(\ell_{\mathrm{Hecke}}(m)\) from Theorem~\ref{thm:deligne-kazhdan-Hecke-isomorphism}.
\end{itemize}

For \(\ell \geq \ell_\dagger\), all required transfers and isomorphisms apply, and the identity
\[
\mathrm{Kaz}_\ell(\pi(\Sigma)) = \pi'(\Sigma')
\]
holds.

\end{proof}

\section*{Acknowledgment}
The author benefited from correspondence with Thomas Haines and Sandeep Varma who pointed out errors in an earlier draft of this article. It's a pleasure to thank them both. The author is also grateful to Jiu-Kang Yu for helpful correspondence and for conveying to me an observation of Cheng-Chiang Tsai concerning a gap in an earlier formulation.

\bibliographystyle{alpha}
\bibliography{depth_tori_refs}

\end{document}